\documentclass[11pt,reqno,oneside]{amsart}
\usepackage{comment}
\usepackage{enumerate}
\usepackage{geometry} 
\usepackage{amsmath}
\usepackage{amsthm}
\usepackage{thmtools}
\usepackage{etoolbox}
\usepackage{hyperref}
\usepackage[capitalize]{cleveref} 
\usepackage{amssymb}
\usepackage{mathrsfs}
\usepackage{lipsum}
\usepackage{graphicx}
\usepackage{subcaption}
\graphicspath{{./figure/}}
\DeclareGraphicsExtensions{.pdf,.jpeg,.png,.jpg}
\usepackage{xcolor}
\usepackage{tikz}
\usepackage{tkz-euclide}
\usepackage{changepage}

\usepackage{titletoc}
\usepackage{float}
\newtheoremstyle{noparen}
{3pt}{3pt}        
{\itshape}        
{}                
{\bfseries}       
{}               
{ }               
{\thmname{#1}\thmnumber{ #2} \thmnote{\normalfont #3}}

\theoremstyle{noparen}

\numberwithin{equation}{section}

\newtheorem{theorem}{Theorem}[section]
\newtheorem{lemma}[theorem]{Lemma}
\newtheorem{proposition}[theorem]{Proposition}
\newtheorem{corollary}[theorem]{Corollary}
\newtheorem{conjecture}{Conjecture}

\newtheoremstyle{nopunct}
{3pt}{3pt}        
{}                
{}                
{\bfseries}       
{}                
{ }               
{\thmname{#1}\thmnumber{ #2} \thmnote{ #3}}
\theoremstyle{nopunct}

\newtheorem{remark}[theorem]{Remark}

\newtheorem{definition}[theorem]{Definition}

\newtheoremstyle{withcitation}
{\topsep}{\topsep}{\itshape}{}{\bfseries}{.}{ }{\thmname{#1}\thmnumber{ #2}\thmnote{ (#3)}}
\theoremstyle{withcitation}

\crefname{lemma}{Lemma}{Lemma}

\newcommand{\condition}[2]{%
\phantomsection\label{cond:#1}
\text{(#2)}
}

\newcommand{\refcond}[2]{%
\hyperref[cond:#1]{(#2)}
}

\makeatletter
\def\l@subsection{\@tocline{2}{0pt}{2.5pc}{5pc}{}}
\makeatother
  
\makeatletter
\let\old@thm\thm
\def\thm{\RifM@\old@thm}
\makeatother

\makeatletter
\let\oldthebibliography\thebibliography
\renewcommand{\thebibliography}[1]{%
  \oldthebibliography{#1}%
  \renewcommand{\@biblabel}[1]{\makebox[\labelwidth][r]{[##1]}}%
}
\makeatother

\title{Rigidity Theorems for the Weyl Problem of Convex Surfaces in Hyperbolic 3-Space}

\author{Xinrong Zhao}
\thanks{Email: \texttt{xrzhao24@m.fudan.edu.cn}}

\address{School of Mathematical Sciences, Fudan University, Shanghai, 200433, P.R. China}
\email{xrzhao24@m.fudan.edu.cn}
\date{}
\begin{document}
	\begin{abstract}
		In this paper, we study the rigidity of noncompact convex sets in hyperbolic 3-space. We prove that any intrinsic isometry between the boundaries of two closed, noncompact convex subsets of hyperbolic 3-space of dimension at least two extends to a global isometry of the ambient space, provided that their ideal boundaries are circle-type closed sets with countably many connected components. Moreover, the same conclusion holds if the ideal boundaries consist of finitely many mutually disjoint disks together with a set of one-dimensional Hausdorff measure zero. This result generalizes a recent rigidity theorem of Luo, Luo, and Rao by allowing the ideal boundaries to contain disk components. As a direct consequence, we establish a uniqueness result concerning the Weyl problem for convex surfaces in hyperbolic 3-space, as proposed by Luo and Wu. In particular, our approach provides an alternative proof of the discrete Schwarz lemma. The proof uses Pogorelov's rigidity theorem for compact convex bodies in $\mathbb{R}^3$, the Pogorelov map, and properties of locally convex surfaces in $\mathbb{R}^3$.
	\end{abstract}
	\maketitle

    \tableofcontents
    \vspace{1em} 

\section{Introduction}

    \subsection{Main results}
	
	Let $\mathbb{H}^3_P$ denote the Poincaré ball model of hyperbolic 3-space. A \emph{circle domain} is an open, connected subset of the Riemann sphere $\hat{\mathbb{C}} = \partial \mathbb{H}^3_P$ whose boundary components are either circles or points. Let $C_P(X)$ denote the hyperbolic convex hull of a closed set $X$ in $\mathbb{H}^3_P \cup \partial \mathbb{H}^3_P$. A compact subset $X \subset \partial \mathbb{H}^3_P$ is called a \emph{circle-type closed set} if its complement is a circle domain. Equivalently, every connected component of such a set $X$ is either a round disk or a point. The classical Weyl problem asks whether positively curved 2-spheres can be isometrically embedded into Euclidean 3-space. The Weyl problem in hyperbolic 3-space for genus-zero surfaces, as posed by Luo and Wu in \cite{LuoWu}, is the following:
	\begin{conjecture}[(\cite{LuoWu})]\label{conj: 1}
		Suppose that $(S, d)$ is a connected planar surface equipped with a complete path metric of curvature at least $-1$. Then $(S,d)$ is isometric to a complete convex surface $Y$ that bounds a convex set $X$ in $\mathbb{H}^3_P$ such that $\overline{X} \cap \partial \mathbb{H}^3_P$ is a circle-type closed set, where $\overline{X}$ denotes the closure of $X$ in $\mathbb{R}^3$. Furthermore, the convex surface $Y$ is unique up to isometries of $\mathbb{H}^3_P$.
	\end{conjecture}
    In this paper, following the framework developed in \cite{luo2026rigidity}, we establish the following rigidity results for convex surfaces in hyperbolic 3-space via the Pogorelov map. When $S$ possesses only countably many topological ends, the uniqueness part of \cref{conj: 1} follows as a direct consequence of \cref{thm: MainThm1}.
	
	\begin{theorem}\label{thm: MainThm1}
		Let $X_1$ and $X_2$ be two closed, noncompact convex subsets in $\mathbb{H}^3_P$ of dimension at least $2$. Suppose that their boundaries $\partial X_1$ and $\partial X_2$ in $\mathbb{H}^3_P$ are isometric with respect to their intrinsic path metrics. For each $i \in \{1, 2\}$, let $\overline{X}_i$ denote the closure of $X_i$ in $\mathbb{R}^3$, and define $X_i^P = \overline{X}_i \cap \partial \mathbb{H}^3_P$. If each $X_i^P$ is a circle-type closed set $\mathscr{D}_i$ with countably many connected components, then any intrinsic isometry between $\partial X_1$ and $\partial X_2$ extends to a global isometry of $\mathbb{H}^3_P$.
	\end{theorem}
	
	 \begin{theorem}\label{thm: MainThm2}
		Let $X_1$ and $X_2$ be two closed, noncompact convex subsets in $\mathbb{H}^3_P$ of dimension at least $2$. Suppose that for each $i \in \{1, 2\}$, $X_i^P$ is the disjoint union of finitely many closed disks together with a set $\mathscr{P}_i$ having zero one-dimensional Hausdorff measure. Then any intrinsic isometry between $\partial X_1$ and $\partial X_2$ extends to a global isometry of $\mathbb{H}^3_P$.
	\end{theorem}
	
	\begin{remark}\label{rmk: 1}
        If $X_i$ is two-dimensional in the above theorems, its ``boundary'' $\partial X_i$ is defined as the metric double of $X_i$ in the sense of Alexandrov \cite{Intrinsic-geometry-of-convex-surfaces}, i.e., $\partial X_i$ is the gluing of two copies $X_i^+$ and $X_i^{-}$ of $X_i$ by the identity map along their 1-dimensional boundaries. Note that if each $X_i^P$ is restricted to be a set of one-dimensional Hausdorff measure zero in $\mathbb{R}^3$, then \cref{thm: MainThm2} reduces to the main theorem established in \cite{luo2026rigidity}. 
	\end{remark}

    \subsection{Background and applications} 

    The study of global rigidity for convex surfaces traces back to Cauchy, who proved that an intrinsic isometry $f$ between the boundaries of two compact convex polyhedra extends to an isometry of $\mathbb{R}^3$, provided that $f$ preserves their combinatorial structure. Alexandrov \cite{MR13540,ConvexPolyhedron} later removed the combinatorial restriction of $f$, thereby establishing the rigidity of convex polyhedra from a purely metric perspective. The rigidity problem for general compact convex bodies was resolved by Pogorelov \cite{Pogorelov}.
	\begin{theorem}[(Pogorelov)]\label{thm: Pogorelov}
		If $P$ and $Q$ are two compact convex bodies in the Euclidean 3-space $\mathbb{R}^3$ whose boundaries are isometric with respect to their intrinsic path metrics, then any isometry between the boundaries of $P$ and $Q$ extends to an isometry of $\mathbb{R}^3$.
	\end{theorem}
    By introducing what is now known as the Pogorelov map, Pogorelov extended this rigidity to compact convex bodies in spherical 3-space $\mathbb{S}^3$ and hyperbolic 3-space $\mathbb{H}^3$. Since then, extensive efforts have been devoted to generalizing this theorem to noncompact settings in $\mathbb{H}^3$ (see, e.g., \cite{MR1280952,schlenker2002,schlenker2003,MR2208419,MR2399657,MR2469522}). More recently, based on the Pogorelov map, Pogorelov's rigidity theorem, and the Tabor–Tabor theorem on the extension of locally convex functions, Luo, Luo, and Rao \cite{luo2026rigidity} proved the following rigidity theorem:
    \begin{theorem}[(\cite{luo2026rigidity})]\label{thm: LuoLuoRao}
        Let $X_1$ and $X_2$ be two closed, noncompact convex subsets in $\mathbb{H}^3_P$ of dimension at least $2$. Suppose that for each $i \in \{1, 2\}$, the one-dimensional Hausdorff measure of $X_i^P$ is zero. Then any intrinsic isometry between $\partial X_1$ and $\partial X_2$ extends to a global isometry of $\mathbb{H}^3_P$.
    \end{theorem}
    As observed in \cite{luo2026rigidity}, the following counterexample by Thurston \cite{Thurston} shows that the condition of vanishing one-dimensional Hausdorff measure is optimal. Let $X_i$ be the hyperbolic convex hull of a non-circular Jordan curve $J_i \subset \partial \mathbb{H}^3_P$ ($i=1,2$). Then each $\partial X_i$ is intrinsically isometric to two copies of $\mathbb{H}^2$, which implies that $\partial X_1$ and $\partial X_2$ are isometric. Note that $X_1$ and $X_2$ are globally isometric if and only if $J_1$ and $J_2$ differ by a Möbius transformation. Thus, taking $J_1$ to be a square and $J_2$ to be a non-square rectangle yields the desired counterexample. This demonstrates that the naive extension of Pogorelov's theorem to arbitrary closed convex sets in $\mathbb{H}^3_P$ fails.
    
    In \cite{LuoWu}, Luo and Wu considered the following special form of the Weyl problem:
	\begin{conjecture}\label{conj: 2}
		Every genus-zero complete hyperbolic surface $S$ is isometric to the boundary $\partial C_P(X)$ of the convex hull of some circle-type closed set $X \subset \partial \mathbb{H}^3_P$. Furthermore, such a circle-type closed set $X$ is unique up to Möbius transformations.
	\end{conjecture}
	By relating \cref{conj: 2} to the Koebe circle domain conjecture, Luo and Wu \cite{LuoWu} proved the existence of such a circle-type closed set $X$ for any genus-zero complete hyperbolic surface $S$ with countably many topological ends. When $S$ has countably many ends and every end is a cusp, \cref{thm: LuoLuoRao} established the uniqueness of the circle-type closed set $X$ in \cref{conj: 2}. These previous developments, in conjunction with \cref{thm: MainThm1}, fully resolve the restricted version of \cref{conj: 2} for the class of genus-zero complete hyperbolic surfaces with countably many topological ends:
    \begin{theorem}\label{thm: Countable-ends}
        Every genus-zero complete hyperbolic surface $S$ with countably many topological ends is isometric to the boundary $\partial C_P(X)$ of the convex hull of some circle-type closed set $X \subset \partial \mathbb{H}^3_P$. Furthermore, such a circle-type closed set $X$ is unique up to Möbius transformations.
    \end{theorem}
    
    One of the primary motivations for \cref{conj: 2} stems from the discrete conformal geometry of polyhedral surfaces; we refer to \cite{MR2100762,MR3375525,MR3807319,MR3825607,luo2026DUT} for a comprehensive account of this theory. In particular, the rigidity of noncompact convex sets in hyperbolic 3-space is closely related to the uniqueness aspect of the uniformization of noncompact polyhedral surfaces. As proposed in \cite{MR4466644}, a version of the discrete uniformization problem for noncompact polyhedral surfaces in the setting of hyperbolic geometry is formulated as follows:
    \begin{conjecture}\label{conj: discrete-uniformization-problem}
        Suppose that $S$ is a genus-zero complete hyperbolic surface with countably many ends, all but at most one of which are cusp ends. Then $S$ is isometric to the boundary $\partial C_P(X)$ of the convex hull of some circle-type closed set $X \subset \partial \mathbb{H}^3_P$. Furthermore, such a circle-type closed set $X$ is unique up to Möbius transformations.
    \end{conjecture}
    In the classical setting, the uniqueness part of the uniformization problem follows from the Schwarz lemma and Liouville's theorem. Recently, in a forthcoming paper, Luo and Luo \cite{luo2026SchwarzLemma} established a discrete version of the Schwarz lemma, thereby completing the proof of \cref{conj: discrete-uniformization-problem}. It is worth noting that \cref{conj: discrete-uniformization-problem} follows as an immediate consequence of \cref{thm: Countable-ends}. In particular, \cref{thm: MainThm2} implies the discrete Schwarz lemma, which can be reformulated in terms of the rigidity of convex hull boundaries as follows:

    \begin{corollary}[(Discrete Schwarz lemma)]
        Let $V_1$ and $V_2$ be two discrete subsets of the upper hemisphere $\mathbb{S}^2_+$ whose set of accumulation points coincides with the equator $\partial \mathbb{S}^2_+$. If $\partial C_P(V_1 \cup \mathbb{S}^2_-)$ is isometric to $\partial C_P(V_2 \cup\mathbb{S}^2_-)$ with respect to their intrinsic metrics, then $V_1$ and $V_2$ differ by a Möbius transformation.
    \end{corollary}

   \subsection{Idea of the proof and organization of the paper}
	
	We now briefly outline the main idea of the proof of \cref{thm: MainThm1} and \cref{thm: MainThm2}. First, we apply the Pogorelov map to obtain a pair of isometric, locally convex surfaces $S_i$ ($i=1,2$) in $\mathbb{R}^3$ that are radial with respect to $0$. When $X_i^P$ has only finitely many disk components, we adapt an argument similar to that in \cite{luo2026rigidity} to extend each surface $S_i$ to a locally convex surface $\tilde{S}_i$ that is homeomorphic to a circle domain $U_i$ whose boundary consists of finitely many circles. The proof of \cref{thm: MainThm2} is then completed by the finite-disk filling argument below. Next, we show that on the image surface under the Pogorelov map, each isolated component of $X_i^P$ can be filled with either a point or a flat elliptic disk in $\mathbb{R}^3$; see \cref{prop: fill-isolated-points} and \cref{prop: fill-isolated-disks}. This is achieved by perturbing the Pogorelov map via isometries of $\mathbb{H}^3_P$. By gluing these points or flat elliptic disks to the corresponding boundary components of ${S}_i$, we obtain two new isometric, locally convex surfaces in which the original isolated holes are filled. 
	We note that each gluing operation eliminates isolated holes and may convert some original non-isolated holes into isolated ones. However, a finite number of filling operations is insufficient to eliminate all holes; see \cref{rmk: cantor-rank-infinite}. To overcome this, we construct a transfinite sequence of surfaces by performing successive gluing operations at successor ordinals and taking appropriate limits at limit ordinals. By invoking a theorem from descriptive set theory, we finally extend the surfaces ${S}_i$ to a pair of isometric, closed convex surfaces. Once this construction is complete, Pogorelov's rigidity theorem (\cref{thm: Pogorelov}) and \cref{prop: RnIsometry-to-HnIsometry} imply \cref{thm: MainThm1}.
	
	The paper is organized as follows. In \cref{sec: Preliminaries}, we discuss some properties of convex surfaces and locally convex surfaces, and introduce the proof framework based on the Pogorelov map. In \cref{sec: single-disk}, we address the case where there is exactly one disk component in each $X_i^P$. Finally, in \cref{sec: the-general-case}, we treat the general case and complete the proofs of \cref{thm: MainThm1} and \cref{thm: MainThm2}.

\section{Preliminaries}\label{sec: Preliminaries}
    \subsection{Convex surfaces and locally convex surfaces}\label{sec: convex surfaces}
    In this subsection, we introduce several results for convex surfaces and locally convex surfaces in $\mathbb{R}^3$. Recall that a \emph{convex body} in $n$-dimensional Euclidean space $\mathbb{R}^{n}$ is a convex set with nonempty interior, and a surface $S\subset\mathbb{R}^3$ is called a \emph{convex surface} if it can be realized as a connected and open subset of the topological boundary of a convex body in $\mathbb{R}^3$. In this paper, we employ the standard identification of $\mathbb{R}^2$ with the hyperplane $\{x \in \mathbb{R}^3 : x_3 = 0\}$. First, we introduce a theorem which is a direct consequence of \cite[Theorem 1]{MR657117}.
    \begin{theorem}\label{thm: interior-Lipschitz}
        Let $K$ be a compact convex body in $\mathbb{R}^n$ containing the origin $0$ in its interior. Let $R:\partial K\to \mathbb{S}^{n-1}$ be the radial projection defined by $R(x)=\frac{x}{\|x\|}$. Then the inverse map $Q: \mathbb{S}^{n-1}\to\partial K$ of $R$ is Lipschitz continuous, i.e., there exists a constant $L>0$ such that
        $$
        \|Q(x)-Q(y)\|\leq L\|x-y\|,\quad\forall x,y\in\mathbb{S}^{n-1}.
        $$
    \end{theorem}
    The following lemma is immediate.
	\begin{lemma}\label{lem: Bi-Lipschitz}
		Let $\Sigma$ be the boundary of a compact convex body $K$ in $\mathbb{R}^n$. Then there exists a constant $C_{\Sigma}>0$ such that for all $x,y\in\Sigma$, $d_{\Sigma}(x,y)\leq C_{\Sigma}\|x-y\|$.
	\end{lemma}
	\begin{proof}
	    Without loss of generality, we may assume that $0$ is an interior point of $K$ and the ball $B_r(0)$ centered at $0$ with radius $r$ is contained in $K$. By \cref{thm: interior-Lipschitz}, we have
        \begin{equation}\label{eq: Bi-Lipschitz1}
            d_{\Sigma}(x,y)\leq L\cdot d_{\mathbb{S}^{n-1}}(R(x),R(y))\leq \frac{\pi L}{2} \|R(x) - R(y)\|,\quad\forall x,y\in\Sigma.
        \end{equation}
        Since
        $$
        R(x) - R(y) = \frac{x}{\|x\|} - \frac{y}{\|y\|} = \frac{x - y}{\|x\|} + y \left( \frac{1}{\|x\|} - \frac{1}{\|y\|} \right),
        $$
        we have
        \begin{equation}\label{eq: Bi-Lipschitz2}
            \|R(x) - R(y)\| \leq \frac{\|x - y\|}{\|x\|} + \|y\| \frac{|\|y\| - \|x\||}{\|x\|\|y\|} \leq \frac{2}{\|x\|} \|x - y\|\leq \frac{2}{r} \|x - y\|.
        \end{equation}
        Combining \eqref{eq: Bi-Lipschitz1} and \eqref{eq: Bi-Lipschitz2}, we arrive at the desired inequality
        $$
        d_{\Sigma}(x,y) \leq \frac{\pi L}{r} \|x - y\|, \quad \forall x,y \in \Sigma.
        $$
	\end{proof}

	\begin{lemma}\label{lem: path-metric-boundary-removable}
		Let $S$ be a surface in the upper half-space $\mathbb{R}^3_{+}=\{x\in\mathbb{R}^3:x_3>0\}$. Suppose that there exists a compact convex body $P$ in $\mathbb{R}^2$ such that $S\cup P$ is the boundary of some compact convex body $K$ in $\mathbb{R}^3$. Then $d_S(x,y)=d_{\overline{S}}(x,y)$ for all $x,y\in S$, where $d_S$ and $d_{\overline{S}}$ are the intrinsic path metrics on $S$ and $\overline{S}=S\cup \partial P$ induced by $\mathbb{R}^3$, respectively.
	\end{lemma}
	\begin{proof}
		Let $x,y\in S$, and let $\sigma:[0,1]\to \overline{S}$ be a rectifiable curve with $\sigma(0)=x$ and $\sigma(1)=y$. It suffices to prove that for each $\varepsilon>0$, there exists a curve $\tilde{\sigma}$ in $S$ connecting $x$ and $y$ with length $l(\tilde{\sigma})\leq l(\sigma)+\varepsilon$.
		Without loss of generality, we may assume that the origin $0$ lies in the relative interior of $P$. For any $t\geq0$ and $a>1$, we define the homothetic transformation with center $A_t = (0,0,-t)$ and scaling factor $a$ by
		$$
		h_{a,t}(x)=A_t+a(x-A_t).
		$$
		Fix $a>1$. Since $h_{a,0}(\sigma)\subset K^c$, by the continuity of $h_{a,t}$ with respect to $t$, there exists $t_{a}\in(0,1)$ such that the curve $\sigma_{a}:=h_{a,t_{a}}(\sigma)\subset K^c$. Clearly, the length $l(\sigma_{a})$ of $\sigma_{a}$ is $a \cdot l(\sigma)$. 
		
		For any compact convex body $C$ in $\mathbb{R}^3$ and any point $x\in \mathbb{R}^3$, let $P_C(x)$ be the orthogonal projection of $x$ onto $C$, i.e., $P_C(x)$ is the unique point in $C$ such that 
		$$
		\|P_C(x)-x\|=\min\{\|y-x\|,y\in C\}.
		$$
		Now for each $\lambda>0$, let $K_{\lambda}=K\cap\{x_3\geq \lambda\}$ and $S_{\lambda}=S\cap \partial K_{\lambda}$. After shrinking $\lambda$ if necessary, we may assume that $S_{\lambda}$ is not empty and there exists an open subset $U_{\lambda}$ of the plane $\{x_3=\lambda\}$ such that $S_{\lambda}\cup U_{\lambda}=\partial K_{\lambda}$. Since $\sigma_{a}$ is a compact subset of $\mathbb{R}^3_{+}$, there exists $\lambda_a>0$ such that $\sigma_{a}\subset\{x_3>\lambda_a\}$ and $\sigma(0),\sigma(1)\in\{x_3>\lambda_a\}$.
		It follows that $\tilde{\sigma}_{a}:=P_{K_{\lambda_{a}}}(\sigma_{a})$ is a curve in $\partial K_{\lambda_a}=S_{\lambda_a}\cup U_{\lambda_a}$ since $\sigma_{a}\subset K^c\subset K_{\lambda_a}^c$. We claim that $\tilde{\sigma}_{a}\subset S$. Suppose, for contradiction, that there exists $s\in[0,1]$ such that $\tilde{\sigma}_{a}(s)\in U_{\lambda_a}$. Since $U_{\lambda_a}$ is an open subset of the plane $\{x_3=\lambda_a\}$, it follows that there exists $\delta>0$ such that $\tilde{\sigma}_{a}(s)+\delta e_3\in K_{\lambda_a}$, where $e_3=(0,0,1)$. By \cite[Theorem 3.1.1]{MR1865628}, the inner product 
		$$
		\big({\sigma}_{a}(s)-\tilde{\sigma}_{a}(s),y-\tilde{\sigma}_{a}(s)\big)\leq 0
		$$ 
		for all $y\in K_{\lambda_a}$. Take $y=\tilde{\sigma}_{a}(s)+\delta e_3$. Then we have $\big({\sigma}_{a}(s)-\tilde{\sigma}_{a}(s),e_3\big)\leq 0$, which contradicts the fact that ${\sigma}_{a}\subset\{x_3>\lambda_a\}$.
		
		By \cite[Proposition 3.1.3]{MR1865628}, for each compact convex body $C$, the projection $P_C$ is a Lipschitz map with Lipschitz constant $1$. Thus, $l(\tilde{\sigma}_{a})\leq l(\sigma_{a})=a\cdot l(\sigma)$ for all $a>1$. By definition of the transformation $h_{a,t_{a}}$, we have
        $$
        \|{\sigma}_{a}(s)-\sigma(s)\|=(a-1)\|\sigma(s)-A_{t_{a}}\|\quad\forall s\in[0,1].
        $$ 
        For $i\in\{0,1\}$, the choice of $\lambda_a$ ensures that $\sigma(i)\in K_{\lambda_a}$, so
		$$
		\lim_{a\to 1^{+}}\|\tilde{\sigma}_{a}(i)-{\sigma}_{a}(i)\|\leq\lim_{a\to 1^{+}}
		\|\sigma(i)-{\sigma}_{a}(i)\|=\lim_{a\to 1^{+}}(a-1)\|\sigma(i)-A_{t_{a}}\|=0.
		$$
		Therefore, when $a$ is sufficiently close to $1$, for each $i\in\{0,1\}$ there exists a curve $\hat{\sigma}_i$ in $S$ connecting $\sigma(i)$ and $\tilde{\sigma}_{a}(i)$ with length $l(\hat{\sigma}_i)<\varepsilon/3$. Let $\tilde{\sigma}=\hat{\sigma}_0\cup\tilde{\sigma}_a\cup\hat{\sigma}_1$. Clearly, $\tilde{\sigma}$ is a curve in $S$ connecting $x=\sigma(0)$ and $y=\sigma(1)$ with length $l(\tilde{\sigma})\leq a \cdot l(\sigma)+2\varepsilon/3$. The proof is then concluded by choosing $a < 1 + \frac{\varepsilon}{3l(\sigma)}$.
	\end{proof}
	
	\begin{proposition}\label{prop: convex-surface-Lipschitz}
		Let $S$ be a surface in $\mathbb{R}^3_{+}$. Suppose that there exists a compact convex body $P$ in $\mathbb{R}^2$ such that $\Sigma=S\cup P$ is the boundary of some compact convex body in $\mathbb{R}^3$. Then, there exists a constant $C>0$ such that $d_{S}(x,y)\leq C\|x-y\|$ for all $x,y\in S$, where $d_{S}$ is the intrinsic path metric on ${S}$ induced by $\mathbb{R}^3$.
	\end{proposition}
	\begin{proof}
		By \cref{lem: Bi-Lipschitz}, there exists a constant $C_{\Sigma}>0$ such that $d_{\Sigma}(x,y)\leq C_{\Sigma}\|x-y\|$ for all $x,y\in \Sigma$. Since  $d_S(x,y)=d_{\overline{S}}(x,y)$ for all $x,y\in S$ by \cref{lem: path-metric-boundary-removable}, we only need to prove that there exists a constant $C>0$ such that $d_{\overline{S}}(x,y)\leq Cd_{\Sigma}(x,y)$ for all $x,y\in S$. Observe that $P$ is a compact convex body in $\mathbb{R}^2$. Again by \cref{lem: Bi-Lipschitz}, there exists a constant $C_{\partial P}>0$ such that $d_{\partial P}(x,y)\leq C_{\partial P}\|x-y\|$ for all $x,y\in \partial P$. Now let $x,y\in S$, and let $\sigma$ be a path connecting $x$ and $y$ in $\Sigma$. We claim that there always exists a path $\sigma'$ connecting $x$ and $y$ in $\overline{S}$ such that $l(\sigma')\leq C_{\partial P}\cdot l(\sigma)$, where $l(\sigma)$ and $l(\sigma')$ are the lengths of the curves $\sigma$ and $\sigma'$, respectively. If $\sigma$ lies entirely in $\overline{S}$, then there is nothing to prove. If $\sigma$ crosses $\partial P$, let $x_1$ and $y_1$ be the first and last points in $\sigma\cap\partial P$ when traveling along $\sigma$ from $x$ to $y$. Let $\sigma'$ be the curve obtained from $\sigma$ by replacing the curve segment $\sigma_1$ of $\sigma$ from $x_1$ to $y_1$ by the shortest curve $\sigma_1'$ connecting $x_1$ and $y_1$ in $\partial P$. Clearly, 
		$$
		l(\sigma_1')=d_{\partial P}(x_1,y_1)\leq  C_{\partial P}\|x_1-y_1\|\leq  C_{\partial P}\cdot l(\sigma_1).
		$$
		It follows that $l(\sigma')\leq C_{\partial P}\cdot l(\sigma)$, which completes the proof.
	\end{proof}
    
    Now we consider locally convex surfaces in $\mathbb{R}^3$. Recall that a surface $S \subset \mathbb{R}^3$ (possibly with boundary) is called \emph{locally convex} at a point $p \in \operatorname{int} S$ if there exists a neighborhood $U$ of $p$ in $\mathbb{R}^3$ such that the intersection $S \cap U$ lies on the boundary of a convex set $K \subset \mathbb{R}^3$. If, in addition, there exists a supporting plane of $K$ that meets $K$ only at the point $p$, then $S$ is said to be \emph{absolutely convex} at $p$. The surface $S$ is said to be a \emph{locally convex surface} in $\mathbb{R}^3$ if $S$ is locally convex at each point $p\in \operatorname{int} S$. For locally convex surfaces, we have the following theorem by John van Heijenoort \cite{LocallyConvex}.
	
	\begin{theorem}
		Suppose that $S$ is a complete locally convex surface without boundary in $\mathbb{R}^3$ such that $S$ is absolutely convex at some point $p\in S$. Then $S$ is the boundary of a convex body in $\mathbb{R}^3$.
	\end{theorem}
	Note that a compact surface $S$ without boundary in $\mathbb{R}^3$ is always absolutely convex at some point. Indeed, if we take a closed ball $B$ in $\mathbb{R}^3$ containing $S$ such that $\partial B\cap S\neq\varnothing$, then $S$ is absolutely convex at each point $p\in \partial B\cap S$. Thus, we obtain the following corollary.
	\begin{corollary}\label{cor: locallyconvex-convex}
		Suppose that $S$ is a compact locally convex surface without boundary in $\mathbb{R}^3$. Then $S$ is the boundary of a compact convex body in $\mathbb{R}^3$.
	\end{corollary}
    A local supporting plane of an open set $U$ in $\mathbb{R}^3$ at a boundary point $p\in \partial U$ is a plane $H$ through $p$ such that there exists a neighborhood of $p$ in $H$ having empty intersection with $U$. It is worth noting that a supporting plane $H$ of $\partial U$ is not necessarily a local supporting plane of $U$ since $H$ might have a nonempty intersection with $U$. The following theorem was first established by Erhard Schmidt; we refer the reader to \cite{LocallyConvex} for a proof.
	\begin{theorem}[(The Erhard Schmidt Theorem)]\label{thm: The-Erhard-Schmidt-Theorem}
		A connected open set in $\mathbb{R}^3$ having a local supporting plane at each boundary point is convex.
	\end{theorem}

	\subsection{The Pogorelov map and the proof framework}\label{sec: pogorelov}
     
    We now define the concept of the Pogorelov map and introduce some fundamental properties of this map that will be used in this paper. For references, see \cite{PogorelovBook,MR1744513,MR2208419,virto2010,luo2026rigidity}. 

	Let \((\mathbb{R}^{1,n}, \langle\cdot,\cdot\rangle)\) denote the $(n+1)$-dimensional Minkowski space, where the Minkowski inner product for vectors $x = (x_0, x_1, \dots, x_n)$ and $y = (y_0, y_1, \dots, y_n)$ is defined as
	$$
	\langle x, y \rangle = -x_0y_0 + \sum_{i=1}^n x_i y_i = -x_0y_0 + (P(x), P(y)).
	$$
	Here, $P(x_0, x_1, \dots, x_n) = (x_1, \dots, x_n)$ is the projection onto the spatial coordinates, and $(u, v)$ is the standard Euclidean inner product of two vectors $u, v \in \mathbb{R}^n$.
	The hyperboloid model $\mathcal{H}^n$ of $n$-dimensional hyperbolic space is defined by
	$$
	\mathcal{H}^n = \{ x \in \mathbb{R}^{1,n} : \langle x, x \rangle = -1,\; x_0 > 0 \}
	= \left\{ x \in \mathbb{R}^{1,n} : x_0^2 = 1 + \sum_{i=1}^n x_i^2,\; x_0 > 0 \right\}.
	$$
	\begin{definition}
		Let $e = (1, 0, \dots, 0) \in \mathbb{R}^{1,n}$. The Pogorelov map 
		$$
        \Phi : (\mathbb{R}^{1,n} \times \mathbb{R}^{1,n}) \setminus \{(x, y) \in \mathbb{R}^{1,n} \times \mathbb{R}^{1,n} : \langle x + y, e \rangle = 0\}\to \mathbb{R}^n \times \mathbb{R}^n
        $$ 
		is defined by 
		\begin{equation*}
			\Phi(x, y) = \left( \frac{2 P(x)}{-\langle x + y, e \rangle}, \frac{2 P(y)}{-\langle x + y, e \rangle} \right).
		\end{equation*}
		In particular, the Pogorelov map $\Phi$ is defined on $\mathcal{H}^n \times \mathcal{H}^n$. Let $P_i$ be the components of the Pogorelov map $\Phi$, i.e., 
		\begin{equation*}
			P_1(x, y) = \frac{2 P(x)}{-\langle x + y, e \rangle}, \quad P_2(x, y) = \frac{2 P(y)}{-\langle x + y, e \rangle}.
		\end{equation*}
	\end{definition}
	
	\begin{remark}
		Although the Pogorelov map is defined using the hyperboloid model $\mathcal{H}^n$, we will mainly employ the Poincaré model $\mathbb{H}^n_P$ throughout this paper. We shall identify $\mathcal{H}^n$ with $\mathbb{H}^n_P$ via the canonical map $f:\mathcal{H}^n\to\mathbb{H}^n_P$, $x\mapsto\frac{P(x)}{-\langle x,e\rangle+1}$.
	\end{remark}
	The following simple observation is immediate.
	\begin{lemma}\label{lem: SumIs2}
		{\rm{(1)}}. For all $x,y\in\mathcal{H}^n$, $\Vert P_1(x,y)\Vert+\Vert P_2(x,y)\Vert<2$. 
		
		{\rm{(2)}}. Let $\{x_n\}_n$ and $\{y_n\}_n$ be two sequences in $\mathcal{H}^n$ such that 
		$$
		\lim_{n \to\infty}(d_{\mathcal{H}^n}(x_n,e)+d_{\mathcal{H}^n}(y_n,e))=+\infty.
		$$
		Then 
		$$
		\lim_{n\to\infty}(\Vert P_1(x_n,y_n)\Vert+\Vert P_2(x_n,y_n)\Vert)=2.
		$$
	\end{lemma}
    From now on, we will always assume that $\Phi$ is restricted to $\mathcal{H}^n\times\mathcal{H}^n$. Let $\Omega = \{(\hat{x}, \hat{y}) \in \mathbb{R}^n \times\mathbb{R}^n : \|\hat{x}\| + \|\hat{y}\| < 2\}$. By \cref{lem: SumIs2}, $\Phi$ maps $\mathcal{H}^n\times\mathcal{H}^n$ to $\Omega$. In fact, by \cite[Proposition 6.3]{virto2010}, $\Phi$ is a diffeomorphism between $\mathcal{H}^n\times\mathcal{H}^n$ and $\Omega$. The following lemma shows that the Pogorelov map preserves intrinsic isometries. The proof is a natural generalization of the argument in \cite[Lemma 3.5]{luo2026rigidity}.
    \begin{lemma}\label{lem: preserve-intrinsic-isometries}
    Suppose that $\alpha(t)$ and $\beta(t)$ are two absolutely continuous curves in $\mathcal{H}^n$. Let $\gamma(t) = P_1(\alpha(t), \beta(t))$ and $\delta(t) = P_2(\alpha(t), \beta(t))$. Then $\alpha(t)$ and $\beta(t)$ have the same length in $\mathcal{H}^n$ and satisfy $\langle \alpha'(t),\alpha'(t)\rangle=\langle \beta'(t),\beta'(t)\rangle$ for almost all $t$ if and only if $\gamma(t)$ and $\delta(t)$ have the same length in $\mathbb{R}^n$ and satisfy $\|\gamma'(t)\|=\|\delta'(t)\|$ for almost all $t$.
    \end{lemma}
    
    \begin{proof}
    The proof closely follows the framework of \cite[Lemma~3.5]{luo2026rigidity}. For brevity, we only present the necessary modifications and refer the reader to the original paper for the complete argument. We replace the unit-speed assumption 
    $$
    \sum_{i=1}^n (\alpha_i')^2 = 1 + (\alpha_0')^2, \quad \sum_{i=1}^n (\beta_i')^2 = 1 + (\beta_0')^2
    $$
    from \cite[equation (3.4)]{luo2026rigidity} with the general relations
    $$
    \sum_{i=1}^n (\alpha_i')^2 = \langle\alpha',\alpha' \rangle + (\alpha_0')^2, \quad  \sum_{i=1}^n (\beta_i')^2 = \langle\beta',\beta' \rangle + (\beta_0')^2.
    $$
    Substituting these general relations into the numerator of the speed formula (3.2) in \cite{luo2026rigidity} yields
    $$
    \sum_{i=1}^n \left[ \alpha_i'(\alpha_0 + \beta_0) - \alpha_i(\alpha_0' + \beta_0') \right]^2 = (\alpha_0' \beta_0 - \alpha_0 \beta_0')^2 + \langle\alpha',\alpha' \rangle (\alpha_0 + \beta_0)^2 - (\alpha_0' + \beta_0')^2.
    $$
    Consequently, the squared Euclidean norm of $\gamma'(t)$ is given by:
    $$
    \|\gamma'(t)\|^2 = \frac{4}{(\alpha_0 + \beta_0)^4} \left[ (\alpha_0' \beta_0 - \alpha_0 \beta_0')^2 + \langle\alpha',\alpha' \rangle (\alpha_0 + \beta_0)^2 - (\alpha_0' + \beta_0')^2 \right].
    $$
    By symmetry, for $\delta(t)$ we analogously obtain:
    $$
    \|\delta'(t)\|^2 = \frac{4}{(\alpha_0 + \beta_0)^4} \left[ (\alpha_0' \beta_0 - \alpha_0 \beta_0')^2 + \langle\beta',\beta' \rangle (\alpha_0 + \beta_0)^2 - (\alpha_0' + \beta_0')^2 \right].
    $$
    Taking the difference of these two expressions and canceling the symmetric terms yields
    $$
    \|\gamma'(t)\|^2 - \|\delta'(t)\|^2 = \frac{4}{(\alpha_0 + \beta_0)^2} \left( \langle \alpha'(t), \alpha'(t) \rangle - \langle \beta'(t), \beta'(t) \rangle \right),
    $$
    which completes the proof.
    \end{proof}

    In the spherical geometry $\mathbb{S}^n$, Pogorelov \cite{PogorelovBook} showed that two isometric surfaces are congruent if and only if their images under the Pogorelov map are congruent in Euclidean space. This congruency-preserving property also holds in hyperbolic geometry. As a direct consequence of \cite[Proposition 3.8]{luo2026rigidity}, we obtain the following proposition.
	\begin{proposition}\label{prop: RnIsometry-to-HnIsometry}
		Let $X$ be a subset of $\mathcal{H}^n$ and $F$ be a map from $X$ to $\mathcal{H}^n$. Suppose that there exists an isometry $T\in\operatorname{Isom}(\mathbb{R}^n)$ such that $P_2(x,F(x))=T(P_1(x,F(x)))$ for all $x\in X$. Then there exists $A \in\operatorname{Isom}(\mathcal{H}^n)$ such that $F(x)=A(x)$ for all $x\in X$.
	\end{proposition}
    We say that a subset $S$ of $\mathbb{H}^3_P$ or $\mathbb{R}^3$ is \emph{radial with respect to $v$} if each geodesic ray emanating from $v$ intersects $S$ in at most one point. Since we only consider surfaces that are radial with respect to $0$ in this paper, we will refer to them simply as \emph{radial surfaces}. A surface $S \subset \mathbb{R}^3$ is called \emph{locally convex with respect to $0$} if for each point $p \in S$, there exists a convex set $K_p$ containing $0$ such that $S \cap \partial K_p$ contains a neighborhood of $p$ in $S$. An important property of the Pogorelov map is that it preserves the local convexity of surfaces. The following theorem was originally proved by Pogorelov in the setting of the 3-sphere $\mathbb{S}^3$ \cite[pp.~314--322]{PogorelovBook}. The proof in \cite{PogorelovBook} remains valid in the hyperbolic setting.
 	\begin{theorem}[(Pogorelov).]\label{thm: PogorelovII}
		Suppose that $X_1$ and $X_2$ are two convex 3-dimensional sets in $\mathcal{H}^3$ that contain $e=(1,0,0,0)$ in their interiors. Let $\Sigma_i$ be a surface contained in $\partial X_i$ and $F : \Sigma_1 \to \Sigma_2$ be an intrinsic isometry. Then 
		\begin{enumerate}
			\item[\rm(1)] $S_1 = \{P_1(x, F(x)) : x \in \Sigma_1\}$ and $S_2 = \{P_2(x, F(x)) : x \in \Sigma_1\}$ are two radial surfaces in $\mathbb{R}^3$ with respect to $0$.
			
			\item[\rm(2)] For each $i\in\{ 1, 2\}$, $S_i$ is a locally convex surface with respect to $0$.
			
			\item[\rm(3)] There exists an intrinsic isometry $G : S_1 \to S_2$ such that $\Phi(x, F(x)) = (y(x), G(y(x)))$ for all $x \in \Sigma_1$, where $y(x)=P_1(x,F(x))$.
		\end{enumerate}
	\end{theorem}
    
    \begin{figure}[htbp]
        \vspace{-2em}
        \centering
        \begin{subfigure}[b]{0.45\textwidth}
            \centering
            \includegraphics[width=\textwidth]{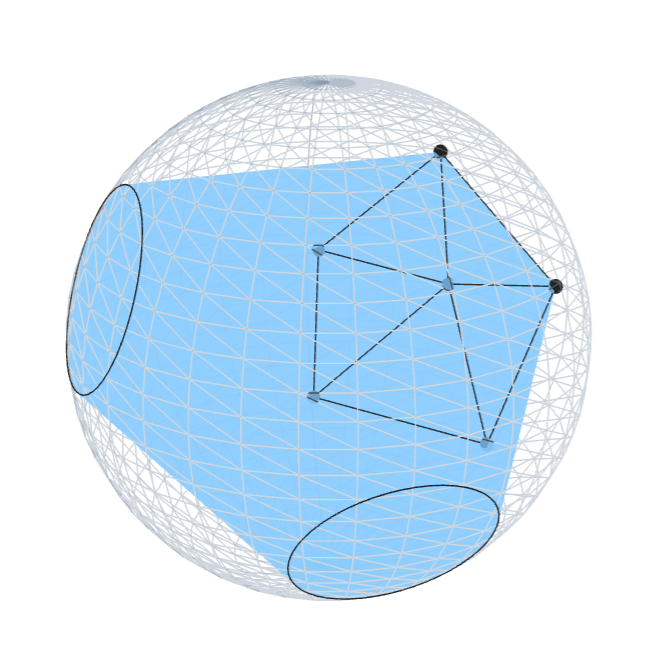}
            \caption{$\partial X_1$ (Klein model)}
            \label{fig:sub1}
        \end{subfigure}
        \hfill 
        \begin{subfigure}[b]{0.45\textwidth}
            \centering
            \includegraphics[width=\textwidth]{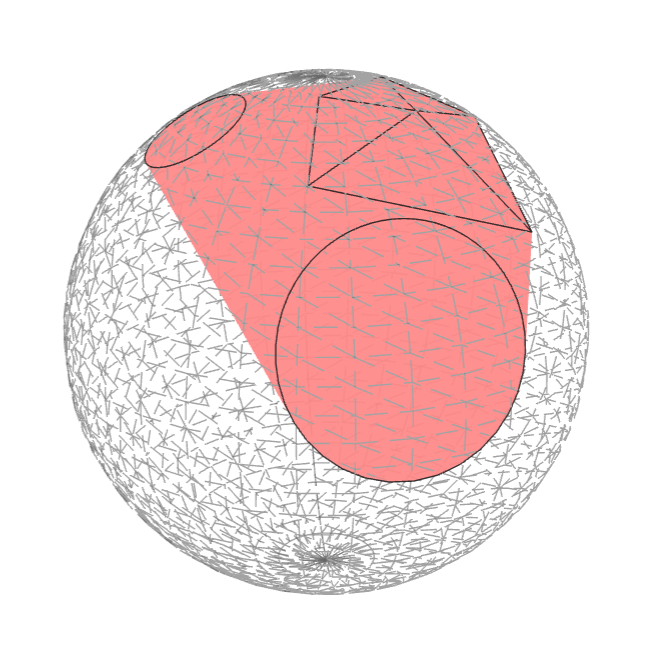} 
            \caption{$T(\partial X_1)$ (Klein model)}
            \label{fig:sub2}
        \end{subfigure}

        \begin{subfigure}[b]{0.45\textwidth}
            \centering
            \includegraphics[width=\textwidth]{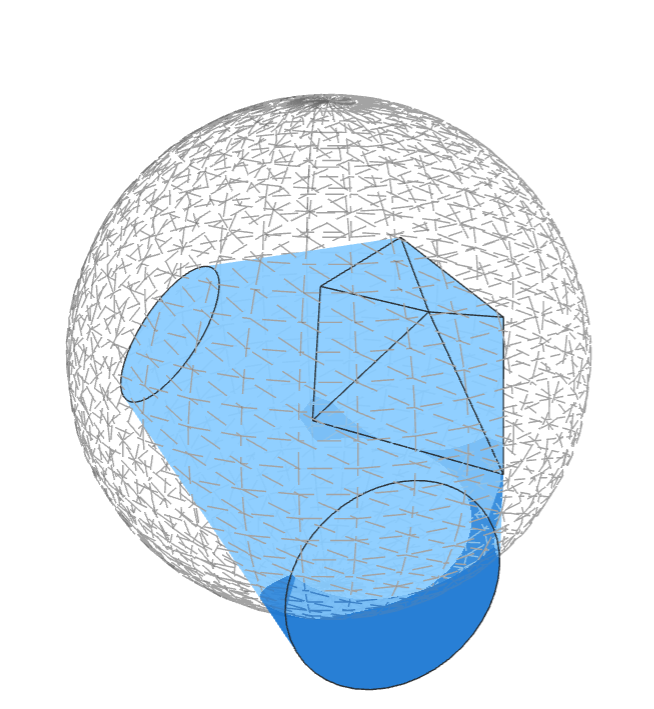} 
            \caption{Mapped surface $S_1$ in $\mathbb{R}^3$}
            \label{fig:sub3}
        \end{subfigure}
        \hfill
        \begin{subfigure}[b]{0.45\textwidth}
            \centering
            \includegraphics[width=\textwidth]{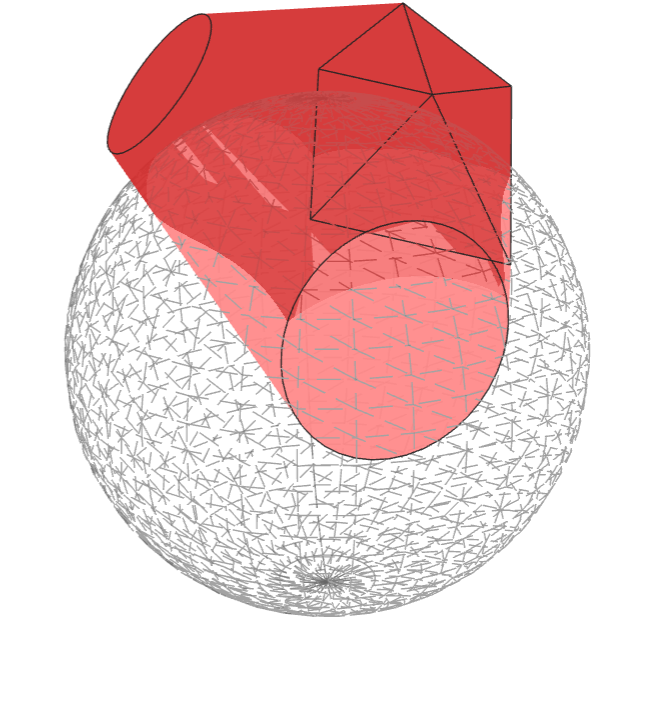}
            \caption{Mapped surface $S_2$ in $\mathbb{R}^3$}
            \label{fig:sub4}
        \end{subfigure}
        \caption{Visualizations of the Pogorelov map $\Phi(x,T(x))$ restricted to $\partial X_1$. Here, $T$ is an isometry of hyperbolic 3-space such that $T(x_1,x_2,x_3) = (ex_1, ex_2, ex_3)$ in the upper half-space model, and $X_1$ is the convex hull of a circle-type closed set consisting of two disks and six isolated points.}
        \label{fig:Pogorelov_map}
    \end{figure}
    Before introducing the proof framework, we need several auxiliary lemmas. Recall that a function $f$ defined on an open set $\Omega$ in $\mathbb{R}^n$ is called \emph{locally convex} if for each $p\in \Omega$, there exists a convex neighborhood $U_p$ of $p$ in $\Omega$ such that $f|_{U_p}$ is a convex function. The following lemma about local convexity is well known.
	\begin{lemma}\label{lem: LocalToGlobalConvex}
		If $\Omega \subset \mathbb{R}^n$ is a convex open set and $f : \Omega \to \mathbb{R}$ is locally convex, then $f$ is a convex function on $\Omega$.
	\end{lemma}
	The following result in \cite{TaborTabor} ensures the extension of a locally convex function defined outside a closed set of vanishing $(n-1)$-dimensional Hausdorff measure.
	\begin{theorem}[{\cite[Corollary 4.2]{TaborTabor}}]\label{thm: ExtendConvex}
		Let $\Omega$ be an open subset of $\mathbb{R}^n$ and $A$ be a closed subset of $\Omega$ such that $\mathscr{H}^{n-1}(A) = 0$. If $f : \Omega \setminus A \to \mathbb{R}$ is locally convex, then $f$ can be extended to a locally convex function on $\Omega$.
	\end{theorem}
    
	To show that a set of zero one-dimensional Hausdorff measure does not change the intrinsic metric of the surface, we need the following area formula, which is well known in geometric measure theory; see, e.g., \cite[Section 8.5]{MR2976521} for a proof.
    \begin{theorem}[(Area formula)]\label{thm: Area-formula}
		If $f \colon \mathbb{R}^n \to \mathbb{R}^m$ \textup{(}$1 \le n \le m$\textup{)} is a Lipschitz function and $E \subset \mathbb{R}^n$ is Lebesgue measurable, then
		\begin{equation}
			\int_{\mathbb{R}^m}  \mathscr{H}^0(E\cap \{f=y\}) \mathrm{d}\mathscr{H}^n(y)= \int_E Jf(x) \,\mathrm{d}x,
		\end{equation}
        where the Jacobian $J f: \mathbb{R}^n \to [0, \infty]$ is defined by
        \begin{equation*}\label{eq: jacobian-definition}
            Jf(x) = 
            \begin{cases}
                \sqrt{\det\left(\nabla f(x)^* \nabla f(x)\right)}, & \text{if } f \text{ is differentiable at } x; \\
                +\infty, & \text{if } f \text{ is not differentiable at } x.
            \end{cases}
        \end{equation*}
	\end{theorem}
    We are now ready to establish the following theorem, which generalizes Theorem 4.4 in \cite{luo2026rigidity} to noncomplete surfaces.
	\begin{theorem}\label{thm: zero-H1-dont-change-metric}
		Suppose that $S \subset \mathbb{R}^3$ is a locally Lipschitz surface in $\mathbb{R}^3$ and $A\subset S$ is a subset such that $\mathscr{H}^1(A) = 0$ and $A\cap K$ is compact for any compact subset $K$ of $S$. Let $d_S$ and $d_{S \setminus A}$ be the intrinsic path metrics on $S$ and $S \setminus A$, respectively. Then for any $x, y \in S \setminus A$, 
		\begin{equation*}
			d_S(x, y) = d_{S \setminus A}(x, y).
		\end{equation*}
	\end{theorem}
	\begin{proof}
		Suppose that $x, y \in S \setminus A$ and $\sigma:[0,l(\sigma)]\to S$ is a rectifiable curve parametrized by arc length with $\sigma(0)=x$ and $\sigma(l(\sigma))=y$. It suffices to prove that for each $\varepsilon>0$, there exists a curve $\tilde{\sigma}$ in $S\setminus A$ connecting $x$ and $y$ with length $l(\tilde{\sigma})\leq l(\sigma)+\varepsilon$. Let $E=\sigma^{-1}(A)$. It is clear that $E$ is compact since $A\cap\sigma$ is compact. By the area formula, 
		\begin{equation}
			m(E)=\int_E \|\sigma'(x) \|\,\mathrm{d}x=\int_A  \mathscr{H}^0(E\cap \{\sigma=y\}) \mathrm{d}\mathscr{H}^1(y) =0,
		\end{equation}
		where $m(E)$ is the Lebesgue measure of $E$ in $[0,l(\sigma)]$. Since $E$ is compact, for each $\delta>0$, there exists a collection of disjoint open intervals $\{I_i=(s_i,t_i),1\leq i\leq n\}$ covering $E$ such that $\sum_i|t_i-s_i|<\delta$. Since $S$ is locally Lipschitz, there exists a finite open cover $\{U_j\}_{j=1}^{m}$ of $\sigma$ in $S$ and a constant $L>0$ such that for each $U_j$, there exists a bi-Lipschitz map $f_j$ from $U_j$ to the unit open disk $\mathbb{D}$ with bi-Lipschitz constant $L$. Let $\sigma_i=\sigma|_{\overline{I_i}}$. Choosing $\delta$ small enough, we may assume that each $\sigma_i$ is contained in some $U_j$.
		
		We claim that there exists a curve $\tilde{\sigma}_i$ in $U_j\setminus A$ connecting $\sigma(s_i)$ and $\sigma(t_i)$ with length $l(\tilde{\sigma}_i)\leq 3L^2 \cdot l(\sigma_i)$. This clearly holds when $\sigma(s_i)=\sigma(t_i)$. Now assume that $\sigma(s_i)\neq\sigma(t_i)$.
		Let $\gamma_i$ be the line segment in $\mathbb{D}$ connecting $\xi_{i}=f_j(\sigma(s_i))$ and $\eta_{i}= f_j(\sigma(t_i))$. Let $\hat{A}_i=f_j(U_j\cap A)$. Then $\hat{A}_i\cap K$ is compact for any compact subset $K$ in $\mathbb{D}$.
        Since $\xi_i,\eta_i\notin \hat{A}_i$ and $\mathbb{D}$ is locally compact, there exists $\tau<l(\gamma_i)$ such that the disks centered at $\xi_i$ and $\eta_i$ with radius $\tau$ are contained in $\mathbb{D}\setminus\hat{A}_i$.
		Let $\nu$ be a unit normal vector to $\gamma_i$ in $\mathbb{R}^2$. For each $t>0$, we define a rectangle
		$$
		R^t = \{z + s\nu \mid z \in \gamma_i, s \in [0, t]\}.
		$$
		Since $\mathbb{D}$ is convex, by definition we have $R^{\tau}\subset\mathbb{D}$.
		For each $t\in(0,\tau)$, let $\gamma_i^t = \{z + t\nu \mid z \in \gamma_i\}$ be the line segment in $R^{\tau}$ parallel to $\gamma_i$ at distance $t$. We first show that there exists $t_0\in(0,\tau)$ such that $\gamma_i^{t_0}\cap \hat{A}_i=\varnothing$. Assume the contrary. Define $\Pi:R^\tau\to[0,\tau]$ by $\Pi(z+s\nu)=s$ for $z\in\gamma_i$ and $0\le s\le\tau$. Then
        $(0,\tau)\subset\Pi(\hat A_i\cap R^\tau)$. This contradicts the fact that $\Pi$ is Lipschitz and $\mathscr{H}^1(\hat{A}_i)=0$. Now let $\tilde{\gamma}_i=\partial R^{t_0}\setminus\gamma_i^{\circ}$, where $\gamma_i^{\circ}$ denotes the relative interior of $\gamma_i$. It follows that $\tilde{\gamma}_i$ is a curve in $\mathbb{D}\setminus \hat{A}_i$ connecting $\xi_i$ and $\eta_i$ with length $l(\tilde{\gamma}_i)= l(\gamma_i)+2t_0\leq3l(\gamma_i)$. Let $\tilde{\sigma}_i=f_j^{-1}(\tilde{\gamma}_i)$. Then we have 
		\begin{equation}\label{eq: tilde-sigma-i}
			l(\tilde{\sigma}_i)\leq L\cdot l(\tilde{\gamma}_i)\leq 3L\cdot l(\gamma_i)\leq 3L^2l(\sigma_i)=3L^2|t_i-s_i|.
		\end{equation}
		Now let $\tilde{\sigma}$ be the curve obtained from $\sigma$ by replacing each $\sigma_i$ with $\tilde{\sigma}_i$. By \eqref{eq: tilde-sigma-i},
		$$
		l(\tilde{\sigma})\leq l(\sigma)+\sum_{i=1}^{n}l(\tilde{\sigma}_i)\leq l(\sigma)+3L^2\sum_{i=1}^{n}|t_i-s_i| \leq l(\sigma)+3L^2\delta.
		$$
		Then the result follows by taking $\delta=\varepsilon/3L^2$.
	\end{proof}
    
	Following the approach in \cite{luo2026rigidity}, we now establish the foundational framework for studying the rigidity of noncompact convex sets via the Pogorelov map. For an isometry $F: \partial X_1\to\partial X_2$,  define 
	$$
	S_1 = \{P_1(x, F(x)) : x \in \partial X_1\}, \quad S_2 = \{P_2(x, F(x)) : x \in \partial X_1\}.
	$$
	Assume that both $X_1$ and $X_2$ contain $0$ in their interiors. By \cref{thm: PogorelovII}, $S_1$ and $S_2$ are two radial surfaces in $\mathbb{R}^3$ with respect to $0$, and there exists an isometry $G : S_1 \to S_2$ such that $\Phi(x, F(x)) = (y, G(y))$ holds for all $x \in \partial X_1$.
	For each $i = 1, 2$, let 
    $$
    W_i = \{x / \|x\| : x \in S_i\},\quad Z_i = \{kx : k \geq 0, x \in W_i\}.
    $$ Clearly $W_i\subset\mathbb{S}^2$, and $Z_i$ is star-shaped with center $0$. Since $S_i$ is a radial surface without boundary, Brouwer's invariance of domain theorem implies that $W_i$ is an open subset of $\mathbb{S}^2$ which is homeomorphic to $S_i$. Consequently, $Z_i \setminus \{0\}$ is homeomorphic to $S_i \times \mathbb{R}$.
	
	Assume that $X_i^P=\mathscr{D}_i\cup\mathscr{P}_i$, where $\mathscr{D}_i$ is a circle-type closed set and $\mathscr{P}_i$ is a set having zero one-dimensional Hausdorff measure in $\mathbb{R}^3$. Let $V_i=\mathbb{S}^2-\mathscr{D}_i$ and $U_i= \{kx : k \geq 0, x \in V_i\}$. Using essentially the same method as in \cite[Lemma 5.1]{luo2026rigidity}, we have the following lemma.
	
	\begin{lemma}\label{lem:MeasureZero}
		The one-dimensional Hausdorff measure $\mathscr{H}^1(V_i \setminus W_i)$ of $V_i \setminus W_i$ is zero. Furthermore, the two-dimensional Hausdorff measure $\mathscr{H}^2(U_i \setminus Z_i)$ is zero.
	\end{lemma}
	\noindent Define the function $q_i : Z_i \to \mathbb{R}$ by 
	\begin{equation*}
		q_i(\alpha) = \inf\{k > 0 : \alpha/k \in S_i\} \geq 0.
	\end{equation*}
    Clearly, $q_i(k\alpha) = k\cdot q_i(\alpha)$ for all $\alpha\in Z_i$ and $k > 0$. The following lemma shows that $q_i$ is a locally convex function on the open set $Z_i \setminus \{0\}$.
   \begin{lemma}\label{lem: equivalence-locally-convex-to-0}
        Let $S$ be a radial surface in $\mathbb{R}^3$. Define 
        $$
        W = \{x / \|x\| : x \in S\}\quad\text{and}\quad Z = \{kx : k \geq 0, x \in W\}.
        $$
        Then $S$ is locally convex with respect to $0$ if and only if the function $q : Z \to \mathbb{R}$ defined by 
    	\begin{equation*}
    		q(\alpha) = \inf\{k > 0 : \alpha/k \in S\} \geq 0,
    	\end{equation*}
        is a locally convex function on the open set $Z \setminus \{0\}$.
    \end{lemma}
    \begin{proof}
        The fact that $q(\alpha)$ is locally convex provided $S$ is locally convex with respect to $0$ is proved in \cite[p.~20]{luo2026rigidity}. We only prove the converse. Suppose that $q(\alpha)$ is a locally convex function on the open set $Z \setminus \{0\}$. By definition, $S=\{\alpha\in Z:q(\alpha)=1\}$. Let $p$ be a point in $S$. Choose a sufficiently small open ball $B$ containing $p$ such that the closure of $B$ is contained in $Z\setminus\{0\}$. Let
        $C_0(B) = \{kx : k > 0, x \in B\}$ be the cone over $B$ from $0$. Then $C_0(B)$ is a convex subset of $Z$. By \cref{lem: LocalToGlobalConvex}, $q(\alpha)$ is a convex function on $C_0(B)$. It follows that 
        $$
        K_p:=\{\alpha\in C_0(B):q(\alpha)\leq 1\}\cup\{0\}
        $$
        is a convex body in $\mathbb{R}^3$ containing $0$. By the choice of the ball $B$, $B\cap S$ is a neighborhood of $p$ in $S$ and $B\cap S\subset\{\alpha\in C_0(B):q(\alpha)= 1\}$. Then by the definition of the set $K_p$, $B\cap S\subset\partial K_p$. It follows that $S$ is locally convex with respect to $0$, which completes the proof.
    \end{proof}
    By \cref{thm: ExtendConvex} and \cref{lem:MeasureZero}, we can extend $q_i$ to a locally convex function $\tilde{q}_i$ on $U_i$ such that $\tilde{q}_i(k\alpha) = k \cdot \tilde{q}_i(\alpha)$ for all $\alpha \in U_i$ and $k > 0$.
	Let 
    $$
    \tilde{Y}_i = {\{\alpha\in U_i : \tilde{q}_i(\alpha) \leq 1\}},\quad\tilde{S}_i = {\{\alpha\in U_i : \tilde{q}_i(\alpha)=1\}}.
    $$
    Then, by \cref{lem: equivalence-locally-convex-to-0}, $\tilde{S}_i$ is locally convex with respect to $0$. Note that by \cref{lem: SumIs2}, $S_i$ is contained in the ball $B_2(0)=\{x:\|x\|<2\}$, which implies that $\tilde{Y}_i$ is bounded in $\mathbb{R}^3$. Let $Q_i: V_i\to\tilde{S}_i$ be the inverse of the radial projection $R:\tilde{S}_i \to\mathbb{S}^2$. Since $\tilde{S}_i$ is a locally convex surface with respect to $0$, the following lemma shows that $Q_i$ is a locally Lipschitz map.
    \begin{lemma}
        Let $S$ be a radial surface in $\mathbb{R}^3$ that is locally convex with respect to $0$. Suppose that $Q$ is the inverse of the radial projection $R:S \to\mathbb{S}^2$. Then $Q$ is a locally Lipschitz map on $R(S)$.
    \end{lemma}
    \begin{proof}
        Let $p$ be an arbitrary point in $S$. Since $S$ is locally convex with respect to $0$, there exists a neighborhood $U_p$ of $p$ in $S$ and a convex set $K_p$ containing $0$ such that $U_p\subset S \cap \partial K_p$. Let $B_r(0)$ be the ball in $\mathbb{R}^3$ centered at $0$ with radius $r>0$. Denote by $K_p(r)$ the convex hull of the set $K_p\cup B_r(0)$ in $\mathbb{R}^3$. We claim that there exists $r>0$ such that $\partial K_p(r)$ contains a neighborhood $U_p'$ of $p$ in $S$. Suppose this claim is false. Then for each $r>0$ and each neighborhood $V_p\subset U_p$ of $p$ in $S$, there exists a point $x\in V_p\setminus \partial K_p(r)$. Let $H$ be a supporting plane of $K_p$ at $x$. Then we must have $H\cap B_r(0)\neq \varnothing$. Otherwise, if $H\cap B_r(0)=\varnothing$, then $K_p\cup B_r(0)$ lies on the same side of $H$. It follows that $H$ is also a supporting plane of $K_p(r)$ and $x\in \partial K_p(r)$, which is a contradiction. Therefore, we can choose a sequence $r_n\to 0$ and a sequence $x_n$ in $S$ converging to $p$ such that for each $n\in\mathbb{N}$, there exists a supporting plane $H_{n}$ of $K_p$ at $x_n$ such that $H_{n}\cap B_{r_n}(0)\neq \varnothing$. Passing to a subsequence if necessary, we may assume that the plane sequence $\{H_n\}_n$ converges to some plane $H_0$. It follows that $H_0$ is a supporting plane of $K_p$ at $p$ and $0\in H_0$, which contradicts the fact that $S$ is a radial surface. Thus, there exists $r_0>0$ and a neighborhood $U_p'$ of $p$ such that $U_p'\subset \partial K_p(r_0)$. By \cref{thm: interior-Lipschitz}, there exists a constant $L>0$ such that
        $$
        \|Q(x)-Q(y)\|\leq L\|x-y\|,\quad\forall x,y\in R(U_p').
        $$
        Since $R$ is a homeomorphism from $S$ to $R(S)$, it follows that $Q$ is locally Lipschitz continuous, which completes the proof.
    \end{proof}

	Since $\mathscr{H}^1(V_i \setminus W_i)=0$ by \cref{lem:MeasureZero} and $Q_i$ is locally Lipschitz continuous, we see that the one-dimensional Hausdorff measure of $\tilde{S}_i\setminus S_i=Q_i(V_i \setminus W_i)$ is zero. Moreover, by construction, $\tilde{S}_i$ is locally convex and therefore locally Lipschitz. It then follows from \cref{thm: zero-H1-dont-change-metric} that $d_{\tilde{S}_i}(x,y)=d_{S_i}(x,y)$ for all $x,y\in S_i$. A natural continuation, following the approach in \cite{luo2026rigidity}, would be to extend the isometry $G : S_1 \to S_2$ to an isometry $\tilde{G}:\tilde{S}_1 \to \tilde{S}_2$. However, since $\tilde S_1$ and $\tilde S_2$ need not be complete, the standard extension theorem for isometries does not immediately yield an isometry between them. We defer this step and first establish the following extension criterion.
    
    \begin{lemma}\label{lem: Isometry-Extension}
        Let $(X, d_X)$ and $(Y, d_Y)$ be two metric spaces with dense subspaces $X_1 \subseteq X$ and $Y_1 \subseteq Y$. Let $(\overline{Y}, d_{\overline{Y}})$ be a complete metric space such that there exists a topological embedding $\iota: Y \hookrightarrow \overline{Y}$ satisfying $d_{\overline{Y}}(\iota(y_1),\iota(y_2))\leq d_{Y}(y_1,y_2)$ for all $y_1,y_2\in Y$. Suppose $f: X_1 \to Y_1$ is an isometry satisfying the following property:
        \medskip
        \begin{adjustwidth}{2em}{2em}
            For each sequence $\{x_n\}_n$ in $X_1$ such that $\{x_n\}_n$ converges to some $x_0 \in X$ and $\{f(x_n)\}_n$ converges in $\overline{Y}$ to some $y_0 \in \overline{Y}$, we have $y_0 \in Y$.
        \end{adjustwidth}
        \medskip
        Then $f$ admits a unique extension to an isometric embedding $\tilde{f}: X \to Y$.
    \end{lemma}
    
    \begin{proof}
        Since $d_{\overline{Y}}(y_1,y_2)\leq d_{Y}(y_1,y_2)$ for all $y_1,y_2\in Y$, the map $f:X_1\to\overline{Y}$ is Lipschitz continuous. Given that $\overline{Y}$ is a complete metric space and $X_1$ is dense in $X$, $f$ extends to a unique continuous map $\tilde{f}:X\to \overline{Y}$. Let $x_0\in X$ and $\{x_n\}_n$ be a sequence in $X_1$ that converges to $x_0$ in $X$. Then, $f(x_n)=\tilde{f}(x_n)$ converges to a point $\tilde{f}(x_0)$ in $\overline{Y}$. By the condition of $f$, $\tilde{f}(x_0)\in Y$. Since $Y$ is an embedded subspace of $\overline{Y}$, it follows that $\tilde{f}$ is a continuous map from $X$ to $Y$. Now we show that $\tilde{f}$ is an isometric embedding. Let $x_1,x_2$ be two points in $X$. Choose sequences $\{x_n^1\}_n$ and $\{x_n^2\}_n$ in $X_1$ converging to $x_1$ and $x_2$, respectively. Since $\tilde{f}$ is continuous, we have
        $$ 
        d_Y(\tilde{f}(x_1), \tilde{f}(x_2)) = \lim_{n \to \infty} d_Y(f(x_n^1), f(x_n^2)) =\lim_{n \to \infty} d_X(x_n^1, x_n^2)=d_X(x_1,x_2),
        $$
        which completes the proof.
    \end{proof}
	
	\subsection{Deformation of the Pogorelov map under isometries}\label{sec: Deformation of the Pogorelov map under isometries}
	
	\subsubsection{$\Phi(\Gamma(x),F(x))$}
	Let $\Gamma$ be an isometry of $\mathbb{H}^3_P$. It is well known that the isometries of $\mathbb{H}^3_P$ are exactly the Möbius transformations of the one-point compactification $\hat{\mathbb{R}}^3$ of $\mathbb{R}^3$ that map $\mathbb{H}^3_P$ onto itself. In this paper, we shall always regard $\Gamma$ as a homeomorphism from the closure $\overline{\mathbb{H}^3_P}$ onto itself. The restriction of $\Gamma$ to the boundary $\partial\mathbb{H}^3_P$ acts as a Möbius transformation of the Riemann sphere. Now, consider the Pogorelov map 
	$$
	\Phi(\Gamma(x),F(x))=\left(\frac{2P(\Gamma(x))}{-\langle \Gamma(x)+ F(x),e\rangle},\frac{2P(F(x))}{-\langle\Gamma(x)+ F(x),e\rangle}\right).
	$$
	Note that in this paper we identify $\mathcal{H}^3$ and $\mathbb{H}^3_P$ via the canonical map $\frac{P(x)}{-\langle x,e \rangle+1}$ and use these two models interchangeably. Under this identification, $e$ corresponds to $0$ in the Poincaré model.
	
	\begin{lemma}\label{lem: Expression-of-L1-General-gamma}
		Let $\Gamma$ be an isometry of $\mathbb{H}^3_P$ and denote by $x_0$ the point in $\mathbb{H}^3_P$ such that $\Gamma(x_0)=0$. Suppose that $\{x_n\}_n$ is a sequence in $\partial X_1$ converging to an ideal point $p\in\partial \mathbb{H}^3_P$ in $\mathbb{R}^3$. Then the limit 
        $$
        A:=\lim_{n\to\infty}\|P_1(x_n,F(x_n))\|
        $$ 
        exists if and only if the limit $$
        A':=\lim_{n\to\infty}\|P_1(\Gamma(x_n),F(x_n))\|
        $$ 
        exists. Furthermore, if these limits exist, they satisfy the relation
		$$
		A'=\frac{2A}{A+c_p(2-A)},\quad\text{where}\quad
        c_p=\frac{1-\|x_0\|^2}{\|p-x_0\|^2}>0.
		$$
		In particular, $A\in(0,2)$ if and only if $A'\in(0,2)$.
	\end{lemma}
	\begin{proof}
		For any two points $x,y\in\mathbb{H}^3_P$, we have
		$$
		\cosh d_{\mathbb{H}^3_P}(x,y)=1+\frac{2\|x-y\|^2}{(1-\|x\|^2)(1-\|y\|^2)}.
		$$
		Since $\{x_n\}_n$ converges to $p$ in $\mathbb{R}^3$, it follows that 
		\begin{equation}\label{eq: infinity-ratio}
			\lim_{n \to \infty} \frac{\langle x_n,x_0\rangle}{\langle x_n,e\rangle} =\lim_{n \to \infty} \frac{\cosh(d_{\mathcal{H}^3}(x_n, x_0))}{\cosh(d_{\mathcal{H}^3}(x_n, e))}=\lim_{n \to \infty} \frac{\cosh(d_{\mathbb{H}^3_P}(x_n, x_0))}{\cosh(d_{\mathbb{H}^3_P}(x_n, 0))} = \frac{\|p-x_0\|^2}{1-\|x_0\|^2}.
		\end{equation}
		Suppose that $A:=\lim_{n\to\infty}\|P_1(x_n,F(x_n))\|$ exists. If $A=0$, the defining formula gives $\langle F(x_n),e\rangle/\langle x_n,e\rangle\to+\infty$ as $n\to\infty$. It then follows from \eqref{eq: infinity-ratio} that $\langle F(x_n),e\rangle/\langle \Gamma (x_n),e\rangle\to+\infty$ and therefore $A'=0$. We may therefore assume $A>0$ in the calculation below. The converse follows by applying the same argument to $\Gamma^{-1}$. Since 
        $$
        \|P(x_n)\|/(-\langle x_n, e\rangle)\to 1\quad \text{as}\ n\to\infty,
        $$
        by the definition of $P_1(x_n,F(x_n))$, we have
		$$
		\lim_{n\to\infty}\frac{\langle F(x_n),e \rangle}{\langle x_n,e \rangle}=\frac{2-A}{A}.
		$$
		Therefore,
		$$
		\lim_{n\to\infty}\frac{\langle F(x_n),e \rangle}{\langle \Gamma(x_n),e \rangle}=\lim_{n\to\infty}\frac{\langle F(x_n),e \rangle}{\langle x_n,e \rangle}\cdot \lim_{n\to\infty}\frac{\langle x_n,e \rangle}{\langle x_n,x_0 \rangle}=\frac{2-A}{A}\cdot\frac{1-\|x_0\|^2}{\|p-x_0\|^2},
		$$
		where the first equality holds because $\Gamma$ is a Lorentz transformation, which implies
		$$
		\langle \Gamma(x_n), e\rangle =\langle  x_n,\Gamma^{-1} (e)\rangle=\langle  x_n,x_0\rangle.
		$$
		It follows that
		$$
		\lim_{n\to\infty}\|P_1(\Gamma(x_n),F(x_n))\|=\lim_{n\to\infty} \frac{2}{1+ \frac{\langle F(x_n),e \rangle}{\langle \Gamma(x_n),e \rangle}}=\frac{2}{1+\frac{2-A}{A}\cdot\frac{1-\|x_0\|^2}{\|p-x_0\|^2}}.
		$$
		This completes the proof.
	\end{proof}
	
	\subsubsection{$\Phi(\Gamma_1(x),\Gamma_2\circ F(x))$}
	In this subsection, suppose that both $X_1^P$ and $X_2^P$ contain $\overline{\mathbb{S}^2_{-}}$ as a disk component.
	Let $e^{i\alpha}=(\cos\alpha,\sin\alpha,0)\in\mathbb{R}^2$ be a unit complex number. Assume that for any $\alpha\in[0,2\pi)$ and any sequence $\{x_n\}_n$ in $\partial X_1$ converging to $e^{i\alpha}$,
	\begin{equation}\label{eq: condition-of-two-circle}
	    \lim_{n\to\infty}P_1(x_n,F(x_n))=\lim_{n\to \infty}P_2(x_n,F(x_n))=e^{i\alpha}.
	\end{equation}
	It follows that $F(x_n)\to e^{i\alpha}$ in $\mathbb{R}^3$ as $n\to\infty$.
	Since 
    $$
    \|P_1(x_n,F(x_n))\|\to 1\quad\text{and}\quad\frac{\|P(x_n)\|}{-\langle x_n, e\rangle}\to 1\quad \text{as}\ n\to\infty,
    $$
    we have 
	$$
	\lim_{n\to\infty} \frac{\langle F(x_n),e \rangle}{\langle x_n,e \rangle}=\frac{2-1}{1}=1.
	$$
	Let $\Gamma_1,\Gamma_2$ be two isometries of $\mathbb{H}^3_P$. We use $x_0$ and $y_0$ to denote the points in $\mathbb{H}^3_P$ such that $\Gamma_1(x_0)=0$ and $\Gamma_2(y_0)=0$, respectively. Consider the Pogorelov map 
	$$
	\Phi(\Gamma_1(x),\Gamma_2 \circ
	F(x))=\left(\frac{2P(\Gamma_1(x))}{-\langle \Gamma_1(x)+ \Gamma_2 \circ F(x),e\rangle},\frac{2P(\Gamma_2 \circ F(x))}{-\langle\Gamma_1(x)+ \Gamma_2 \circ F(x),e\rangle}\right).
	$$
	Since both $x_n$ and $F(x_n)$ converge to $e^{i\alpha}$ as $n\to\infty$, by \eqref{eq: infinity-ratio} we have
	\begin{equation}\label{eq: h(alpha)-definition}
		\begin{aligned}
			\lim_{n\to \infty} \frac{\langle \Gamma_2 \circ F(x_n),e\rangle}{\langle \Gamma_1(x_n),e\rangle}&=\lim_{n\to \infty} \frac{\langle \Gamma_2 \circ F(x_n),e\rangle}{\langle F(x_n),e\rangle}\cdot \frac{\langle  F(x_n),e\rangle}{\langle  x_n,e\rangle} 
			\cdot\frac{\langle  x_n,e\rangle}{\langle \Gamma_1(x_n),e\rangle}\\
			&= \lim_{n\to \infty} \frac{\langle  F(x_n),\Gamma_2^{-1}(e)\rangle}{\langle F(x_n),e\rangle}\cdot \frac{\langle  F(x_n),e\rangle}{\langle  x_n,e\rangle} 
			\cdot\frac{\langle  x_n,e\rangle}{\langle x_n,\Gamma_1^{-1}(e)\rangle}\\
			&=\lim_{n\to \infty} \frac{\langle  F(x_n),y_0\rangle}{\langle F(x_n),e\rangle}
			\cdot\frac{\langle  x_n,e\rangle}{\langle x_n,x_0\rangle}\\
			&=\frac{\|e^{i\alpha}-y_0\|^2}{1-\|y_0\|^2}\frac{1-\|x_0\|^2}{\|e^{i\alpha}-x_0\|^2}.
		\end{aligned}
	\end{equation}
	Assuming \eqref{eq: condition-of-two-circle}, we define 
	$$
	L_{i;\, \Gamma_1,\Gamma_2}(\alpha)=\lim_{n\to \infty}  P_i(\Gamma_1(x_n),\Gamma_2 \circ
	F(x_n)),\quad \forall \alpha\in[0,2\pi),\ i\in\{1,2\}.
	$$
	It is clear that $L_{i;\, \Gamma_1,\Gamma_2}$ exists and is independent of the choice of $\{x_n\}_n$ in $\partial X_1$ converging to $e^{i\alpha}$. The following lemma will be useful in the proof of \cref{thm: MainThm1} in \cref{sec: the-general-case}.
	
	\begin{lemma}\label{lem: unit-circle-to-ellipse}
		Suppose that \eqref{eq: condition-of-two-circle} holds. Then for each $i\in\{1,2\}$, the curve $L_{i;\, \Gamma_1,\Gamma_2}(\alpha)$ parametrizes an ellipse in some plane in $\mathbb{R}^3$. Moreover, there exists an isometry $T\in\operatorname{Isom}(\mathbb{R}^3)$ such that $L_{2;\, \Gamma_1,\Gamma_2}(\alpha)=T\left(L_{1;\, \Gamma_1,\Gamma_2}(\alpha)\right)$ for all $\alpha\in[0,2\pi)$.
	\end{lemma}
	
	\begin{proof}
		Let $\gamma_1$ and $\gamma_2$ denote the restrictions of $\Gamma_1$ and $\Gamma_2$ to the unit sphere, respectively. By \eqref{eq: h(alpha)-definition} we have
		$$
		L_{1;\, \Gamma_1,\Gamma_2}(\alpha)=\frac{2}{1+h(\alpha)}\gamma_1(e^{i\alpha})\quad\text{and}\quad
		L_{2;\, \Gamma_1,\Gamma_2}(\alpha)=\frac{2h(\alpha)}{1+h(\alpha)}\gamma_2(e^{i\alpha}),
		$$
		where
		$$
		h(\alpha):=\frac{\|e^{i\alpha}-y_0\|^2}{1-\|y_0\|^2}\frac{1-\|x_0\|^2}{\|e^{i\alpha}-x_0\|^2}.
		$$
		Let $a\in\mathbb{H}^3_P$. It is known that if an isometry $\Gamma$ of $\mathbb{H}^3_P$ satisfies $\Gamma(a)=0$, then the differential $D\gamma(p):T_p\mathbb{S}^2\to T_{\gamma(p)}\mathbb{S}^2$ of its restriction $\gamma=\Gamma|_{\mathbb{S}^2}$ at any point $p\in\mathbb{S}^2$ acts by a conformal scaling
        $$
        \|D\gamma(p)(v)\|=\frac{1 - \|a\|^2}{\|p - a\|^2}\|v\|, \quad\forall v\in T_p\mathbb{S}^2.
        $$ 
        A detailed discussion can be found, for example, in \cite[Section 3.4]{MR698777}. Therefore, setting $\gamma=\gamma_2\circ\gamma_1^{-1}$ and $L(\alpha)=\gamma_1(e^{i\alpha})$, we obtain
		$$
		h(\alpha) = \frac{\|D\gamma_1(e^{i\alpha})\|}{\|D\gamma_2(e^{i\alpha})\|} = \frac{1}{\|D\gamma(L(\alpha))\|} =\frac{\|L(\alpha)-a\|^2}{1-\|a\|^2},
		$$
        where $a$ is the point in $\mathbb{H}^3_P$ such that $\Gamma_2\circ\Gamma_1^{-1}(a)=0$.
		Substituting this back into the expression for $L_{1;\, \Gamma_1,\Gamma_2}(\alpha)$, we have
		$$
		L_{1;\, \Gamma_1,\Gamma_2}(\alpha)=\frac{1-\|a\|^2}{1-( L(\alpha),a )} L(\alpha).
		$$
        Let $\lambda(\alpha)=\frac{1-\|a\|^2}{1-( L(\alpha),a )}$.
		It follows that
		$$
		\big( L_{1;\, \Gamma_1,\Gamma_2}(\alpha),a \big) = \lambda(\alpha)( L(\alpha),a )=\frac{(1-\|a\|^2)( L(\alpha),a )}{1-( L(\alpha),a )} =\lambda(\alpha)-(1-\|a\|^2),
		$$
		which implies that 
        $$
        \lambda(\alpha)=\big( L_{1;\, \Gamma_1,\Gamma_2}(\alpha),a \big)+1-\|a\|^2.
        $$
		Since Möbius transformations map circles on $\mathbb{S}^2$ to circles, the image $L(\alpha)$ is a circle and thus lies on some plane $H$. Let $N$ be a normal vector to the plane $H$; then $( L(\alpha),N )=C$ for some constant $C\in\mathbb{R}$. We then have
		$$
		\big( L_{1;\, \Gamma_1,\Gamma_2}(\alpha),N \big)=\lambda(\alpha)( L(\alpha),N )=C\lambda(\alpha)=C\big( L_{1;\, \Gamma_1,\Gamma_2}(\alpha),a \big)+C(1-\|a\|^2).
		$$
		Thus, $\big( L_{1;\, \Gamma_1,\Gamma_2}(\alpha), N-Ca \big) = C(1-\|a\|^2)$ is constant. Consequently, $L_{1;\, \Gamma_1,\Gamma_2}(\alpha)$ lies on a plane in $\mathbb{R}^3$. Since $L_{1;\, \Gamma_1,\Gamma_2}(\alpha)$ also lies on the cone subtended by the circle $L(\alpha)$ from the origin, its image must be an ellipse.
		
		Finally, we show that there exists an isometry $T\in\operatorname{Isom}(\mathbb{R}^3)$ such that $L_{2;\, \Gamma_1,\Gamma_2}(\alpha)=T\left(L_{1;\, \Gamma_1,\Gamma_2}(\alpha)\right)$. Set $\Gamma:=\Gamma_2\circ\Gamma_1^{-1}$. By the classical decomposition of Möbius transformations, we can write $\Gamma=A\circ\Gamma_a$, where $A\in O(3)$ is an orthogonal transformation and 
		\begin{equation}\label{eq:Ball-transformation-a-to-0}
		    \Gamma_a(x)= \frac{a\|x-a\|^2-(1-\|a\|^2)(x-a)}{1-2(a,x)+\|a\|^2\|x\|^2}.
		\end{equation}
		Then, the restriction $\gamma_a$ of $\Gamma_a$ to $\mathbb{S}^2$ can be expressed as
		$$
        \begin{aligned}
            \gamma_a(p) = \frac{a\|p - a\|^2 - (1 -\|a\|^2)(p - a)}{\|p- a\|^2}& = a-\|D\gamma(p)\|(p-a)\\
            &= (1+\|D\gamma(p)\|)a-\|D\gamma(p)\|p.
        \end{aligned}
		$$
		Since $h(\alpha)=1/\|D\gamma(L(\alpha))\|$, we have
		$$
		\begin{aligned}
			L_{2;\, \Gamma_1,\Gamma_2}(\alpha)&=\frac{2}{1+\|D\gamma(L(\alpha))\|}A\big(\gamma_a(L(\alpha))\big)
			\\&=\frac{2}{1+\|D\gamma(L(\alpha))\|}A\Big(\left(1+\|D\gamma(L(\alpha))\|\right)a-\|D\gamma(L(\alpha))\|L(\alpha)\Big)
			\\&=A \left(2a-\frac{2\|D\gamma(L(\alpha))\|}{1+\|D\gamma(L(\alpha))\|}L(\alpha)\right)
			\\&=A\big(2a-L_{1;\, \Gamma_1,\Gamma_2}(\alpha)\big),
		\end{aligned}
		$$
		which completes the proof.
	\end{proof}

\section{The case of a single disk}\label{sec: single-disk}
	In this section, we study the case where both $X_1^P$ and $X_2^P$ contain only one disk component. Up to an isometry of $\mathbb{H}^3_P$, we may assume that both $X_1^P$ and $X_2^P$ are the union of the closed half-sphere $\overline{\mathbb{S}^2_{-}}=\mathbb{S}^2\cap\{x_3\leq0\}$ and a set of zero one-dimensional Hausdorff measure. In this paper, we identify $\mathbb{S}^1$ with the equator $\mathbb{S}^2\cap\{x_3=0\}$.

\subsection{The half-space case}\label{subsec:single-half-sapce}
    In this subsection, we consider the situation where the interior of $X_1$ does not contain the origin $0$. Then we must have $X_1=\mathbb{H}^3_P\cap \overline{\mathbb{R}^3_{-}}$ and $X^P_1=X^P_2=\overline{\mathbb{S}^2_{-}}$. It follows that $\partial X_1=\mathbb{H}^3_P\cap \mathbb{R}^2$ is the hyperbolic plane. Let
    $$
    S_1=\{P_1(x,F(x)):x\in\partial X_1\}\quad \text{and}\quad S_2=\{P_2(x,F(x)):x\in\partial X_1\}.
    $$
    Suppose, for contradiction, that $\partial X_2\neq \mathbb{H}^3_P\cap\mathbb{R}^2$. This implies that $\partial X_2\subset \mathbb{H}^3_P\cap\mathbb{R}^3_{+}$ and therefore $S_2\subset \mathbb{R}^3_{+}$. We first show that there exists an intrinsic isometry $G:S_1\to S_2$. For every $i\in\{1,2\}$ and every $x\in\partial X_1$, set $\varphi_i(x):=P_i(x,F(x))$. Let $\gamma:[0,\infty)\to\partial X_1$ be a unit-speed geodesic ray with $\gamma(0)=0$, and set $b(r):=\cosh d_{\mathbb H^3_P}(F(\gamma(r)),0)$. Since $F$ is an intrinsic isometry, the function $r\mapsto d_{\mathbb H^3_P}(F(\gamma(r)),0)$ is $1$-Lipschitz. Hence $|b'|\leq\sqrt{b^2-1}$ almost everywhere. By definition, the map $\varphi_1$ preserves radial directions and its radial coordinate $u(r)$ along $\gamma$ satisfies 
    \begin{equation}\label{eq: planar-component-monotone} 
    \left\{\begin{aligned} u(r)&=\frac{2\sinh r}{\cosh r+b(r)},\\ u'(r)&= \frac{2\bigl(1+b(r)\cosh r-b'(r)\sinh r\bigr)} {(\cosh r+b(r))^2} \geq \frac{2(1+b(r)\cosh r-\sqrt{b(r)^2-1}\sinh r)}{(\cosh r+b(r))^2}>0,\quad a.e. \end{aligned}\right.
    \end{equation} 
    Thus $\varphi_1: \partial X_1\to S_1$ is a bijection. By Brouwer's invariance of domain theorem, $S_1$ is an open planar domain in $B_2(0) \cap \mathbb{R}^2$ containing $0$.
    The uniform positive lower bound for $\partial_r u$ on compact subsets of $\partial X_1 \setminus \{0\}$, together with the local Lipschitz continuity of $u$ with respect to the angular variable, implies that $\varphi_1^{-1}$ is locally Lipschitz away from $0$. To see this, write $\varphi_1$ in polar coordinates as $(r, \theta) \mapsto (u(r,\theta), \theta)$. Let $K\subset (0,\infty) \times \mathbb{S}^1$ be a compact set, and let $c_K > 0$ and $L_K > 0$ be the uniform positive lower bound of $\partial_r u$ and the Lipschitz constant of $u$ with respect to $\theta$ on $K$, respectively. For any two points $(r_1, \theta_1), (r_2, \theta_2) \in K$, denote $\rho_i := u(r_i, \theta_i)$ for $i = 1, 2$. Integrating $\partial_r u$ along the radial direction yields
    $$
    \begin{aligned} 
    c_K \vert{}r_1 - r_2\vert{}  &\leq  \vert{}u(r_1,\theta_1) - u(r_2,\theta_1)\vert{} \\ &\leq \vert{}\rho_1 - \rho_2\vert{} + \vert{}u(r_2,\theta_2) - u(r_2,\theta_1)\vert{} \\ &\leq \vert{}\rho_1 - \rho_2\vert{} + L_K d_{\mathbb{S}^1}(\theta_1, \theta_2),
    \end{aligned}
    $$
    which implies that $\varphi_1^{-1}$ is Lipschitz on $\varphi_1(K)$.
    Near the origin, the local Lipschitz continuity of $\varphi_1^{-1}$ follows from the Cartesian representation $\varphi_1(x) = a(x)x$, where the scalar function $a$ is locally Lipschitz with $a(0) > 0$. Therefore, we conclude that $\varphi_1^{-1}$ is locally Lipschitz.
    
    For $\varphi_2$, let $R(x)=x/\|x\|$ on $\mathbb{R}^3\setminus\{0\}$, and let $Q_2:\mathbb{S}^2\setminus X_2^P=\mathbb{S}^2_+\to\partial X_2$ be the inverse of $R|_{\partial X_2}$.
    Since $0\in\operatorname{int}X_2$, the map $Q_2$ is locally Lipschitz for the spherical metric and the intrinsic hyperbolic metric on $\partial X_2$.
    To see this, we view $X_2$ in the Klein projective model $\mathbb{B}^3$, in which $X_2$ is a Euclidean convex set containing the origin in its interior. For any point $p \in \partial X_2$, we choose a radius $r < 1$ such that $p \in B_r(0)$. Applying Theorem 2.1 to the Euclidean convex body $X_2 \cap \overline{B}_r(0)$, the Euclidean inverse radial projection is locally Lipschitz near $R(p)$ with respect to the ambient Euclidean metric, and hence also with respect to the intrinsic metric on the convex boundary. Since the hyperbolic Riemannian metric is bi-Lipschitz equivalent to the Euclidean metric on the compact ball $\overline{B}_r(0)$, pulling back along spherical arcs yields the required local Lipschitz bound for $Q_2$ under the intrinsic hyperbolic metric. Since $S_2\subset\mathbb{R}^3\setminus\{0\}$, $R$ is also locally Lipschitz. Note that $R\circ\varphi_2=R\circ F$. We then have
    $$
    \varphi_2^{-1}=F^{-1}\circ Q_2\circ R,
    $$
    which implies that $\varphi_2^{-1}$ is a locally Lipschitz map from $S_2$ to $\partial X_1$ with respect to their intrinsic metrics. Now, define $G := \varphi_2 \circ \varphi_1^{-1} \colon S_1 \to S_2$. Let $\gamma$ be an absolutely continuous curve in $S_1$, and set $\alpha := \varphi_1^{-1}(\gamma)$ and $\beta := F(\alpha)$. Since $\varphi_1^{-1}$ is locally Lipschitz, $\alpha$ and $\beta$ are absolutely continuous curves in $\mathbb{H}^3_P$ having the same length. It then follows from \cref{lem: preserve-intrinsic-isometries} that $\gamma$ and $G(\gamma)$ have the same length in $\mathbb{R}^3$, which implies that $d_{S_2}(G(x_1), G(x_2)) \leq d_{S_1}(x_1, x_2)$ for all $x_1, x_2 \in S_1$. A symmetric argument applied to $G^{-1} = \varphi_1 \circ \varphi_2^{-1}$ yields the reverse inequality, proving that $G$ is an intrinsic isometry from $S_1$ to $S_2$.
    
    Define the function $q : \mathbb{R}^3_{+} \to \mathbb{R}$ by
    \begin{equation*}
    q(\alpha) = \inf\{k > 0 : \alpha/k \in S_2\} \geq 0.
    \end{equation*}
    \begin{lemma}
    The function $q$ is a convex function on $\mathbb{R}^3_{+}$.
    \end{lemma}
    \begin{proof}
    Let $\{\Gamma_n\}_{n}$ be a sequence of isometries of $\mathbb{H}^3_P$ converging to the identity map $\operatorname{id}_{\mathbb{H}^3_P}$ such that $\Gamma_n^{-1}(0)\in \mathbb{H}^3_P\cap\mathbb{R}^3_{-}$ for all $n\in\mathbb{N}$. Consider the corresponding Pogorelov map $\Phi(\Gamma_n(x),F(x))$. Since $\Gamma_n^{-1}(0)\in \mathbb{H}^3_P\cap\mathbb{R}^3_{-}$ for all $n\in\mathbb{N}$, we see that $\Gamma_n(X_1)$ contains $0$ in its interior. Let
    $$
    S_{1,n}=\{P_1(\Gamma_n(x),F(x)):x\in\partial X_1\}\quad \text{and}\quad S_{2,n}=\{P_2(\Gamma_n(x),F(x)):x\in\partial X_1\}.
    $$
    It follows from \cref{thm: PogorelovII} that both $S_{1,n}$ and $S_{2,n}$ are locally convex surfaces with respect to $0$. Define the function $q^{n} : \mathbb{R}^3_{+} \to \mathbb{R}$ by
    \begin{equation*}
    q^{n}(\alpha) = \inf\{k > 0 : \alpha/k \in S_{2,n}\}\geq 0.
    \end{equation*}
    Since $S_{2,n}$ is locally convex with respect to $0$, \cref{lem: equivalence-locally-convex-to-0} implies that $q^{n}$ is locally convex on $\mathbb{R}^3_{+}$. Therefore, by \cref{lem: LocalToGlobalConvex}, $q^{n}$ is a globally convex function on $\mathbb{R}^3_{+}$. It is evident that $q$ is the pointwise limit of the sequence of convex functions $\{q^n\}_n$. Hence, $q$ is also a convex function on $\mathbb{R}^3_{+}$, which completes the proof.
    \end{proof}
    \noindent Let
    $$
    K = \overline{\{\alpha\in \mathbb{R}^3_{+} : q(\alpha) \leq 1\}}.
    $$
    Since $q$ is a convex function on $\mathbb{R}^3_{+}$ by the above lemma and $S_2\subset B_2(0)$ by \cref{lem: SumIs2}, $K$ is a compact convex body in $\overline{\mathbb{R}^3_+}$. It follows that $S_2=\partial K\setminus\mathbb{R}^2$ is a convex surface. 
    
    Now, we show that the convex surface $S_2$ cannot be isometric to any planar domain, thereby yielding the desired contradiction. For this purpose, we recall the notion of the extrinsic curvature measure of convex surfaces.
    Let $M$ be a subset of a convex surface $S$ in $\mathbb{R}^3$. The spherical image of $M$ is the set of outer normal vectors to the supporting planes of $S$ at points of $M$; it is a subset of the unit sphere $\mathbb{S}^2$. For any Borel set $A$ of $S$, its spherical image is Lebesgue measurable in $\mathbb{S}^2$; we then define the measure of the spherical image of $A$ to be the \emph{extrinsic curvature measure} of $A$ on the surface $S$. For regular surfaces, the Gauss theorem asserts that the extrinsic curvature measure of a domain coincides with its intrinsic integral curvature. Alexandrov generalized this fundamental result to arbitrary convex surfaces. In particular, if a convex surface $S$ is locally isometric to a planar domain, then the extrinsic curvature measure of every Borel subset of $S$ must equal zero. We refer the reader to \cite{Intrinsic-geometry-of-convex-surfaces} for a comprehensive introduction to this theory. 
    
    Returning to the proof, since $S_2$ is isometric to the planar domain $S_1$, the extrinsic curvature measure of every open subset of $S_2$ must equal zero. To derive a contradiction, we choose a closed ball $B$ in $\mathbb{R}^3$ such that $\partial K \setminus B$ is a nonempty set contained entirely within $S_2$. Fixing the center of $B$, we continuously expand its radius until the convex body $K$ is entirely contained within the enlarged ball $\tilde{B}$. By this construction, there exists at least one contact point $p \in S_2 \cap \partial \tilde{B}$. Let $H$ be the supporting plane of $\tilde{B}$ at $p$. Clearly, $H$ is also a supporting plane of $K$ at $p$ and $H\cap K=\{p\}$. Let $\vec{n}\in\mathbb{S}^2$ be the outer normal vector to the supporting plane $H$ of $S_2$ at $p$, and let $U$ be a neighborhood of $p$ in $S_2$. Since $H \cap K = \{p\}$, there exists $\delta > 0$ such that the portion of $\partial K$ lying between $H$ and the parallel plane $H - \delta \vec{n} $ is entirely contained in $U$.  It follows that there exists a small neighborhood $W$ of $\vec{n}$ in $\mathbb{S}^2$ such that for any plane $H'$ passing through $p$ with a normal vector $\vec{n}' \in W$, the intersection $H' \cap \partial K$ is contained in $U$. By translating the plane $H'$ along $\vec{n}'$ away from $K$ if necessary, we eventually obtain a supporting plane of $K$ at some point $p' \in U$ with the outer normal vector $\vec{n}'$. This implies that the entire neighborhood $W$ on $\mathbb{S}^2$ is contained in the spherical image of $U$. Consequently, the extrinsic curvature measure of $U$ is strictly positive, which is a contradiction. This contradiction implies that $\partial X_2 = \mathbb{H}^3_P \cap \mathbb{R}^2 = \partial X_1$. Since any isometry of the hyperbolic plane $\mathbb{H}^3_P \cap \mathbb{R}^2$ onto itself extends to a global isometry of $\mathbb{H}^3_P$, the conclusion follows.

\subsection{The non-half-space case}
	By the arguments in \cref{subsec:single-half-sapce}, we can now assume that both $X_1$ and $X_2$ contain $0$ in their interiors. Recall the definitions of $S_i$, $\tilde{q}_i$, and $\tilde{S}_i$ in \cref{sec: pogorelov}. In the current setting, $\tilde{q}_i$ is a convex function defined on the upper half-space $\mathbb{R}^3_{+}$. Let 
    $$
    {K}_i = \overline{\{\alpha\in \mathbb{R}^3_{+} : \tilde{q}_i(\alpha) \leq 1\}},\quad i\in\{1,2\}.
    $$
    It is clear that $K_i$ is a compact convex body in $\overline{\mathbb{R}^3_+}$ and $\mathbb{R}^2$ is a supporting plane of $K_i$.
	In view of \cref{thm: Pogorelov} and \cref{prop: RnIsometry-to-HnIsometry}, it suffices to prove that ${G}:{S}_1\to{S}_2$ can be extended to an isometry between $\partial{K}_1$ and $\partial K_2$. Let $P_i=\partial K_i\setminus\tilde{S}_i=\partial K_i\cap\mathbb{R}^2$ and denote by $\overline{S_i}$ the closure of $\tilde{S}_i$ in $\mathbb{R}^3$. Then $P_i$ is a compact convex set in $\mathbb{R}^2$ containing $0$. 
    The main idea is to deform $P_i$ by composing the Pogorelov map $\Phi$ with isometries of $\mathbb{H}^3_P$ that preserve the lower hemisphere $\mathbb{S}^2_{-}$. Note that the set of orientation-preserving isometries on $\mathbb{H}^3_P$ mapping $\mathbb{S}^2_{-}$ onto itself is in one-to-one correspondence with the set of Möbius transformations on the unit open disk $\mathbb{D}$ via restriction. 
    Accordingly, throughout this section, any Möbius transformation $\Gamma$ of $\mathbb{D}$ will be identified with its unique extension to an orientation-preserving isometry of $\mathbb{H}^3_P$.
	
	Let $\Gamma$ be an isometry of $\mathbb{H}^3_P$ such that $\Gamma(X_1)$ contains $0$ in its interior. Consider the Pogorelov map $\Phi(\Gamma(x), F(x))$. Define 
	$$
	S_{1,\Gamma} = \{P_1(\Gamma(x), F(x)) : x \in \partial X_1\}, \quad S_{2,\Gamma} = \{P_2(\Gamma(x), F(x)) : x \in \partial X_1\}.
	$$
	Following an argument analogous to that in \cref{sec: pogorelov}, there exists an isometry $G_{\Gamma}$ from $S_{1,\Gamma}$ to $S_{2,\Gamma}$. Furthermore, each $S_{i,\Gamma}$ admits an extension to a surface $\tilde{S}_{i,\Gamma}$ that is locally convex with respect to $0$. Under this extension, the radial projection $R$ induces homeomorphisms from $\tilde{S}_{1,\Gamma}$ onto $\mathbb{S}^2\setminus \Gamma(\overline{\mathbb{S}^2_{-}})$ and from $\tilde{S}_{2,\Gamma}$ onto $\mathbb{S}^2_{+}$.
    
    Now we assume that $\Gamma$ preserves the lower hemisphere $\mathbb{S}^2_{-}$. Then $\tilde{S}_{i,\Gamma}=\partial K_{i,\Gamma}\cap \mathbb{R}^3_{+}$ for some compact convex body $K_{i,\Gamma}$ in $\overline{\mathbb{R}^3_{+}}$. Let $P_{i,\Gamma}=K_{i,\Gamma}\cap\mathbb{R}^2$. Then $P_{i,\Gamma}$ is a compact convex set in $\mathbb{R}^2$.
	In the rest of this section, for each $a\in\mathbb{R}$, we identify the quotient space $[a, a+2\pi]/\{a, a+2\pi\}$ with the unit circle $\mathbb{S}^1$ via the homeomorphism $\alpha \mapsto e^{i\alpha}$. The restriction $\gamma:\mathbb{S}^1\to\mathbb{S}^1$ of the Möbius transformation $\Gamma$ to the unit circle is understood to be a real valued function on $[0,2\pi)$ such that $e^{i\gamma(\alpha)}=\Gamma(e^{i\alpha})$ for all $\alpha\in[0,2\pi)$. The following lemma describes the relationship between the sets $P_{1,\Gamma}$ and $P_1$.
	
	\begin{lemma}\label{lem: Expression-of-L1-gamma}
		Let $\Gamma$ be an orientation-preserving isometry of $\mathbb{H}^3_P$ mapping $\mathbb{S}^2_{-}$ onto itself. Assume that $P_1$ has nonempty interior in $\mathbb{R}^2$. Suppose that
		$$
		L_1(\alpha)=r_1(\alpha)e^{i\alpha},\quad\alpha\in[\theta_1,\theta_2],
		$$
		is a curve segment of $\partial P_1$ contained in $\mathbb{R}^3\setminus\{0\}$, where $0\leq\theta_1<\theta_2\leq 2\pi$. Then
		$$
		L_{1,\Gamma}(\alpha)=\frac{2{r_1(\alpha)}}{(2-r_1(\alpha))\gamma'(\alpha)+{r_1(\alpha)}}e^{i\gamma(\alpha)},\quad\alpha\in[\theta_1,\theta_2]
		$$
		is a curve segment of $\partial P_{1,\Gamma}$ contained in $\mathbb{R}^3\setminus\{0\}$, where $\gamma$ is the restriction of $\Gamma$ to $\mathbb{S}^1$, i.e., $e^{i\gamma(\alpha)}=\Gamma(e^{i\alpha})$ for all $\alpha\in[0,2\pi)$. In particular, $0$ is an interior point of $P_1$ if and only if $0$ is an interior point of $P_{1,\Gamma}$.
	\end{lemma}
	\begin{proof}
		Let 
		$$
		r_{1,\Gamma}(\alpha)=\sup\{t\geq0:te^{i\alpha}\in P_{1,\Gamma}\},\quad \alpha\in[0,2\pi),
		$$ 
        be the radial function of $P_{1,\Gamma}$.
		For each $\alpha\in[\theta_1,\theta_2]$, it is clear that
		$$
		r_1(\alpha)=\lim_{n\to\infty} \sup\{\|P_1(x, F(x))\|:x\in \partial X_1\cap B_{1/n}(e^{i\alpha})\},
		$$
		and
		$$
		r_{1,\Gamma}(\gamma(\alpha))=\lim_{n\to\infty} \sup\{\|P_1(\Gamma(x), F(x))\|:x\in \partial  X_1\cap B_{1/n}(e^{i\alpha})\},
		$$
        where $B_{1/n}(e^{i\alpha})$ is the ball in $\mathbb{R}^3$ centered at $e^{i\alpha}=(\cos\alpha,\sin\alpha,0)$ with radius $1/n$.
		Let $x_0$ be the point in $\mathbb{H}^3_P$ such that $\Gamma (x_0)=0$. Since $\gamma'(\alpha)=\frac{1-\|x_0\|^2}{\|e^{i\alpha}-x_0\|^2}$, by \cref{lem: Expression-of-L1-General-gamma}, we have
		$$
		r_{1,\Gamma}(\gamma(\alpha))=\frac{2{r_1(\alpha)}}{(2-r_1(\alpha))\gamma'(\alpha)+{r_1(\alpha)}}>0,\quad\forall \alpha\in[\theta_1,\theta_2].
		$$
		Thus, 
		$$
		L_{1,\Gamma}(\alpha)=r_{1,\Gamma}(\gamma(\alpha))e^{i\gamma(\alpha)},\quad\alpha\in[\theta_1,\theta_2]
		$$
		is a curve segment of $\partial P_{1,\Gamma}$, which completes the proof.
		
	\end{proof}

	We now show that each $P_i$, $i\in\{1,2\}$, is a compact convex body in $\mathbb{R}^2$. The proof relies on the following lemma:

    \begin{lemma}\label{lem:Pi-not-segment-or-point}
        Let $D\subset\mathbb{S}^2_+$ be a closed spherical disk centered at $e_3=(0,0,1)$, and let $U$ be a neighborhood of $D$ in $\mathbb{S}_+^2$. Let $\psi$ be a positive continuous function on $U\setminus D$. Suppose that
        $$
            S=\{\psi(w)w:w\in U\setminus D\}
        $$
        is locally convex with respect to $0$. Then, for every $p\in\partial D$,
        \begin{equation}\label{eq:psi-at-least-two-not-zero}
            \limsup_{\substack{x\to p,\ x\in U\setminus D}}\psi(x)>0.
        \end{equation}
    \end{lemma}
    
    \begin{proof}
    Identify the upper hemisphere with the tangent plane
    $$
        P=T_{e_3}\mathbb{S}^2=\{x_3=1\}
    $$
    by the radial map $w\mapsto w/w_3$. Let $Z=\{kx:k\geq0,x\in S\}$ and let $q:Z\to\mathbb{R}$ be defined as
    \begin{equation*}
    		q(\alpha) = \inf\{k > 0 : \alpha/k \in S\}.
    \end{equation*}
    Then by \cref{lem: equivalence-locally-convex-to-0}, we see that $q$ is a locally convex function on $Z\setminus\{0\}$. Let $F:=q|_{P\cap Z}$. It follows that $F$ is locally convex and we have
    \begin{equation}\label{eq:F-is-q-restriction}
        F\left(\frac{w}{w_3}\right)=\frac{1}{\psi(w)w_3}, \quad\forall w\in U \setminus D.
    \end{equation}
    
    Let $B, V \subset P$ be the subsets corresponding to $D$ and $U$ under the radial map, respectively. Note that $B$ is a Euclidean disk. Fix $y\in\partial B$, and let $\nu$ and $\tau$ be the outward unit normal vector and a unit tangent vector of $\partial B$ at $y$, respectively. Choose $a,s_0>0$ sufficiently small such that $y+s\nu+t\tau\in V$ whenever $0\leq s\leq s_0$ and $|t|\leq a$. For each $0<s<s_0$, the segment
    $$
        I_s=[y+s\nu-a\tau,\,y+s\nu+a\tau]
    $$
    lies in $V\setminus B$. Hence, by the convexity of $F|_{I_s}$,
    \begin{equation}\label{eq:lem47-midpoint}
        F(y+s\nu)
        \leq\frac{F(y+s\nu-a\tau)+F(y+s\nu+a\tau)}2.
    \end{equation}
    The right-hand side remains bounded as $s\to 0^{+}$, since $y+s\nu \pm a\tau$ converge to the fixed points $y\pm a\tau\in V\setminus B$; see \cref{fig:lem4-7}.

    \begin{figure}[htbp]
        \centering
        \begin{tikzpicture}[
            x=1.35cm,y=1.35cm,
            point/.style={circle,fill=black,inner sep=1.05pt},
            openpoint/.style={circle,draw=black,fill=white,inner sep=1.1pt},
            every node/.style={font=\small},
            >=Latex]
            \def\R{1.55}
            \def\s{0.55}
            \def\a{0.88}
            \fill[black!7] (0,0) circle (\R);
            \draw[thick] (0,0) circle (\R);

            \coordinate (Y) at (\R,0);
            \coordinate (M) at ({\R+\s},0);
            \coordinate (EzeroP) at (\R,\a);
            \coordinate (EzeroM) at (\R,-\a);
            \coordinate (EsP) at ({\R+\s},\a);
            \coordinate (EsM) at ({\R+\s},-\a);

            \draw[dashed] (EzeroM)--(EzeroP);
            \draw[very thick] (EsM)--(EsP);
            \draw[->,thick] (Y)--(M) node[midway,below=4pt] {$s\nu$};
            \draw[->,thick] ($(Y)+(-0.10,0.06)$)--($(Y)+(-0.10,0.62)$)
                node[midway,left=2pt] {$\tau$};

            \node[point] at (Y) {};
            \node[anchor=north east] at ($(Y)+(-0.03,-0.03)$) {$y$};
            \node[point] at (M) {};
            \node[anchor=north west] at ($(M)+(0.04,-0.03)$) {$y+s\nu$};
            \node[openpoint] at (EzeroP) {};
            \node[openpoint] at (EzeroM) {};
            \node[point] at (EsP) {};
            \node[point] at (EsM) {};
            \node at (-0.45,0.14) {$B$};
            \node[anchor=west] at ({\R+0.08},0.55) {$I_0$};
            \node[anchor=west] at ({\R+\s+0.08},0.55) {$I_s$};
        \end{tikzpicture}
        \caption{The tangent segment $I_0$ and its outward translate $I_s$. The endpoint values in \eqref{eq:lem47-midpoint} converge to finite values outside $B$.}
        \label{fig:lem4-7}
    \end{figure}

    Set $w_s=(y+s\nu)/\lVert y+s\nu\rVert$. Then $w_s\to p:=y/\lVert y\rVert\in\partial D$ as $s\to 0^+$, and \eqref{eq:F-is-q-restriction} implies that
    $$
        \psi(w_s)=\frac{\lVert y+s\nu\rVert}{F(y+s\nu)}.
    $$
    As $s \to 0^+$, the numerator is bounded away from zero, and the denominator is bounded above. Thus, $\liminf_{s \to 0^+} \psi(w_s) > 0$ and \eqref{eq:psi-at-least-two-not-zero} holds for $p=y/\|y\|$. Since $y \in \partial B$ was chosen arbitrarily, the assertion holds at every point of $\partial D$, which completes the proof.
    \end{proof}

	\begin{proposition}\label{prop: Pi-contain-0-interior}
		The compact convex sets $P_1$ and $P_2$ in $\mathbb{R}^2$ both have $0$ in their interiors.
	\end{proposition}
	\begin{proof}
		We only prove the statement for $P_1$. The conclusion for $P_2$ then follows by applying the same argument to the Pogorelov map $\Phi(x,F^{-1}(x))$ on $\partial X_2$. Suppose, for contradiction, that $P_1$ does not have $0$ in its interior. 
		
		\textbf{Case 1.} $P_1=\{0\}$ or $P_1$ is a line segment from $0$ to some point $x_0 \in \mathbb{R}^2$.
		Let $D$ be a disk satisfying the conditions in \cref{lem:Pi-not-segment-or-point}. Given the assumption that $X_1$ contains $0$ in its interior, there exists an isometry $\Gamma$ of $\mathbb{H}^3_P$ mapping $\overline{\mathbb{S}^2_{-}}$ onto $D$ such that $0$ lies in the interior of $\Gamma(X_1)$. Consider the Pogorelov map $\Phi(\Gamma(x), F(x))$. Let $Q_{1,\Gamma}$ be the inverse map of the radial projection $R: \tilde{S}_{1,\Gamma}\to\mathbb{S}^2$. We define $\psi(x)=\|{Q}_{1,\Gamma}(x)\|$ for all $x\in R(\tilde{S}_{1,\Gamma})$. 
        Since the radial function of $P_1$ is nonzero in at most one direction, by \cref{lem: Expression-of-L1-General-gamma} there exists at most one point $p\in\partial D$ such that
		\begin{equation*}
			\limsup_{x \to p} \psi(x) > 0.
		\end{equation*}
        However, since $\tilde{S}_{1,\Gamma}$ is locally convex with respect to $0$, by \cref{lem:Pi-not-segment-or-point} we have
		\begin{equation*}
			\limsup_{x \to p} \psi(x) > 0 \quad\forall p\in\partial D,
		\end{equation*}
        which is a contradiction.
      
		\textbf{Case 2.} $0$ is not an interior point of $P_1$ in $\mathbb{R}^2$ and there exist at least two directions $p_1,p_2\in \mathbb{S}^1$ such that $P_1\cap\{tp_i:t>0\}\neq\varnothing$. After a rotation, we may assume that $p_1=1$ and $p_2=e^{i\theta}$ for some $\theta\in(0,\pi]$.
		
		\begin{enumerate}
			\item[]\textbf{Case 2.1.} $\theta<\pi$. Since $P_1$ is convex, for each $\alpha\in(0,\theta)$, $\partial P_1\cap\{te^{i\alpha}:t>0\}$ contains exactly one element. Let 
			$$
			L_1(\alpha)=r_1(\alpha)e^{i\alpha},\quad\alpha\in[0,\theta]
			$$
			be the polar parametrization of the curve segment of $\partial P_1$ between angle $0$ and $\theta$. It is clear that $r_1(\alpha)>0$ for all $\alpha\in[0,\theta]$. Now let 
			$$
			\Gamma(z)=e^{-i\theta}\frac{2z-(1+e^{i\theta})}{(1+e^{-i\theta})z-2}.
			$$
			Then, $\Gamma$ is a Möbius transformation of $\mathbb{D}$ and can be extended as an orientation-preserving isometry of $\mathbb{H}^3_P$ mapping $\mathbb{S}^2_{-}$ onto itself. Consider the Pogorelov map $\Phi(\Gamma(x),F(x))$. Since $\Gamma(1)=1$ and $\Gamma(e^{i\theta})=e^{-i\theta}$, by \cref{lem: Expression-of-L1-gamma} $\partial P_{1,\Gamma}$ contains a curve segment that subtends a cone angle of $2\pi-\theta>\pi$ at the origin. The convexity of $P_{1,\Gamma}$ then forces $0$ to be an interior point of $P_{1,\Gamma}$. Note that $P_1=(P_{1,\Gamma})_{\Gamma^{-1}}$, applying \cref{lem: Expression-of-L1-gamma} once more shows that $0$ is an interior point of $P_1$, which is a contradiction.
			
			\item[]\textbf{Case 2.2.} $\theta=\pi$. Choose a Möbius transformation $\Gamma$ of $\mathbb{D}$ such that $\Gamma(1)=1$ and $\Gamma(-1)=i$. Following the same argument as in the proof of \cref{lem: Expression-of-L1-gamma}, we see that $P_{1,\Gamma}$ has nonempty intersection with both $\{t:t>0\}$ and $\{te^{i\pi/2}:t>0\}$. It follows that $P_{1,\Gamma}$ falls into Case 2.1 with $\theta=\frac{\pi}{2}$, which has already been shown to yield a contradiction.
		\end{enumerate}
	\end{proof}

    \begin{corollary}\label{cor: G-extends-closure-G}
        The isometry $G : S_1 \to S_2$ extends to a unique isometry $\overline{G} : \overline{S_1} \to \overline{S_2}$, where $\overline{S_1}$ and $\overline{S_2}$ are the closures of $S_1$ and $S_2$ in $\mathbb{R}^3$, respectively.
    \end{corollary}
    \begin{proof}
         Since each $P_i$ has nonempty interior in $\mathbb{R}^2$, $\tilde{S}_i$ satisfies the conditions in \cref{lem: path-metric-boundary-removable}. Then, according to \cref{thm: zero-H1-dont-change-metric} and \cref{lem: path-metric-boundary-removable}, we have 
         $$
         d_{S_i}(x,y)=d_{\tilde{S}_i}(x,y)=d_{\overline{S_i}}(x,y),\quad\forall x,y\in S_i,\ i\in\{1,2\}.
         $$
         It follows that the space $(\overline{S_i},d_{\overline{S_i}})$ is the metric completion of $({S}_i,d_{{S}_i})$. Since ${G}$ is an isometry between $({S}_1,d_{{S}_1})$ and $({S}_2,d_{{S}_2})$, a classical extension theorem for metric spaces ensures that $G$ extends to a unique isometry between $(\overline{S_1},d_{\overline{S_1}})$ and $(\overline{S_2},d_{\overline{S_2}})$.
    \end{proof}

    The following lemma shows a relation between the shapes of $P_1$ and $P_2$.
	
	\begin{lemma}\label{lem: L2-formula-via-L1}
        Suppose that
		$$
		L_1(\alpha)=r_1(\alpha)e^{i\alpha},\quad\alpha\in[0,2\pi),
		$$
		is the polar parametrization of $\partial P_1$, where $r_1(\alpha)\in(0,2)$ for all $\alpha\in[0,2\pi)$. Then there exists a continuous and strictly monotone function $f: [0, 2\pi) \to \mathbb{R}$ such that 
		$$
		L_2(\alpha) = (2 - r_1(\alpha)) e^{i f(\alpha)},\quad\alpha\in[0,2\pi),
		$$
		parametrizes the curve $\partial P_2$. Moreover, $\overline{G}(L_1(\alpha))=L_2(\alpha)$ for all $\alpha\in[0,2\pi)$.
	\end{lemma}
	
	\begin{proof}		
		Define $L_2(\alpha) = \overline{G}(L_1(\alpha))$ and consider the normalized direction map $h(\alpha) = \frac{L_2(\alpha)}{\|L_2(\alpha)\|}$. Since $\overline{G}$ acts as a homeomorphism between the boundary curves, $h$ is a continuous bijection from $[0,2\pi)$ onto $\mathbb{S}^1$. It follows that there exists a continuous and strictly monotone function $f: [0, 2\pi) \to \mathbb{R}$ such that $h(\alpha) = e^{i f(\alpha)}$. Furthermore, by \cref{lem: SumIs2}, the radial length 
		$$
		\|L_2(\alpha)\|=2-\|L_1(\alpha)\|=2-r_1(\alpha),\quad \forall \alpha\in[0,2\pi).
		$$
		Combining these facts, we see that $\partial P_2$ can be explicitly parametrized as 
		$$
		L_2(\alpha) = (2 - r_1(\alpha)) e^{i f(\alpha)},\quad\alpha\in[0,2\pi).
		$$
		This completes the proof.
	\end{proof}
	
	Under the identification between $[a, a+2\pi]/\{a, a+2\pi\}$ and $\mathbb{S}^1$, it is clear that the map $f$ introduced in \cref{lem: L2-formula-via-L1} defines a homeomorphism of $\mathbb{S}^1$. By composing the Pogorelov map $\Phi(x,F(x))$ with a reflection of $\mathbb{H}^3_P$ if necessary, we may assume that $f$ preserves the orientation of $\mathbb{S}^1$. Then, we have the following regularity result.

	\begin{lemma}\label{lem: f'-formula}
		The map $f$ belongs to $\operatorname{Diff}^1(\mathbb{S}^1)$. Moreover, the derivative 
		$$
		f'(\alpha)=\dfrac{r_1(\alpha)}{2-r_1(\alpha)}>0,\quad\forall \alpha\in[0,2\pi),
		$$ 
		is a Lipschitz function on $[0,2\pi]$.
	\end{lemma}
	\begin{proof}
		We first prove that $f$ is Lipschitz continuous on $[0,2\pi]$. Since $P_1$ is a compact convex set containing the origin in its interior, by \cref{thm: interior-Lipschitz}, the parametrization $L_1(\alpha) = r_1(\alpha) e^{i\alpha}$ is a Lipschitz map on $[0,2\pi]$ with respect to the intrinsic path metric $d_{\partial P_1}$ of $\partial P_1$. According to \cref{lem: L2-formula-via-L1}, we have $L_2(\alpha) = \overline{G}(L_1(\alpha))$ for all $\alpha \in [0,2\pi]$. Since $\overline{G}:\overline{S_1}\to\overline{S_2}$ is an isometry with respect to their intrinsic path metrics, $L_2: [0,2\pi] \to \partial P_2$ is a Lipschitz map with respect to the intrinsic path metric $d_{\partial P_2}$ of $\partial P_2$.
        By \cref{lem: Bi-Lipschitz}, for each $i\in\{1,2\}$, the intrinsic path metric $d_{\partial P_i}$ of $\partial P_i$ is bi-Lipschitz equivalent to the ambient Euclidean metric. Therefore, we conclude that each map $L_i$ is Lipschitz continuous with respect to the ambient Euclidean metric.
		Recall that $L_2(\alpha) = (2 - r_1(\alpha)) e^{i f(\alpha)}$. Since $r_1:[0,2\pi]\to (0,2)$ is Lipschitz continuous and the image of $r_1$ lies in a compact interval within $(0,2)$, the function $\frac{1}{2 - r_1(\alpha)}$ is Lipschitz continuous on $[0,2\pi]$. Therefore, the map 
        $$
        \alpha \mapsto e^{i f(\alpha)}=\frac{L_2(\alpha)}{2 - r_1(\alpha)}
        $$
        itself is Lipschitz. It follows that $f$ is a Lipschitz function on $[0,2\pi]$.
		
		For almost every $\alpha \in [0,2\pi]$, the arc-length elements of $\partial P_1$ and $\partial P_2$ exist and are given by
		\begin{equation}\label{eq: ds1-elements}
			ds_1 = \|L_1'(\alpha)\| \, d\alpha = \sqrt{ r_1(\alpha)^2 + r_1'(\alpha)^2 } \, d\alpha
		\end{equation}
		and
		\begin{equation}\label{eq: ds2-elements}
			ds_2 = \|L_2'(\alpha)\| \, d\alpha = \sqrt{ r_1'(\alpha)^2 + (2 - r_1(\alpha))^2 f'(\alpha)^2 } \, d\alpha,
		\end{equation}
		respectively. Since $\overline{G}$ is an intrinsic isometry between $\partial P_1$ and $\partial P_2$, it follows that $ds_1 = ds_2$ for almost every $\alpha \in [0,2\pi]$. By the expressions \eqref{eq: ds1-elements} and \eqref{eq: ds2-elements}, we have the identity
		$$
		r_1(\alpha)^2 = (2 - r_1(\alpha))^2 f'(\alpha)^2 \quad \text{for a.e. }\alpha\in[0,2\pi].
		$$
		Since $f$ is assumed to preserve the orientation of $\mathbb{S}^1$, it follows that $f'(\alpha)\geq0$ for almost every $\alpha\in[0,2\pi]$, which yields
		\begin{equation}\label{eq: f_prime}
			f'(\alpha) = \frac{r_1(\alpha)}{2 - r_1(\alpha)}>0 \quad \text{for a.e. }\alpha\in[0,2\pi].
		\end{equation}
		Since $f$ is Lipschitz and the right-hand side of \eqref{eq: f_prime} is a continuous function of $\alpha$ on $[0,2\pi]$, it follows that $f$ is continuously differentiable on $\mathbb{S}^1$ and $f'(\alpha) = \frac{r_1(\alpha)}{2 - r_1(\alpha)}>0$ for all $\alpha\in[0,2\pi]$. Thus, $f \in \operatorname{Diff}^1(\mathbb{S}^1)$. The Lipschitz continuity of $f'$ follows directly from the fact that both $r_1(\alpha)$ and $\frac{1}{2 - r_1(\alpha)}$ are Lipschitz functions on $[0,2\pi]$. This completes the proof.
	\end{proof}
	
	For each isometry $\Gamma$ of $\mathbb{H}^3_P$ preserving the lower hemisphere, define $\gamma:\mathbb{S}^1\to \mathbb{S}^1$ to be the restriction of $\Gamma$ to the unit circle, i.e., $\Gamma(e^{i\alpha})=e^{i\gamma(\alpha)}$. We consider the Pogorelov map $\Phi(\Gamma(x),F(x))$. By \cref{lem: Expression-of-L1-gamma} and \cref{lem: f'-formula},
	\begin{equation}\label{eq: L1gamma-formula}
		L_{1,\Gamma}(\alpha)=\frac{2f'(\alpha)}{\gamma'(\alpha)+f'(\alpha)}e^{i\gamma(\alpha)}
	\end{equation}
	is a parametrization of $\partial P_{1,\Gamma}$. We now show that $f(\alpha)$ must be the restriction of some Möbius transformation of the unit disk to the unit circle $\mathbb{S}^1$.
	
	\begin{proposition}\label{prop: f -is-Möbius}
		Suppose that both $P_1$ and $P_2$ in $\mathbb{R}^2$ contain the origin in their interiors. Then the map $f$ is the restriction of some Möbius transformation $\Gamma_0$ on the unit open disk $\mathbb{D}$ to the unit circle, i.e., $\Gamma_0(e^{i\alpha})=e^{if(\alpha)}$ for all $\alpha\in[0,2\pi)$.
	\end{proposition}
	\begin{proof}
		\textbf{Step 1. The convexity inequality for $\Phi(\Gamma(x), F(x))$.}
		
		Fix $z_0=r_0e^{i\theta_0}$ in the unit open disk $\mathbb{D}$. Let $\Gamma(z)$ be a Möbius transformation of $\mathbb{D}$ mapping $z_0$ to $0$. It follows that $\Gamma(z)=e^{i\theta}\frac{z-z_0}{1-\overline{z_0}z}$ for some $\theta\in[0,2\pi)$. The transformation $\Gamma$ extends uniquely to an orientation-preserving isometry of $\mathbb{H}^3_P$ that maps $\mathbb{S}^2_{-}$ onto itself. Denote by $\gamma:\mathbb{S}^1\to \mathbb{S}^1$ the restriction of $\Gamma$ to the unit circle, i.e., $\Gamma(e^{i\alpha})=e^{i\gamma(\alpha)}$. Then we have
		$$
		\gamma'(\alpha)=\frac{1-\|z_0\|^2}{\|e^{i\alpha}-z_0\|^2}=\frac{1-r_0^2}{1+r_0^2-2r_0\cos(\alpha-\theta_0)},\quad\forall\alpha\in[0,2\pi).
		$$
		We have shown that $P_{1,\Gamma}$ is a compact convex set in $\mathbb{R}^2$ containing $0$ in its interior. We now use up to second-order derivatives of $f$ to characterize the convexity of $P_{1,\Gamma}$. It is well known that the convexity of a planar domain $U$ is equivalent to the existence of a supporting line at every boundary point of $U$. Since $f'(\alpha)$ is Lipschitz continuous, the second derivative $f''(\alpha)$ exists for almost every $\alpha\in[0,2\pi)$. Let $L_{1,\Gamma}(\beta)$ be defined as in \eqref{eq: L1gamma-formula}. It follows that for almost every $\beta\in [0,2\pi)$, the tangent vector $L_{1,\Gamma}'(\beta)$ exists, and the line parametrized by $t \mapsto L_{1,\Gamma}(\beta)+tL_{1,\Gamma}'(\beta)$ is a supporting line of $P_{1,\Gamma}$ at the point $L_{1,\Gamma}(\beta)$.
		
		Since $0$ lies in the interior of $P_{1,\Gamma}$ and the boundary $\partial P_{1,\Gamma}$ is parametrized by $L_{1,\Gamma}(\beta)$ in the counterclockwise order as $\beta$ increases from $0$ to $2\pi$, the inner product of the vector $L_{1,\Gamma}(\alpha)-L_{1,\Gamma}(\beta)$ with the inward normal vector $i L_{1,\Gamma}'(\beta)$ must be nonnegative for every $\alpha$ and almost every $\beta$. In terms of complex arithmetic, this is equivalent to 
		\begin{equation}\label{eq: ComplexExpression}
			\operatorname{Im}\left[ (L_{1,\Gamma}(\alpha) - L_{1,\Gamma}(\beta)) \overline{L_{1,\Gamma}'(\beta)} \right] \geq 0, \quad \text{ for a.e. } \alpha,\beta \in [0,2\pi).
		\end{equation}
		Substituting the expression \eqref{eq: L1gamma-formula} for $L_{1,\Gamma}$ into \eqref{eq: ComplexExpression} and multiplying both sides by the strictly positive factor 
		$$
		\frac{\big(f'(\beta)+\gamma'(\beta)\big)^2\big(f'(\alpha)+\gamma'(\alpha)\big)}{4},
		$$
		we obtain
		\begin{equation}\label{eq: FirstExpression1}
			\begin{aligned}
				&f'(\alpha) \left( {f''(\beta)\gamma'(\beta) - f'(\beta)\gamma''(\beta)} \right) \sin(\gamma(\alpha)-\gamma(\beta))   \\
				-& f'(\alpha){f'(\beta)}\gamma'(\beta)\big(\gamma'(\beta)+f'(\beta)\big) \cos(\gamma(\alpha)-\gamma(\beta))+ f'(\beta)^2\gamma'(\beta)\big(\gamma'(\alpha)+f'(\alpha)\big) \geq 0.
			\end{aligned}
		\end{equation}

		Next, we insert the explicit formula $\gamma'(\alpha)=\frac{1-r_0^2}{1+r_0^2-2r_0\cos(\alpha-\theta_0)}$ into \eqref{eq: FirstExpression1}. For simplicity, set $k(\alpha)=\frac{1}{\gamma'(\alpha)}$. It follows that 
		\begin{equation}\label{eq: k-and-gamma}
			\gamma'(\alpha)=\frac{1}{k(\alpha)}, \quad \gamma''(\alpha)=-\frac{k'(\alpha)}{k(\alpha)^2}.
		\end{equation}
		Since $\gamma$ is the restriction of the Möbius transformation $\Gamma$ to $\mathbb{S}^1$, we have the trigonometric identities
		\begin{equation}\label{eq: cos-and-sin}
			\begin{aligned}
				&\cos\big(\gamma(\alpha)-\gamma(\beta)\big) = \frac{k(\alpha)k(\beta) - (1-\cos(\alpha-\beta))}{k(\alpha)k(\beta)},\\
				&\sin\big(\gamma(\alpha)-\gamma(\beta)\big) = \frac{k(\beta)\sin(\alpha-\beta) + k'(\beta)(1-\cos(\alpha-\beta))}{k(\alpha)k(\beta)}.
			\end{aligned}
		\end{equation}
		Substituting equations \eqref{eq: k-and-gamma} and \eqref{eq: cos-and-sin} into \eqref{eq: FirstExpression1} and multiplying both sides of the inequality by the positive factor $k(\alpha)k(\beta)^3$, we obtain
		\begin{equation}\label{eq: SecondExpression1}
			\begin{aligned}
				&f'(\alpha) \big[ f''(\beta) k(\beta) + f'(\beta) k'(\beta)\big]\big[ k(\beta)\sin(\alpha-\beta) + k'(\beta)(1-\cos(\alpha-\beta)) \big]\\
				-&f'(\alpha)f'(\beta) \big(1+f'(\beta)k(\beta)\big) \big[ k(\alpha)k(\beta) - (1-\cos(\alpha-\beta)) \big]\\
				+&f'(\beta)^2 k(\beta)^2 \big[1+f'(\alpha)k(\alpha)\big]\geq0.
			\end{aligned}
		\end{equation}
		
		Let $A=\frac{1+r_0^2}{1-r_0^2}$ and $B=\frac{2r_0}{1-r_0^2}$. Then $k(\beta)$ and its derivative can be expressed as
		\begin{equation}\label{eq: k-expression}
			k(\beta)=A-B\cos(\beta-\theta_0),\quad k'(\beta)=B\sin(\beta-\theta_0).
		\end{equation}
		Consequently, we can deduce the relations
		\begin{equation}\label{eq: kalpha-and-k'beta}
			\begin{aligned}
				&k(\alpha)=A\big(1-\cos(\alpha-\beta)\big)+\cos(\alpha-\beta)k(\beta)+\sin(\alpha-\beta)k'(\beta),\\
				&k'(\beta)^2=-k(\beta)^2+2Ak(\beta)-1.
			\end{aligned}
		\end{equation} 
		Insert the expression \eqref{eq: kalpha-and-k'beta} into \eqref{eq: SecondExpression1} and divide both sides by $k(\beta)$. After rearranging the terms with respect to $k(\beta)$ and $k'(\beta)$, we obtain 
		\begin{equation}\label{eq: ThirdExpression1}
			\begin{aligned}
				V_1 k(\beta)+V_2k'(\beta)+V_3 &:= \left[f'(\alpha)f''(\beta)\sin(\alpha-\beta)-f'(\alpha)f'(\beta)+f'(\beta)^2\right]k(\beta)\\
				&\quad +\left[f'(\alpha)f''(\beta)\big(1-\cos(\alpha-\beta)\big)\right]k'(\beta)\\
				&\quad +\big[Af'(\alpha)f'(\beta)+f'(\alpha)f'(\beta)^2\big]\big(1-\cos(\alpha-\beta)\big)\\
				&\geq 0.
			\end{aligned}
		\end{equation}
		Substituting the expression \eqref{eq: k-expression} of $k(\beta)$ and $k'(\beta)$ into \eqref{eq: ThirdExpression1}, the inequality transforms into
		\begin{equation}\label{eq: all-theta0}
			AV_1+V_3- B V_1 \cos(\beta-\theta_0) +B V_2 \sin(\beta-\theta_0) \geq 0.
		\end{equation}
		Since $P_{1,\Gamma}$ is convex for any Möbius transformation $\Gamma$ of $\mathbb{D}$, the inequality \eqref{eq: all-theta0} holds for all possible choices of $\theta_0\in[0,2\pi)$. This implies that
		\begin{equation}\label{eq: after-all-theta0}
		    AV_1+V_3 \geq B\sqrt{V_1^2+V_2^2}.
		\end{equation}
		Dividing both sides by $A$ and substituting the expression of $V_3$ in \eqref{eq: after-all-theta0}, we obtain
		$$
		V_1+f'(\alpha)f'(\beta)\big(1-\cos(\alpha-\beta)\big)+\frac{1}{A}f'(\alpha)f'(\beta)^2\big(1-\cos(\alpha-\beta)\big) \geq \frac{B}{A}\sqrt{V_1^2+V_2^2}.
		$$
		Note that the terms $V_1$ and $V_2$ are independent of $r_0$. Letting $r_0\to 1$, we observe that $\frac{B}{A}\to 1$ and $\frac{1}{A}\to 0$. Thus, passing to the limit yields
		$$
		V_1+f'(\alpha)f'(\beta)\big(1-\cos(\alpha-\beta)\big) \geq \sqrt{V_1^2+V_2^2}.
		$$
		Finally, substituting the explicit expressions for $V_1$ and $V_2$ back into the above inequality and carefully rearranging the terms, we arrive at the following convexity inequality for the curves $L_{1,\Gamma}$:
		\begin{equation}\label{eq: FinalExpression1}
			\frac{f'(\beta)}{f'(\alpha)} \geq \left[ \frac{f''(\beta)}{f'(\beta)} \sin\left(\frac{\alpha-\beta}{2}\right) - \cos\left(\frac{\alpha-\beta}{2}\right) \right]^2,\quad\text{ for a.e. } \alpha, \beta \in [0,2\pi).
		\end{equation}
		
		\textbf{Step 2. The convexity inequality for $\Phi(x,\Gamma\circ F(x))$.}
		
		Let $\Gamma$ be an orientation-preserving isometry of $\mathbb{H}^3_P$ mapping $\mathbb{S}^2_{-}$ onto itself. Consider the Pogorelov map $\Phi(x,\Gamma\circ F(x))$. Define 
    	$$
    	S_{1}^\Gamma = \{P_1(x, \Gamma\circ F(x)) : x \in \partial X_1\}, \quad S_{2}^\Gamma = \{P_2(x, \Gamma\circ F(x)) : x \in \partial X_1\}.
    	$$
    	By an argument similar to that in \cref{sec: pogorelov}, we can extend each $S_{i}^\Gamma$ to a locally convex surface $\tilde{S}_{i}^\Gamma=\partial K_{i}^\Gamma\cap \mathbb{R}^3_{+}$, where $K_{i}^\Gamma$ is a compact convex body in $\overline{\mathbb{R}^3_{+}}$. Let $P_{i}^\Gamma=K_{i}^\Gamma\cap\mathbb{R}^2$. Then $P_{i}^\Gamma$ is a compact convex set in $\mathbb{R}^2$.
        Let 
        $$
        L_1^{\Gamma}(\alpha)=r_1^{\Gamma}(\alpha)e^{i\alpha},\quad \alpha\in[0,2\pi),
        $$
        be the polar parametrization of $\partial P_{1}^\Gamma$. It is clear that 
        $$
        L_2^{\Gamma}(\alpha):=(2-r_1^{\Gamma}(\alpha)) e^{i\gamma(f(\alpha))},\quad\alpha\in[0,2\pi),
        $$
        is a parametrization of $\partial P_{2}^\Gamma$, where $\gamma$ is the restriction of $\Gamma$ to the unit circle $\mathbb{S}^1$.
		Let $g(\alpha) = \gamma(f(\alpha))$ for all $\alpha \in [0,2\pi)$. Then $g(\alpha)$ also satisfies the inequality \eqref{eq: FinalExpression1}. That is,
		\begin{equation}\label{eq: Expression-for-g}
			\frac{g'(\beta)}{g'(\alpha)} \geq \left[ \frac{g''(\beta)}{g'(\beta)} \sin\left(\frac{\alpha-\beta}{2}\right) - \cos\left(\frac{\alpha-\beta}{2}\right) \right]^2\quad\text{ for a.e. } \alpha, \beta \in [0,2\pi).
		\end{equation}
		By the chain rule, we compute that
		$$
		g'(\beta) = \gamma'(f(\beta)) f'(\beta),\quad g''(\beta) = \gamma''(f(\beta)) [f'(\beta)]^2 + \gamma'(f(\beta)) f''(\beta).
		$$
		Substituting the expressions for $g'$ and $g''$ into the inequality \eqref{eq: Expression-for-g}, we obtain
		\begin{equation}\label{eq: FirstExpGamma}
			\frac{\gamma'(f(\beta)) f'(\beta)}{\gamma'(f(\alpha)) f'(\alpha)} \geq \left[ \left( \frac{\gamma''(f(\beta))}{\gamma'(f(\beta))} f'(\beta) + \frac{f''(\beta)}{f'(\beta)} \right) \sin\left(\frac{\alpha-\beta}{2}\right) - \cos\left(\frac{\alpha-\beta}{2}\right) \right]^2.
		\end{equation}
		
		Let $\Delta_1=\alpha-\beta$ and $\Delta_2=f(\alpha)-f(\beta)$. Define 
		$$
		Q = \frac{f''(\beta)}{f'(\beta)} \sin \left(\frac{\Delta_1}{2}\right) - \cos \left(\frac{\Delta_1}{2}\right),\quad R = f'(\beta) \sin \left(\frac{\Delta_1}{2}\right).
		$$
		Inserting the expression $\gamma' = 1/k$ and $\gamma''= -k'/k^2$ into \eqref{eq: FirstExpGamma} yields
		$$
		k(f(\alpha)) k(f(\beta)) f'(\beta) \ge f'(\alpha) \left[ k(f(\beta)) Q - k'(f(\beta)) R \right]^2.
		$$
		Utilizing the relations in \eqref{eq: kalpha-and-k'beta} once more, we can expand the above inequality and obtain that
		\begin{equation}\label{eq: SecondExpGamma}
			U_2 \cdot k(f(\beta))^2 + U_1 \cdot k(f(\beta))k'(f(\beta)) + A U_0 \cdot  k(f(\beta)) + f'(\alpha)R^2 \ge 0,
		\end{equation}
		where the coefficients
		\begin{equation}\label{eq: UiExpression}
			\begin{aligned}
				U_2 &:= f'(\beta)\cos (\Delta_2) - f'(\alpha)(Q^2 - R^2),\\
				U_1 &:= f'(\beta)\sin (\Delta_2) + 2f'(\alpha)QR,\\
				U_0 &:= f'(\beta)\big(1-\cos(\Delta_2)\big) - 2f'(\alpha)R^2
			\end{aligned}
		\end{equation}
        depend on $f$, $\alpha$, and $\beta$, but not on the chosen Möbius transformation. Let $\phi=f(\beta)-\theta_0$. Observe that
		$$
		\lim_{r_0\to 1^-}\frac{k(f(\beta))}{A}=1-\cos\phi, \quad \lim_{r_0\to 1^-}\frac{k'(f(\beta))}{A}= \sin\phi.
		$$
		Dividing both sides of \eqref{eq: SecondExpGamma} by $A^2(1-\cos \phi)$ when $\phi\neq0$ and letting $r_0\to 1^-$, we deduce that
		$$
		U_2(1-\cos \phi)+U_1\sin\phi+U_0 \geq 0.
		$$
		Since $\Gamma$ is arbitrary, this inequality holds for almost all $\theta_0\in[0,2\pi)$ and consequently for almost all $\phi\in[0,2\pi)$. This implies the inequality
		$$
		U_2+U_0 \geq \sqrt{U_1^2+U_2^2}.
		$$
		Substitute the expressions of $U_i$ for $i\in\{0,1,2\}$ in \eqref{eq: UiExpression} into the above inequality. After suitable simplification, we obtain
		$$
		-4 f'(\alpha) f'(\beta) \left[ Q \sin\left(\frac{\Delta_2}{2}\right) + R \cos\left(\frac{\Delta_2}{2}\right) \right]^2 \ge 0,\quad \text{ for a.e. } \alpha, \beta \in [0,2\pi).
		$$
		Since the left-hand side of this inequality is nonpositive, it must vanish for almost every pair $(\alpha,\beta)$. Substituting the definitions of $Q$ and $R$ back into this identity yields the equality
		\begin{equation}\label{eq: ThirdExpGamma}
			\begin{aligned}
				&\left( \frac{f''(\beta)}{f'(\beta)}\sin\left(\frac{\alpha-\beta}{2}\right) - \cos\left(\frac{\alpha-\beta}{2}\right) \right) \sin\left(\frac{f(\alpha)-f(\beta)}{2}\right) \\
				= &- f'(\beta)\sin\left(\frac{\alpha-\beta}{2}\right) \cos\left(\frac{f(\alpha)-f(\beta)}{2}\right),\quad \text{ for a.e. } \alpha, \beta \in [0,2\pi).
			\end{aligned}
		\end{equation}
		Divide both sides of \eqref{eq: ThirdExpGamma} by $\sin\left(\frac{\alpha-\beta}{2}\right) \sin\left(\frac{f(\alpha)-f(\beta)}{2}\right)$. It follows that
		$$
		\frac{f''(\beta)}{f'(\beta)}  = \cot\left(\frac{\alpha-\beta}{2}\right) - f'(\beta) \cot\left(\frac{f(\alpha)-f(\beta)}{2}\right), \quad \text{ for a.e. } \alpha, \beta \in [0,2\pi).
		$$
		Notice that the left-hand side is independent of $\alpha$. Therefore, differentiating the right-hand side with respect to $\alpha$ yields
		$$
		-\frac{1}{2} \frac{1}{\sin^2\left(\frac{\alpha-\beta}{2}\right)} + \frac{1}{2} \frac{f'(\alpha) f'(\beta)}{\sin^2\left(\frac{f(\alpha)-f(\beta)}{2}\right)} = 0,\quad \text{ for a.e. } \alpha, \beta \in [0,2\pi),
		$$
		which can be reformulated as
		$$
		\sin^2\left(\frac{f(\alpha)-f(\beta)}{2}\right) = f'(\alpha) f'(\beta) \sin^2\left(\frac{\alpha-\beta}{2}\right), \quad \text{ for a.e. } \alpha, \beta \in [0,2\pi).
		$$
		Since both sides of the above equation are continuous functions, it follows that the equality holds for all $\alpha,\beta\in[0,2\pi)$. By \cref{lem: Möbius-Characterization}, $f$ is the restriction of some Möbius transformation $\Gamma_0$ to the unit circle, which completes the proof.
	\end{proof}
	\begin{lemma}\label{lem: Möbius-Characterization}
		Suppose that $f \in \operatorname{Diff}^1(\mathbb{S}^1)$ is orientation-preserving and satisfies the equation
		\begin{equation}\label{eq: Möbius-Characterization}
			\sin^2\left(\frac{f(\alpha)-f(\beta)}{2}\right)= {f'(\alpha) f'(\beta)} \sin^2\left(\frac{\alpha-\beta}{2}\right), \quad\forall\alpha,\beta\in[0,2\pi).
		\end{equation} 
        Then $f$ is the restriction to the unit circle of a Möbius transformation $\Gamma$ on the unit disk.
	\end{lemma}
	\begin{remark}
    	The conclusion of the preceding lemma is a classical result. Indeed, the density $\frac{|d\alpha\,d\beta|}{\sin^2((\alpha-\beta)/2)}$ is, up to a constant factor, the Liouville measure on the space of unoriented geodesics in $\mathbb{H}^2$. Equation \eqref{eq: Möbius-Characterization} states precisely that $f$ preserves this measure. Since the Liouville measure recovers the cross-ratio, it follows that $f$ preserves the cross-ratio and, consequently, is the restriction to the unit circle of a Möbius transformation of the unit disk. For the sake of completeness, we provide a direct proof of this lemma here.
    \end{remark}
	\begin{proof}
		Consider $\mathbb{S}^1$ as the unit circle of the complex plane. For any two points $z = e^{i\alpha}$ and $w = e^{i\beta}$ on $\mathbb{S}^1$, we have
		$$
		|z - w|^2 = 4\sin^2\left(\frac{\alpha-\beta}{2}\right).
		$$
		Define the homeomorphism $F: \mathbb{S}^1 \to \mathbb{S}^1$ by $F(e^{i\theta}) = e^{if(\theta)}$. Then the equation \eqref{eq: Möbius-Characterization} can be rewritten as
		\begin{equation}\label{eq: Möbius-Characterization1}
			|F(z) - F(w)|^2 = f'(\alpha)f'(\beta) |z - w|^2.
		\end{equation}
		For any four distinct points $z_1, z_2, z_3, z_4 \in \mathbb{S}^1$, we define the cross-ratio $[z_1, z_2, z_3, z_4]=\frac{(z_1 - z_3) (z_2 - z_4)}{(z_1 - z_2) (z_3 - z_4)}$. By \eqref{eq: Möbius-Characterization1}, we have
		\begin{equation}\label{eq: Möbius-Characterization2}
			\begin{aligned}
				\Big|[F(z_1),F(z_2),F(z_3),F(z_4)]\Big|^2&=\frac{|F(z_1) - F(z_3)|^2 |F(z_2) - F(z_4)|^2}{|F(z_1) - F(z_2)|^2 |F(z_3) - F(z_4)|^2}\\
				&= \frac{|z_1 - z_3|^2 |z_2 - z_4|^2}{|z_1 - z_2|^2 |z_3 - z_4|^2}=\big|[z_1, z_2,z_3, z_4]\big|^2
			\end{aligned}
		\end{equation}
		This implies that the map $F$ preserves the absolute value of the cross-ratio for any four points on the unit circle. Let $w_1=F^{-1}(1)$, $w_2=F^{-1}(i)$, and $w_3=F^{-1}(-1)$. Then, there exists a Möbius transformation $\Gamma$ on $\mathbb{D}$ such that $\Gamma(w_i)=F(w_i)$ for each $i\in\{1,2,3\}$. Since $\Gamma$ preserves the cross-ratio, for each $w\in\mathbb{S}^1\setminus\{w_1,w_2,w_3\}$, by \eqref{eq: Möbius-Characterization2} we have
		$$
		\begin{aligned}
			\Big|[F(w_i),F(w_j),F(w_k),F(w)]\Big|&=\big|[w_i, w_j,w_k, w]\big|\\
			&=\Big|[\Gamma(w_i),\Gamma(w_j),\Gamma(w_k),\Gamma(w)]\Big|\\
			&=\Big|[F(w_i),F(w_j),F(w_k),\Gamma(w)]\Big|,
		\end{aligned}  
		$$
		where $(i,j,k)$ is a permutation of $(1,2,3)$. By our choices of $F(w_l)$ for $l\in\{1,2,3\}$, it then follows that
		$$
		\frac{|F(w)-1|}{|F(w)+1|}=\frac{|\Gamma(w)-1|}{|\Gamma(w)+1|}\quad\text{and}\quad \frac{|F(w)-1|}{|F(w)-i|}=\frac{|\Gamma(w)-1|}{|\Gamma(w)-i|}.
		$$
		Let $c>0$. Note that the equation $\frac{|z-1|}{|z+1|}=c$ describes an Apollonius circle with its center on the real axis when $c\neq1$ and describes the imaginary axis when $c=1$. It follows that $F(w)$ must coincide with either $\Gamma(w)$ or its complex conjugate $\overline{\Gamma(w)}$. In either case, we have $|F(w)-1|=|\Gamma(w)-1|$. Combined with the identity $\frac{|F(w)-1|}{|F(w)-i|}=\frac{|\Gamma(w)-1|}{|\Gamma(w)-i|}$, we obtain that $|F(w)-i|=|\Gamma(w)-i|$. This implies that $F(w)=\Gamma(w)$. Thus, $F(w)=\Gamma(w)$ for all $w\in\mathbb{S}^1$, which completes the proof.
	\end{proof}

	\begin{theorem} 
		Let $X_1$ and $X_2$ be two closed, noncompact convex subsets of dimension at least $2$ in $\mathbb{H}^3_P$ such that for each $i\in\{1,2\}$, the set $X_i^P=\overline{X}_i\cap\partial \mathbb{H}^3_P$ consists of a single disk and a set of zero one-dimensional Hausdorff measure. Then any isometry between the boundaries of $X_1$ and $X_2$ extends to an isometry of $\mathbb{H}^3_P$.
	\end{theorem}
	\begin{proof}
		At the beginning of this section, we dealt with the case where at least one of $X_1$ and $X_2$ does not contain $0$ in its interior. Now, assume that both $X_1$ and $X_2$ contain $0$ in their interiors. By \cref{prop: Pi-contain-0-interior}, $P_1$ and $P_2$ both contain $0$ in their interiors. It follows from \cref{lem: L2-formula-via-L1} that there exists a continuous and strictly monotone function $f: [0, 2\pi] \to \mathbb{R}$ such that 
		\begin{equation}\label{eq: proof-section-3}
		    L_2(\alpha) = (2 - r_1(\alpha)) e^{i f(\alpha)},\quad\alpha\in[0,2\pi),
		\end{equation}
		parametrizes the curve $\partial P_2$, where $r_1(\alpha)$ is the radial function of $\partial P_1$. Furthermore, by \cref{prop: f -is-Möbius}, $f$ is the restriction to the unit circle of some Möbius transformation $\Gamma_0$ on the unit open disk $\mathbb{D}$, i.e., $\Gamma_0(e^{i\alpha})=e^{if(\alpha)}$ for all $\alpha\in[0,2\pi)$. Consider the Pogorelov map $\Phi(\Gamma_0 (x),F(x))$. By the formula \eqref{eq: L1gamma-formula},
		$$
		L_{1,\Gamma_0}(\alpha)=\frac{2f'(\alpha)}{f'(\alpha)+f'(\alpha)}e^{if(\alpha)}=e^{if(\alpha)},\quad\alpha\in[0,2\pi),
		$$
		is a parametrization of the curve $\partial P_{1,\Gamma_0}$. It follows that $P_{1,\Gamma_0}$ is exactly the unit closed disk $\overline{\mathbb{D}}$ in $\mathbb{R}^2$. By the expression \eqref{eq: proof-section-3}, we have $P_{2,\Gamma_0}=P_{1,\Gamma_0}=\overline{\mathbb{D}}$. Consequently, the isometry $G_{\Gamma_0}$ extends naturally to an isometry $\hat{G}_{\Gamma_0}: \partial K_{1,\Gamma_0} \to \partial K_{2,\Gamma_0}$. It is notable that $\hat{G}_{\Gamma_0}|_{\mathbb{S}^1}=\operatorname{id}_{\mathbb{S}^1}$. By the Pogorelov theorem, $\hat{G}_{\Gamma_0}$ extends to an isometry $T\in\operatorname{Isom}(\mathbb{R}^3)$. It then follows from \cref{prop: RnIsometry-to-HnIsometry} that there exists an isometry $A\in\operatorname{Isom}(\mathbb{H}^3_P)$ such that $F(x)=A\circ\Gamma_0(x)$ for all $x\in \partial X_1$. This completes the proof.
	\end{proof}

\section{Proof of the main theorems}\label{sec: the-general-case}

\subsection{Half-space and lower-dimensional cases}\label{subsec:special-cases}
In this subsection, we consider the case where at least one of $X_1$ and $X_2$ is a hyperbolic half-space or is two-dimensional. We first show that the intrinsic isometry between $\partial X_1$ and $\partial X_2$ extends to a homeomorphism between their respective end spaces.

 \begin{lemma}\label{lem:component-compactification}
    Let $U\subset\mathbb S^2$ be a connected open subset and let $\mathcal C(U)$ denote
    the collection of connected components of $\mathbb S^2\setminus U$.
    Consider the equivalence relation $\sim_U$ on $\mathbb{S}^2$ given by
    $$
    x\sim_U y
    \quad\Longleftrightarrow\quad
    x=y
    \ \text{or}\
    x,y\in D
    \ \text{for some }D\in\mathcal C(U),
    $$
    and denote the quotient space by $\widehat U:=\mathbb S^2/\!\sim_U$.
    Then $\widehat U$ is homeomorphic to the sphere $\mathbb S^2$. Moreover, every homeomorphism $f:U_1\to U_2$ admits a unique
    extension to a homeomorphism $\widehat f:\widehat U_1\longrightarrow\widehat U_2$.
    \end{lemma}
    
    \begin{proof}
    This is a direct consequence of Moore's decomposition theorem and
    the canonical boundary-component correspondence under homeomorphisms
    of planar domains; see, e.g., \cite{Moore1925}
    and \cite[Section~3]{NtalampekosYounsi2020} for reference.
    \end{proof}

The main ingredient of this subsection is the following proposition concerning the invariance of disk components under an intrinsic isometry.

\begin{proposition}\label{prop:point-end-extremal}
    Let $X,Y\subset\mathbb H_P^3$ be closed noncompact convex sets with
    $\dim X\ge2$ and $\dim Y=3$. Assume that
    $$
    \mathscr{H}^1(X^P)=0
    \qquad\text{and}\qquad
    F:\partial X\longrightarrow\partial Y
    $$
    is an intrinsic isometry. Then every connected component of $Y^P$ is a
    point.
\end{proposition}

\begin{proof}
Suppose that $D$ is a nondegenerate component of $Y^P$. 
Since $\mathscr{H}^1(X^P)=0$, every component of $X^P$ is a point. By \cref{lem:component-compactification}, the component $D$ corresponds to a unique point $p\in X^P$.

We recall the basic notions of extremal length. Let $(S,g)$ be a Riemann surface  and $\rho:S\to\mathbb{R}_{\geq 0}$ be a Borel measurable function. Then the conformal metric $d=\rho g$ is the metric such that the length of a rectifiable curve $\gamma$ is defined to be $l_d(\gamma)=\int_\gamma\rho\,ds_g$, where $ds_g$ is the length element of the metric $g$.
The area of the metric $d=\rho g$ is defined to be $A(d)=\int_S \rho^2\, d{A_g}$, where $dA_g$ is the area element of the Riemannian metric $g$. Let $\Gamma$ be a family of rectifiable curves in $(S,g)$, its length in $d$ is defined to be $l_d(\Gamma) = \inf\left\{l_d(\gamma) : \gamma\in \Gamma\right\}$.
Then the extremal length $\operatorname{EL}_S(\Gamma)$ of $\Gamma$ is defined as 
\begin{equation*}
    \operatorname{EL}_S(\Gamma) = \sup\left\{\frac{l_d^2(\Gamma)}{A(d)} : d \text{ is a finite area conformal metric on } X\right\}.
\end{equation*}
It is clear that the extremal length depends only on the conformal structure of $(S,g)$.
Now let $(S, d_S)$ be a convex surface in $\mathbb{R}^3$ or $\mathbb{H}^3$ equipped with its induced path metric. The work of Reshetnyak shows that there exists a smooth Riemannian metric $g$ and a nonnegative Borel measurable function $\rho$ such that $d_S$ is induced by $\rho g$. Thus, $(S, d_S)$ admits a unique conformal structure, which allows one to define the extremal length of curve families on $(S, d_S)$; see, e.g., \cite[Section 2.3]{LuoWu} for further details on conformal structures on convex surfaces.

Fix a closed topological disk $A \subset \partial X$, and let $\Gamma$ be the family of rectifiable Jordan curves in $\partial X$ separating $p$ from $A$. We claim that
\begin{equation}\label{eq:point-end-extremal-zero-positive}\operatorname{EL}_{\partial X}(\Gamma)=0,
\qquad
\operatorname{EL}_{\partial Y}(F(\Gamma))>0,\end{equation}
which contradicts the invariance of extremal length under the intrinsic isometry $F$. 
Instead of the induced hyperbolic metric, we will primarily employ the induced Euclidean metric in the proof. Every convex surface $S \subset \mathbb{H}^3$ carries an intrinsic path metric $d_S$ induced by the ambient hyperbolic metric. Under the Poincaré ball model $\mathbb{H}_P^3 \subset \mathbb{R}^3$ of $\mathbb{H}^3$, the surface $S$ can also be viewed as a surface in the unit ball $\mathbb{B}^3$ in $\mathbb{R}^3$; we denote by $d^{\mathbb{E}}_S$ its path metric induced from $\mathbb{R}^3$. Since the metrics $d_S$ and $d^{\mathbb{E}}_S$ are conformal, they determine the same extremal length; see, e.g., \cite[Appendix A]{LuoWu} for reference. We denote by $ds_S^{\mathbb{E}}$ and $dA_S^{\mathbb{E}}$ the line and area elements of $d^{\mathbb{E}}_S$, respectively.


(1). $\operatorname{EL}_{\partial X}(\Gamma)=0$. By
\cite[Theorem~4.1]{LuoWu}, for every hyperbolic convex surface $\partial Z$, its Euclidean area in $\mathbb{B}^3$ is at most $16\pi$. We first show that there is a universal $C>0$ such that for
every $0<r\le1/2$, the Euclidean area
\begin{equation}\label{eq:quadratic-area-point-end}
\operatorname{Area}_{\mathbb{E}}\bigl(\partial Z\cap B_r(p)\bigr)\le Cr^2,
\end{equation}
where $B_r(p)$ is the ball in $\mathbb{R}^3$ centered at $p$ with radius $r$. Let $a=(1-r)p$ and let $\Gamma_a$ be the M\"obius transformation of $\mathbb{B}^3$ defined in \eqref{eq:Ball-transformation-a-to-0} sending $a$ to $0$. The Euclidean conformal factor of $\Gamma_a$ is
$$
 \sigma_a(x)
 =\frac{1-\|a\|^2}{1-2a\cdot x+\|a\|^2\|x\|^2}
 =\frac{2r-r^2}{\|p-(1-r)x\|^2}.
$$
If $\|x-p\|<r$, then
$$
 \|p-(1-r)x\|\leq r+(1-r)\|x-p\|\leq2r,
$$
which implies that $\sigma_a(x)\geq\frac1{4r}$. Then, applying the Euclidean area upper bound to $\partial\Gamma_a( Z)$ gives
$$
    \frac1{16r^2}\operatorname{Area}_{\mathbb{E}}\bigl(\partial Z\cap B_r(p)\bigr)
    \leq\operatorname{Area}_{\mathbb{E}}(\partial\Gamma_a(Z))\leq16\pi,
$$
which proves \eqref{eq:quadratic-area-point-end}. The same estimate holds in the two-dimensional case, with area counted on both copies of the double.

Let $u:\partial X\cup X^P\to\mathbb{R}_{\geq 0}$ be the function defined by $u(x)=\|x-p\|$ for every $x\in\partial X\cup X^P$. Set
$$
U_t:=\{x\in\partial X:u(x)<t\},
$$
and choose $r_0>0$ with $2r_0<\min\{\operatorname{dist}_{\mathbb{E}}(p,A),1/2\}$. Since the distance map is Lipschitz, we have $\mathscr{H}^1(u(X^P))=0$. Hence, for almost every $0<t<2r_0$, the level set $u^{-1}(t)$ has empty intersection with $X^P$ and $u^{-1}(t)\subset \partial X$ separates $p$ and $A$.
Moreover, since $u(x)$ is $1$-Lipschitz, the coarea formula for rectifiable sets \cite[Theorem~3.2.22]{Federer} yields
$$
    \int_0^{2r_0} \mathscr{H}^1(u^{-1}(t)) \,dt\leq\int_{U_{2r_0}} dA_{\partial X}^{\mathbb{E}} =\operatorname{Area}_{\mathbb{E}}(U_{2r_0}),
$$
which implies that $\mathscr{H}^1(u^{-1}(t))<\infty$ for almost every $0<t<2r_0$. It then follows from \cite[Corollary~7.5]{LytchakSeparate} that for almost every $0<t<2r_0$, $u^{-1}(t)$ contains a rectifiable Jordan curve separating $p$ and $A$. Therefore, for almost every $0<t<2r_0$ and every conformal metric $d_{\rho}:=\rho d_{\partial X}^{\mathbb{E}}$ of finite positive area, we have $l_{d_\rho}(\Gamma)\le\int_{u^{-1}(t)}\rho\,ds_{\partial X}^{\mathbb{E}}$. Applying the coarea formula, the Cauchy--Schwarz inequality, and \eqref{eq:quadratic-area-point-end}, we obtain for every $r \in (0, r_0)$,
\begin{align*}
r\,l_{d_\rho}(\Gamma)
\le
\int_r^{2r}\int_{u^{-1}(t)}\rho\,ds_{\partial X}^{\mathbb{E}}\,dt
&\le
\int_{U_{2r}}\rho\,dA_{\partial X}^{\mathbb{E}}\\
&\le
\operatorname{Area}_{\mathbb{E}}(U_{2r})^{1/2}
\left(\int_{U_{2r}}\rho^2\,dA_{\partial X}^{\mathbb{E}}\right)^{1/2}\\
&\le
2\sqrt C\,r
\left(\int_{U_{2r}}\rho^2\,dA_{\partial X}^{\mathbb{E}}\right)^{1/2}.
\end{align*}
Dividing both sides by $r$ and letting $r \downarrow 0$, the integral on the right-hand side vanishes since $U_{2r}\downarrow\varnothing$ in $\partial X$. Hence $l_{d_\rho}(\Gamma)=0$ for every finite area conformal metric $d_\rho$, which implies that $\operatorname{EL}_{\partial X}(\Gamma)=0$.

(2). $\operatorname{EL}_{\partial Y}(F(\Gamma))>0$. Since $\operatorname{dim} Y=3$, after applying a hyperbolic isometry to $Y$ and postcomposing $F$ with that isometry, we may assume that $0\in\operatorname{int}Y$. Let
$$
R:\partial Y\longrightarrow\mathbb S^2\setminus Y^P,
\qquad
R(x)=\frac{x}{\|x\|},
$$
be the radial projection, which is a homeomorphism. For every $\gamma\in\Gamma$, since $\gamma$ separates $p$ and $A$, the Jordan curve
$R(F(\gamma))$ separates
$D$ and $R(F(A))$. Because $D$ and $R(F(A))$ both have positive diameter, the spherical length of $R(F(\gamma))$ is uniformly bounded below by some constant $c>0$. Moreover, since $0\in\operatorname{int}Y$ implies that the radial projection $R$ is $L$-Lipschitz for some $L>0$, the Euclidean length of $F(\gamma)$ has a uniform lower bound $L^{-1}c>0$ for each $\gamma\in\Gamma$. Taking the induced Euclidean metric $d_{\partial Y}^{\mathbb{E}}$ as a test metric in the definition of extremal length and using the bound $\operatorname{Area}_{\mathbb{E}}(\partial Y)\le 16\pi$, we obtain
$$
\operatorname{EL}_{\partial Y}(F(\Gamma))
\geq \frac{l_{d_{\partial Y}^{\mathbb{E}}}^2(F(\Gamma))}{\operatorname{Area}_{\mathbb{E}}(\partial Y)}\geq
\frac{L^{-2}c^2}{16\pi}
>0.
$$
This completes the proof.
\end{proof}

Now assume that at least one of $X_1$ and $X_2$ is a hyperbolic half-space or is two-dimensional. Suppose first that both $X_1$ and $X_2$ are three-dimensional and $X_2$ is a hyperbolic half-space. Since $\partial X_2\cong\mathbb H^2$ has only one end, $X_1^P$ has exactly one
component and hence is either a round disk or a point. The case where $X_1^P$ is a disk is treated in \cref{sec: single-disk}. If $X_1^P$ were a point, \cref{prop:point-end-extremal} would force every component of $X_2^P$ to be a point, contradicting the fact that $X_2^P$ is a round disk.
It remains to consider the case where $X_1$ is two-dimensional. Since $X_1^P$ lies in the ideal circle of
a totally geodesic plane, it has no disk component and therefore
$\mathscr{H}^1(X_1^P)=0$. If $\dim X_2=2$, the same argument gives
$\mathscr{H}^1(X_2^P)=0$, and \cref{thm: LuoLuoRao} applies. Suppose instead
that $\dim X_2=3$. Since $\partial X_1$ is intrinsically isometric to
$\partial X_2$, \cref{prop:point-end-extremal} implies that $X_2^P$ has no
disk components and hence $\mathscr{H}^1(X_2^P)=0$. This mixed-dimensional case is then ruled out by \cref{thm: LuoLuoRao}. Therefore, throughout the remainder of this section, we may assume that both $X_1$ and $X_2$ are three-dimensional and that neither of them is a hyperbolic half-space.

\subsection{The general case}
	We first consider the case where each set $X_i^P$ is a circle-type closed set $\mathscr{D}_i$ with countably many components.	Let $\{D_{i,j}\}_{{j\in J_i}}$ be the set of components of $\mathscr{D}_i$, $i\in\{1,2\}$. Then each $D_{i,j}$ is either a closed disk or a single point. The index set $J_i$ can be naturally identified with the quotient space obtained by collapsing each component $D_{i,j}$ in $\mathscr{D}_i$ to a single point, i.e., each $j\in J_i$ corresponds to the equivalence class $[D_{i,j}]$ in the quotient space. It is clear that $J_i$ is a compact Hausdorff space. Since $\mathscr{D}_i$ has countably many components, $J_i$ is a countable set. A classical result states that the continuous image of a compact metric space in a Hausdorff space is metrizable (see, e.g., \cite[Corollary 23.2]{MR264581} for a proof). Thus, the quotient space $J_i$ is completely metrizable, i.e., it admits a compatible metric $d$ such that $(J_i,d)$ is a complete metric space. Our arguments in this section employ ordinals and transfinite recursion. For a comprehensive background on these topics, we refer the reader to \cite{JechSetTheory}.
	
	\begin{definition}
		For any topological space $X$, we call
		$$
		X' = \{x \in X : x \text{ is a limit point of } X\}
		$$ 
		the \emph{Cantor--Bendixson derivative} of $X$. Clearly, $X'$ is closed and $X$ is perfect if and only if $X = X'$. Using transfinite recursion, for all $\lambda \in \text{ORD}$, we define the \emph{iterated Cantor--Bendixson derivative} $X^\lambda$ as follows:
		\begin{enumerate}
			\item $X^0 = X$,
			\item $X^{\lambda+1} = (X^\lambda)'$,
			\item $X^\lambda = \bigcap_{\xi < \lambda} X^\xi$, if $\lambda$ is a limit ordinal.
		\end{enumerate}
		Thus, $\{X^\lambda\}_{\lambda \in \text{ORD}}$ is a decreasing transfinite sequence of closed subsets of $X$. If there exists an ordinal $\lambda_0$ such that $X^\lambda = X^{\lambda_0}$ for all $\lambda \ge \lambda_0$, we call the smallest ordinal with this property the \emph{Cantor--Bendixson rank} of $X$.
	\end{definition}
    The following theorem is well known; see \cite[Section 6]{MR1321597} for a proof.
	\begin{theorem}\label{thm: Cantor-Bendixson derivative}
		Suppose that $X$ is a separable completely metrizable space. Then, $X$ is countable if and only if $X^{\lambda_0}=\varnothing$ for some countable ordinal $\lambda_0$.
	\end{theorem}
	
	For each $\lambda \in \text{ORD}$, let $J_i^{\lambda}$ be the iterated Cantor--Bendixson derivative of the quotient space $J_i$. Since $J_i$ is countable and completely metrizable, by \cref{thm: Cantor-Bendixson derivative}, there exists a countable ordinal $\lambda_0$ such that $J_i^{\lambda_0}=\varnothing$ for each $i\in\{1,2\}$. Let $\mathscr{D}_i^{\lambda}=\bigcup_{j\in J_i^{\lambda}} D_{i,j}$. Clearly, $\mathscr{D}_i^{\lambda}$ is a compact subset of $\mathbb{S}^2$, and $\mathscr{D}_i^{\lambda}\setminus\mathscr{D}_i^{\lambda+1}$ is exactly the union of all isolated disks and isolated points in $\mathscr{D}_i^{\lambda}$.
	Denote by $\mathscr{C}_i^{\lambda}$ the complement of $\mathscr{D}_i^{\lambda}$ in $\mathbb{S}^2$ for all $\lambda\in\text{ORD}$ and $i\in\{1,2\}$. Then, $\mathscr{C}_i^{\lambda}$ is a circle domain in $\mathbb{S}^2$, and the radial projection $R:{S}_i\to\mathbb{S}^2$ is a homeomorphism from ${S}_i$ onto $\mathscr{C}_i^{0}$. 

    Recall from \cref{subsec:special-cases} that both $X_1$ and $X_2$ are assumed to be three-dimensional and that neither is a hyperbolic half-space. After applying suitable hyperbolic isometries to $X_1$ and $X_2$ separately, we may assume that $0\in\operatorname{int}X_i$ for $i=1,2$. For a disk component $D_{i,j}$, let $H_{i,j}=C_P(D_{i,j})\cap\mathbb H_P^3$ be the closed hyperbolic half-space corresponding to $D_{i,j}$. For every $0\leq\lambda\leq\lambda_0$, let 
    $$
    \mathscr{G}_i^{\lambda}:=\{j\in J_i\setminus J_i^\lambda: D_{i,j}\text{ is a disk component}\}.
    $$ 
    We call a pair of hyperbolic isometries $(\Gamma_1,\Gamma_2)$ \emph{$\lambda$-admissible} if
    \begin{equation}\label{eq: stage-admissible}
       \Gamma_i^{-1}(0)\in\operatorname{int}X_i\setminus
       \bigcup_{j\in\mathscr{G}_i^{\lambda}}H_{i,j},
       \qquad i=1,2.
    \end{equation}
    In other words, $0$-admissibility requires $0\in\operatorname{int} \Gamma_i(X_i)$ and $\lambda$-admissibility requires additionally that every closed disk $\Gamma_i(D_{i,j})$, $j\in\mathscr{G}_i^{\lambda}$, is contained in an open hemisphere. We claim that such admissible pairs exist for $\lambda=\lambda_0$ and therefore for all $0\leq\lambda\leq\lambda_0$. First note that the half-spaces $H_{i,j}\subset X_i$ are pairwise disjoint and locally finite. Indeed, for any compact subset $K\subset X_i$, there exists $c_K>0$ such that for every half-space $H$ having nonempty intersection with $K$, the radius of the disk $\overline{H}\cap\partial \mathbb{H}^3_P$ is at least $c_K$. Since the half-spaces $H_{i,j}$ are pairwise disjoint, it follows that there are only finitely many half-spaces $H_{i,j}$ having nonempty intersection with $K$. Suppose, for contradiction, that 
    $$
    \operatorname{int} X_i=\bigsqcup_{j \in \mathscr{G}_{i}^{\lambda_0}}  (H_{i,j}\cap\operatorname{int} X_i).
    $$
    Then since $H_{i,j}$ are closed and
    $\{H_{i,j}\cap\operatorname{int} X_i\}_j$ are locally finite, we see that 
    $$
    \operatorname{int} X_i=(H_{i,j_0}\cap\operatorname{int} X_i)\sqcup\bigsqcup_{j \in \mathscr{G}_{i}^{\lambda_0}\setminus\{j_0\}}  (H_{i,j}\cap\operatorname{int} X_i)
    $$
    is the disjoint union of two closed subsets, which contradicts the connectivity of $\operatorname{int} X_i$.
   
	Let $F:\partial X_1\to \partial X_2$ be an intrinsic isometry. For each $0$-admissible pair $(\Gamma_1,\Gamma_2)$ of isometries of $\mathbb{H}^3_P$, consider the Pogorelov map 
	$$
	\Phi(\Gamma_1(x),\Gamma_2 \circ
	F(x))=\left(\frac{2P(\Gamma_1(x))}{-\langle \Gamma_1(x)+ \Gamma_2 \circ F(x),e\rangle},\frac{2P(\Gamma_2 \circ F(x))}{-\langle\Gamma_1(x)+ \Gamma_2 \circ F(x),e\rangle}\right).
	$$
	Define
	$$
	S_i(\Gamma_1,\Gamma_2) = \{P_i\big(\Gamma_1(x),\Gamma_2 \circ
	F(x)\big) : x \in \partial X_1\}, \quad i\in\{1,2\}.
	$$
    By \cref{thm: PogorelovII}, there exists an isometry ${G}_{\Gamma_1,\Gamma_2}$ from $S_1(\Gamma_1,\Gamma_2)$ to $S_2(\Gamma_1,\Gamma_2)$ with respect to their intrinsic path metrics. Since $\Gamma_i$ is an isometry of $\mathbb{H}^3_P$, it induces a Möbius transformation of $\mathbb{S}^2$. The radial projection $R:{S}_i(\Gamma_1,\Gamma_2)\to\mathbb{S}^2$ is then a homeomorphism from ${S}_i(\Gamma_1,\Gamma_2)$ onto $\Gamma_i(\mathscr{C}_i^{0})$. Clearly, we have $(\Gamma_i(\mathscr{D}_i))^{\lambda}=\Gamma_i(\mathscr{D}_i^{\lambda})$ for all $\lambda\in \text{ORD}$.
	
	Now, for each $0\leq\lambda\leq\lambda_0$ and each $\lambda$-admissible pair $(\Gamma_1,\Gamma_2)$ of isometries of $\mathbb{H}^3_P$, we want to construct two locally convex surfaces ${S}_1^{\lambda}(\Gamma_1,\Gamma_2)$ and ${S}_2^{\lambda}(\Gamma_1,\Gamma_2)$ with respect to $0$ and an isometry ${G}^{\lambda}_{\Gamma_1,\Gamma_2}:{S}_1^{\lambda}(\Gamma_1,\Gamma_2)\to{S}_2^{\lambda}(\Gamma_1,\Gamma_2)$ with respect to their intrinsic metrics satisfying the following collection of properties:
	\begin{enumerate}
		\item[\condition{P1}{$\mathrm{P}_1$}] ${S}_i(\Gamma_1,\Gamma_2)\subset{S}_i^{\lambda}(\Gamma_1,\Gamma_2)$ and ${G}^{\lambda}_{\Gamma_1,\Gamma_2}\big|_{{S}_1(\Gamma_1,\Gamma_2)}={G}_{\Gamma_1,\Gamma_2}$.
		\item[\condition{P2}{$\mathrm{P}_2$}] ${S}_i^{\xi}(\Gamma_1,\Gamma_2)\subset{S}_i^{\lambda}(\Gamma_1,\Gamma_2)$ for all $\xi<\lambda$, and ${G}^{\lambda}_{\Gamma_1,\Gamma_2}\big|_{{S}_1^{\xi}(\Gamma_1,\Gamma_2)}={G}^{\xi}_{\Gamma_1,\Gamma_2}$ for all $\xi<\lambda$.
		\item[\condition{P3}{$\mathrm{P}_3$}] $0\notin {S}_i^{\lambda}(\Gamma_1,\Gamma_2)$ and the radial projection $R:{S}_i^{\lambda}(\Gamma_1,\Gamma_2)\to\mathbb{S}^2$ is a homeomorphism from ${S}_i^{\lambda}(\Gamma_1,\Gamma_2)$ onto $\Gamma_i(\mathscr{C}_i^{\lambda})$.
	\end{enumerate}
    
    \begin{remark}\label{rmk:boundary-limits-under-filling}
        Fix a stage $\lambda$ of the construction below and suppress the $\lambda$-admissible pair from the notation. The filling procedure preserves paired limits at all remaining unfilled boundary components. More precisely, any limit of pairs $(x_n, G^\lambda(x_n))$ whose radial projections $R(x_n)$ approach an unfilled component can be realized by pairs $(z_n, G(z_n))$ with $z_n \in S_1$, while preserving the limiting radial directions on both sides.
        To see this, we first reduce to the case where the sequence contains no points in the interiors of the inserted disks. Suppose $x_n$ belongs to the interior of an inserted planar elliptic disk $E_n \subset \overline{B_2(0)} \setminus \{0\}$. If $R(E_n)$ is a spherical disk of radius $\delta_n < \pi/2$, then the chord estimate yields $\operatorname{dist}_{\mathbb{E}}(x_n, \partial E_n) \leq 4\sin\delta_n$. Since distinct inserted disks have pairwise disjoint radial images, any sequence approaching an unfilled component encounters disks with spherical radii $\delta_n \to 0$. Moreover, since $G^\lambda$ restricts to a Euclidean isometry on each inserted disk, choosing a closest boundary point $y_n \in \partial E_n$ gives 
        $$
        \Vert G^\lambda(x_n) - G^\lambda(y_n)\Vert = \Vert x_n - y_n\Vert \to 0\quad\text{as}\quad n\to\infty.
        $$
        Consequently, replacing $x_n$ with $y_n$ preserves both the Euclidean limits and the limiting radial directions on both the domain and target surfaces. It therefore suffices to consider points in $\overline{S_1}$, which contains all added points and boundary curves of inserted disks. Since the inverse radial map is continuous at these points and $G^\lambda$ is continuous, every pair in this boundary set is directly attainable as a limit of pairs from the original surface $S_1$. The same reasoning applies at limit stages via increasing unions. Thus, the boundary transformation formulas and arguments of \cref{sec: single-disk} apply directly to all unfilled components at every stage, requiring no additional analysis of the disk interiors.
    \end{remark}
    
	We proceed by using transfinite recursion. Let ${S}_i^{0}(\Gamma_1,\Gamma_2)={S}_i(\Gamma_1,\Gamma_2)$ and ${G}^0_{\Gamma_1,\Gamma_2}={G}_{\Gamma_1,\Gamma_2}$. Clearly, ${S}_i^{0}(\Gamma_1,\Gamma_2)$ and ${G}^0_{\Gamma_1,\Gamma_2}$ satisfy the properties {\refcond{P1}{$\mathrm{P}_1$}}--\refcond{P3}{$\mathrm{P}_3$}. Suppose that for all $\xi<\lambda$ we have constructed two locally convex surfaces ${S}_i^{\xi}(\Gamma_1,\Gamma_2)$ with respect to $0$ and an isometry ${G}^{\xi}_{\Gamma_1,\Gamma_2}$ between them satisfying properties {\refcond{P1}{$\mathrm{P}_1$}}--\refcond{P3}{$\mathrm{P}_3$}. 

\medskip
    \noindent\textbf{Limit stages.} Let $\lambda$ be a limit ordinal and $(\Gamma_1,\Gamma_2)$ be a $\lambda$-admissible pair. For simplicity, we shall present the proof only for the case where $\Gamma_1 = \Gamma_2 = \mathrm{id}_{\mathbb{H}^3_P}$ and suppress the explicit dependence on $\Gamma_1$ and $\Gamma_2$ in our notation; the general case can be established by a completely analogous argument.  Define ${S}_i^{\lambda}=\bigcup_{\xi<\lambda}{S}_i^{\xi}$ and let ${G}^{\lambda}(q)=\lim_{\xi\to\lambda}{G}^{\xi}(q)$ for all $q\in{S}_1^{\lambda}$. Clearly, ${S}_i^{\lambda}$ is a locally convex surface with respect to $0$ such that $0\notin {S}_i^{\lambda}$, ${S}_i^{\xi}\subset{S}_i^{\lambda}$ for all $\xi<\lambda$, and ${G}^{\lambda}|_{{S}_1^{\xi}}={G}^{\xi}$ for all $\xi<\lambda$. The radial projection $R:{S}_i^{\lambda}\to\mathbb{S}^2$ is a homeomorphism from ${S}_i^{\lambda}$ to $\mathscr{C}_i^{\lambda}$ since $\mathscr{C}_i^{\lambda}=\bigcup_{\xi<\lambda}\mathscr{C}_i^{\xi}$. Note that since each surface ${S}_i^{\lambda}$ is locally convex and therefore locally Lipschitz, the intrinsic path metric on ${S}_i^{\lambda}$ is well defined. We only need to check that ${G}^{\lambda}$ is an isometry from ${S}_1^{\lambda}$ to ${S}_2^{\lambda}$ with respect to their intrinsic path metrics. It is clear that ${G}^{\lambda}$ is a bijection. Let $p_1,p_2$ be two arbitrary points in ${S}_1^{\lambda}$, and let $\sigma$ be a rectifiable path in ${S}_1^{\lambda}$ connecting them. Since $\sigma$ is a compact subset of ${S}_1^{\lambda}$ and $\{{S}_1^{\xi}\}_{\xi<\lambda}$ is an exhaustion of ${S}_1^{\lambda}$ by open sets, there exists $\xi<\lambda$ such that $\sigma\subset{S}_1^{\xi}$. Then $\sigma'={G}^{\lambda}(\sigma)={G}^{\xi}(\sigma)$ is a rectifiable path in ${S}_2^{\xi}$ connecting ${G}^{\lambda}(p_1)$ and ${G}^{\lambda}(p_2)$. Since $G^{\xi}$ is an isometry, we have $l(\sigma) = l(\sigma')$, where $l(\cdot)$ denotes the length of a rectifiable curve. Therefore,
	$$
	\begin{aligned}
		d_{{S}_1^{\lambda}}(p_1,p_2)&=\inf\{l(\sigma):\text{$\sigma $ connecting $p_1$ and $p_2$ in ${S}_1^{\lambda}$}\}\\
		&=\inf\{l\big({G}^{\lambda}(\sigma)\big):\text{$\sigma $ connecting $p_1$ and $p_2$ in ${S}_1^{\lambda}$}\}\\
		&\geq \inf\{l(\sigma'):\text{$\sigma'$ connecting ${G}^{\lambda}(p_1)$ and ${G}^{\lambda}(p_2)$ in ${S}_2^{\lambda}$}\}\\
		&=d_{{S}_2^{\lambda}}\left({G}^{\lambda}(p_1),{G}^{\lambda}(p_2)\right).
	\end{aligned}
	$$
	A similar argument on the inverse map $({G}^{\lambda})^{-1}$ shows that the converse of the above inequality is also true and therefore ${G}^{\lambda}$ is an isometry.

    \medskip
    \noindent\textbf{Successor stages.} Now assume that $\lambda$ is a successor ordinal, i.e., $\lambda=\xi+1$ for some $\xi\in\text{ORD}$. The point extension and boundary analysis below are carried out for every $\xi$-admissible pair, whereas the actual disk filling requires $(\xi+1)$-admissibility. Let $\mathscr{P}_i^{\xi}$ be the set of isolated point components in $\mathscr{D}_i^{\xi}$. Since ${S}_i^{\xi}(\Gamma_1,\Gamma_2)$ is locally convex with respect to $0$, by an argument similar to that in \cref{sec: pogorelov}, ${S}_i^{\xi}(\Gamma_1,\Gamma_2)$ can be extended to a locally convex surface $\tilde{S}_i^{\xi}(\Gamma_1,\Gamma_2)$ with respect to $0$ such that the radial projection ${R}:\tilde{S}_i^{\xi}(\Gamma_1,\Gamma_2)\to\mathbb{S}^2$ is a homeomorphism from $\tilde{S}_i^{\xi}(\Gamma_1,\Gamma_2)$ onto $\Gamma_i(\mathscr{C}_i^{\xi}\cup \mathscr{P}_i^{\xi})$. Now, we proceed to show that ${G}^{\xi}_{\Gamma_1,\Gamma_2}$ can be extended to an isometry between $\tilde{S}_1^{\xi}(\Gamma_1,\Gamma_2)$ and $\tilde{S}_2^{\xi}(\Gamma_1,\Gamma_2)$.

    \begin{lemma}\label{lem: bottom-is-convex}
		Let $S$ be a radial surface in $\mathbb{R}^3_+$ which is bounded and locally convex with respect to $0$. Let $R:S\to\mathbb{S}^2$ denote the radial projection and $Q: R(S) \to S$ be its inverse. Suppose there exists $\delta>0$ such that the set $\Lambda_{\delta}=\{x_3\leq\delta\}\cap\mathbb{S}^2_{+}$ is contained in $R(S)$. Then the following statement holds:
        \begin{enumerate}
            \item[{\rm(1)}] There exists a compact convex set $P\subset\mathbb{R}^2$ such that $S'=S\cup P$ is a locally convex surface in $\mathbb{R}^3$. 
        \end{enumerate}
        If we further assume that $P$ has nonempty interior in $\mathbb{R}^2$, then:
        \begin{enumerate}
            \item[{\rm(2)}] $d_{\hat{S}}(x,y)=d_{S}(x,y)$ for all $x,y\in S$, where $\hat{S}=S\cup\partial P$, $d_{\hat{S}}$ and $d_{S}$ are intrinsic path metrics on $\hat{S}$ and $S$, respectively.
            \item[{\rm(3)}] There exists a neighborhood $U$ of $P$ in $S'$ and a constant $C_U$ such that $d_{S}(x,y)\leq C_U\|x-y\|$ for all $x,y\in S\cap U$.
        \end{enumerate}
	\end{lemma}
	\begin{proof}
		(1). Clearly, $Q: R(S) \to S$ is a homeomorphism. Let $C_\delta = \{x_3 = \delta\} \cap R(S)$, $C'_\delta = Q(C_\delta)$, and $\Lambda'_\delta = Q(\Lambda_\delta)$. Since $C'_\delta$ is a compact subset of the open half-space $\mathbb{R}^3_+$, there exists a sufficiently small $\varepsilon > 0$ such that $C'_\delta \subset \{x_3 > \varepsilon\}$. Let $\Sigma' = \{x_3 < \varepsilon\} \cap \Lambda'_\delta$ and $\Sigma = R(\Sigma')$. It follows that $\Sigma$ is an open subset of $\mathbb{S}^2$. 
		We claim that there exists $\tau > 0$ such that the set $\Lambda_{\tau} = \{x_3 \le \tau\} \cap \mathbb{S}^2_+$ is contained in $\Sigma$. Suppose for contradiction that no such $\tau$ exists. Then, we can find a sequence $\{w_n\} \subset \Lambda_\delta \setminus \Sigma$ that converges to some point $w$ in the equator $\mathbb{S}^1$. Since $S$ is bounded, there exists $N>0$ such that for all $n>N$, $Q(w_n)$ lies below the plane $\{x_3=\varepsilon\}$ and therefore is contained in $\Sigma'$. Thus, $w_n=R(Q(w_n))\in\Sigma$ for all $n>N$, which is a contradiction.
		
		For each $w \in \mathbb{S}^2_+$, let $H(w)$ denote the unique intersection point of the ray $\{tw : t \ge 0\}$ with the horizontal plane $\{x_3 = \varepsilon\}$. We define a radial map $\hat{Q}: \mathbb{S}^2_+ \to \mathbb{R}^3_+$ by
		\begin{equation}
			\hat{Q}(w)=\left\{
			\begin{aligned}
				&Q(w), w\in\Sigma\\
				&H(w), w\in\mathbb{S}^2_+\setminus\Sigma
			\end{aligned}\right.
		\end{equation}
		We now show that $\hat{Q}$ is continuous, which in turn guarantees that the set 
        $$
        \Omega=\{tw:w\in\mathbb{S}^2_+, 0<t<\|\hat{Q}(w)\|\}
        $$ 
        is open in $\mathbb{R}^3$. It suffices to prove that $Q(w)=H(w)$ for all $w\in\overline{\Sigma}\cap(\mathbb{S}^2_+\setminus\Sigma)$. By definition, this is equivalent to showing that $Q(w)\in\{x_3=\varepsilon\}$ for all $w\in\overline{\Sigma}\cap(\mathbb{S}^2_+\setminus\Sigma)$. Since 
        $$
        Q(\Sigma)=\Sigma'=\{x_3<\varepsilon\}\cap\Lambda'_{\delta},
        $$
        it follows that for each $w\in\overline{\Sigma}\cap\mathbb{S}^2_+$, $Q(w)\in\{x_3\leq\varepsilon\}\cap\Lambda_{\delta}'$. If $Q(w)\notin\{x_3=\varepsilon\}$, then $Q(w)\in\{x_3<\varepsilon\}\cap\Lambda'_\delta=\Sigma'$, which implies $w\in \Sigma$. This contradicts the fact that $w\in \mathbb{S}^2_+\setminus\Sigma$.

		Next, we show that $\Omega$ is a convex domain in $\mathbb{R}^3$. Clearly, we have 
        $$
        \partial\Omega \subset \{x_3 = 0\} \cup \{x_3 = \varepsilon\} \cup {\Sigma'}.
        $$
        Let $p\in\partial\Omega$. If $p\in\{x_3 = 0\} \cup \{x_3 = \varepsilon\}$, then the fact $\Omega\subset\{0<x_3<\varepsilon\}$ implies that either $\{x_3 = 0\}$ or $\{x_3 = \varepsilon\}$ serves as a local supporting plane of $\Omega$ at $p$. Now suppose $p\in\Sigma'$.
        Since $\Sigma'$ is an open subset of $S$, $\Sigma'$ is locally convex with respect to $0$. This guarantees the existence of a local supporting plane of $\Omega$ at $p$. Since each boundary point of $\Omega$ admits a local supporting plane, by \cref{thm: The-Erhard-Schmidt-Theorem}, $\Omega$ is a convex domain in $\mathbb{R}^3$. Let $P=\overline{\Omega}\cap\mathbb{R}^2$. Then $P$ is convex and $S\cup P$ is a locally convex surface in $\mathbb{R}^3$.

        (2). By the preceding argument, there exists a neighborhood $U$ of $\partial P$ in $\hat{S}$ such that $U$ is contained in the convex surface $\partial\Omega$, the result follows by an argument similar to that in \cref{lem: path-metric-boundary-removable}.
		
		(3). For each $a>0$, let 
        $$
        U_{a}=\{x\in\partial\Omega:d_{\partial\Omega}(x,P)<a\}\cap \{x_3<\varepsilon\}.
        $$
        It is clear that $U_a$ is a neighborhood of $P$ in $S'=S\cup P$ and there exists a constant $M_a>0$ such that 
        $$
        d_S(x,y)\leq d_{U_a\setminus P}(x,y)\leq M_a,\quad\forall x,y\in U_a\setminus P.
        $$
        Fix $\delta>0$. If $x,y\in U_a\setminus P$ satisfy $\|x-y\|\geq\delta$, we have 
        $$
        d_S(x,y)\leq M_a\leq\frac{M_a}{\delta}\|x-y\|.
        $$
        Then we only need to consider the case $\|x-y\|<\delta$.
        
        By \cref{prop: convex-surface-Lipschitz}, there exists a constant $C>0$ such that 
        \begin{equation}\label{eq: d-partial-Omega-setminus-P}
            d_{\partial\Omega\setminus P}(x,y)\leq C\|x-y\|\quad\forall x,y\in U_a\setminus P.
        \end{equation}
        It suffices to show that there exists $a>0$ such that $d_S(x,y)\leq d_{\partial\Omega\setminus P}(x,y)$ for all $x,y\in U_a\setminus P$ satisfying $\|x-y\|<\delta$. Let $a=\varepsilon/k$, where $k>0$ is a constant to be determined. Suppose that there exist $x,y\in U_a\setminus P$ such that $d_S(x,y)>d_{\partial\Omega\setminus P}(x,y)$. Then there exists a rectifiable curve $\sigma$ in $\partial\Omega\setminus P$ connecting $x$ and $y$ such that 
        $$
        l(\sigma)<\min\{d_S(x,y),d_{\partial\Omega\setminus P}(x,y)+\varepsilon/k\}.
        $$
        It follows that $\sigma\not\subset S$. Thus $\sigma\cap\{x_3=\varepsilon\}\ne\varnothing$. It follows that $l(\sigma)\geq2(1-1/k)\varepsilon$. On the other hand, by \eqref{eq: d-partial-Omega-setminus-P} we have
		$$
		   l(\sigma)< d_{\partial\Omega\setminus P}(x,y)+\varepsilon/k\leq C\|x-y\|+\varepsilon/k< C\delta +\varepsilon/k.
		$$
        Choose $k=3$ and $\delta=\varepsilon/(kC)$. It follows that
        $$
        \frac{4}{3}\varepsilon=2\left(1-\frac{1}{k}\right)\varepsilon\leq l(\sigma)\leq C\delta +\frac{\varepsilon}{k}=\frac{2}{3}\varepsilon,
        $$
        which is a contradiction. Therefore, we have $d_{S}(x,y)\leq \frac{kCM_a}{\varepsilon}\|x-y\|$ for all $x,y\in U_a\setminus P$, where $k=3$ and $a=\varepsilon/k$. This completes the proof.
	\end{proof}
    
    \begin{proposition}\label{prop: fill-isolated-points}
        The isometry ${G}^{\xi}_{\Gamma_1,\Gamma_2}$ can be extended to an isometry $$
        \tilde{G}^{\xi}_{\Gamma_1,\Gamma_2}: \tilde{S}_1^{\xi}(\Gamma_1,\Gamma_2)\to\tilde{S}_2^{\xi}(\Gamma_1,\Gamma_2).
        $$
    \end{proposition}
    \begin{proof}
        Fix a $\xi$-admissible pair $(\Gamma_1,\Gamma_2)$ and suppress it from the notation. By \cref{thm: zero-H1-dont-change-metric}, ${S}_i^{\xi}$ is a dense metric subspace of $\tilde{S}_i^{\xi}$ with respect to their intrinsic path metrics. It suffices to show that for each sequence $\{x_n\}_n$ in ${S}_1^{\xi}$ such that $\{x_n\}_n$ converges to some point $x_0\in \tilde{S}_1^{\xi}$ and $\{G^{\xi}(x_n)\}_n$ converges to some point $y_0\in \overline{{S}_2^{\xi}}$, we have $y_0\in \tilde{S}_2^{\xi}$. Indeed, by \cref{lem: Isometry-Extension}, this implies that ${G}^{\xi}$ can be extended to an isometric embedding from $\tilde{S}_1^{\xi}$ to $\tilde{S}_2^{\xi}$. After a similar argument, the inverse map of ${G}^{\xi}$ can also be extended to an isometric embedding from $\tilde{S}_2^{\xi}$ to $\tilde{S}_1^{\xi}$, which establishes the claim. 
        
        Suppose, for contradiction, that there exists a sequence $\{x_n\}_n$ in ${S}_1^{\xi}$ such that $\{x_n\}_n$ converges to some point $x_0\in \tilde{S}_1^{\xi}$ but $\{{G}^{\xi}(x_n)\}_n$ converges to some point $y_0\notin \tilde{S}_2^{\xi}$ in $\mathbb{R}^3$. Then we have the radial projection $R(x_0)\in\mathscr{P}_1^{\xi}$. Passing to a subsequence if necessary, we may assume that $\{R(G^\xi(x_n))\}_n$ converges to a point $w_0\in\mathbb{S}^2$. By \cref{lem:component-compactification}, $w_0$ lies in some isolated component $D$ of $\mathscr{D}_2^{\xi}$. Since $y_0\notin \tilde{S}_2^{\xi}$, $D$ is an isolated disk component of $\mathscr{D}_2^{\xi}$. After composing an isometry of $\mathbb{H}^3_P$ if necessary, we may assume that $D=\overline{\mathbb{S}^2_{-}}$. By \cref{lem: bottom-is-convex}, there exists a compact convex set $P\subset\mathbb{R}^2$ such that $\tilde{S}_2^{\xi}\cup P$ is a locally convex surface in $\mathbb{R}^3$. Since $\tilde{S}_2^{\xi}(\Gamma_1,\Gamma_2)$ is locally convex with respect to $0$ for all $\xi$-admissible pairs $(\Gamma_1,\Gamma_2)$, by an argument similar to that in the proof of \cref{prop: Pi-contain-0-interior}, we see that $P$ contains $0$ as an interior point in $\mathbb{R}^2$. It follows that $\{y_0\}=\{tw_0:t\geq0\}\cap \partial P$.
        Let $\hat{S}=\tilde{S}_2^{\xi}\cup\partial P$. Then $\hat{S}$ is a surface with $\partial P$ as its boundary.       
        Since $R(x_0)\in\mathscr{P}_1^{\xi}$ is an isolated point component of $\mathscr{D}_1^{\xi}$, we can choose an open disk neighborhood $U$ of $x_0$ in $\tilde{S}_1^{\xi}$ such that $U\setminus\{x_0\}\subset{S}_1^{\xi}$. 
        Consider the map
        $$
        \left.{G}^{\xi}\right|_{U\setminus \{x_0\}}:U\setminus \{x_0\}\to \mathbb{R}^3.
        $$ 
        It is Lipschitz continuous and therefore extends uniquely to a continuous map $G_0: U\to \mathbb{R}^3$. Clearly, $G_0(x_0) = y_0$, and $G_0$ maps $U$ into $\hat{S}$. Since ${G}^{\xi}$ is injective and $y_0 \notin S_2^{\xi}$, it follows that $G_0$ is a continuous injection. By shrinking $U$ if necessary, we may assume that $G_0(U)$ lies in a half-plane chart of $\hat{S}$ at $y_0$. Composing $G_0$ with this chart gives a continuous injection from an open disk into the closed half-plane $\{(x,y):y\geq0\}$, whose image contains a point in the boundary-line $\{(x,0):x\in\mathbb{R}\}$. This contradicts Brouwer's invariance of domain theorem and thus completes the proof.
    \end{proof}

    In view of \cref{prop: fill-isolated-points}, it suffices to consider the isolated disk components. For the remainder of this successor step, we replace $S_i^\xi, G^\xi$, and $\mathscr{C}_i^\xi$ with their point-filled extensions, remove the corresponding point indices from $J_i^\xi$, and retain the same notation. Consequently, $J_i^\xi \setminus J_i^{\xi+1}$ indexes only isolated disks, while the subsequent sets $J_i^{\xi+1}$ and $\mathscr{C}_i^{\xi+1}$ remain unchanged. We also replace $\mathscr D_i^\xi$ by $\mathscr D_i^\xi\setminus\mathscr P_i^\xi$, so that $\mathscr D_i^\xi=\mathbb S^2\setminus\mathscr C_i^\xi$ continues to hold with the new notation. Our next goal is to show that the boundary of each isolated disk component $D_{i,j}$ corresponds to the boundary of a planar elliptical disk in $\mathbb R^3\setminus\{0\}$. To this end, we first establish two auxiliary lemmas.
	\begin{lemma}\label{lem: continuous-extend-of-G}
	    Let $G: \Sigma_1\to\Sigma_2$ be an intrinsic isometry between two locally convex surfaces in $\mathbb{R}^3$. Suppose that there exists a compact convex body $P$ in $\mathbb{R}^2$ such that 
        $\Sigma_1'=\Sigma_1\sqcup P$ is a locally convex surface in $\mathbb{R}^3$ and there exists a neighborhood $U$ of $P$ in $\Sigma_1'$ and a constant $C_U$ such that 
        \begin{equation}\label{eq: d-Sigma-leq-CU}
            d_{\Sigma_1}(x,y)\leq C_U\|x-y\|,\quad \forall x,y\in \Sigma_1\cap U.
        \end{equation}
        Denote by $\overline{\Sigma_2}$ the closure of $\Sigma_2$ in $\mathbb{R}^3$. Then, $G$ can be extended to a continuous map $\hat{G}$ from $\Sigma_1\cup\partial P$ to $\overline{\Sigma_2}$ such that $\hat{G}(\partial P)\subset \overline{\Sigma_2}\setminus\Sigma_2$.
	\end{lemma}
    \begin{proof}
		Suppose that $\{u_n\}_n$ is a sequence in $\Sigma_1$ that converges to a point $q_1 \in \partial P$ in $\mathbb{R}^3$. By \eqref{eq: d-Sigma-leq-CU}, $\{u_n\}_n$ is a Cauchy sequence in the path metric space $\Sigma_1$. It follows that $v_{n}=G(u_{n})$ is also a Cauchy sequence in $\Sigma_2$. Therefore, $v_n$ converges to some point $q_2\in \overline{\Sigma_2}$ in $\mathbb{R}^3$. If $q_2\in\Sigma_2$, then
		$$
		q_1=\lim_{n\to\infty} u_n=\lim_{n\to\infty} G^{-1}(v_n)=G^{-1}(q_2)\in \Sigma_1,
		$$
		which is a contradiction to the fact that $\Sigma_1\cap P=\varnothing$.
		Now we show that the limit $q_2$ is unique. Suppose, for contradiction, that there are two sequences $\{u_n^1\}_n$ and $\{u_n^2\}_n$ in $\Sigma_1$ such that both $u_n^i$ converge to a point $q_1\in \partial P$, but $v_n^i=G(u_n^i)$ converges in $\mathbb{R}^3$ to different points $q_2^i\in \overline{\Sigma_2}\setminus \Sigma_2$. Let $\{w_n\}_n$ be the alternating sequence of $\{u_n^1\}_n$ and $\{u_n^2\}_n$, i.e., $w_{2n}=u_n^1$ and $w_{2n+1}=u_n^2$ for all $n\in\mathbb{N}$. By \eqref{eq: d-Sigma-leq-CU}, $w_n$ is a Cauchy sequence in $\Sigma_1$ and thus $G(w_n)$ is a Cauchy sequence in $\mathbb{R}^3$, a contradiction. For each $q_1\in \partial P$, we denote by $H(q_1)$ the corresponding limit $q_2\in \overline{\Sigma_2}\setminus\Sigma_2$.
		
		Define the function $\hat{G}:\Sigma_1\cup \partial P\to\overline{\Sigma_2}$ by 
		$$
		\hat{G}(q)=\left\{
		\begin{array}{ll}
			G(q),\quad &q\in \Sigma_1,\\
			H(q),\quad &q\in \partial P.
		\end{array}\right.
		$$
		We have already proved that for any sequence $\{u_n\}_n\subset\Sigma_1$ converging to a point $q_1 \in \partial P$, $\hat{G}(u_n)\to\hat{G}(q_1)$ as $n\to\infty$. To show $\hat{G}$ is continuous, we only need to prove that for all $\{u_n\}_n\subset\partial P$ converging to a point $q_1 \in \partial P$, $\hat{G}(u_n)\to \hat{G}(q_1)$ as $n\to\infty$. For each $u_n$, there exists a sequence $\{u_k^n\}_k \subset \Sigma_1$ converging to $u_n$ as $k \to \infty$. It follows that the sequence $\{ G(u_k^n)\}_k$ converges to the point $\hat{G}(u_n)\in \overline{\Sigma_2}\setminus \Sigma_2$. For each $n \in \mathbb{N}$, we choose an element $w_n \in \{u_k^n\}_{k \in \mathbb{N}}$ satisfying
		$$
		\max\{\|w_n - u_n\|, \|G(w_n) - \hat{G}(u_n)\|\} \leq \frac{1}{n}.
		$$
		Given $u_n \to q_1$, the construction ensures that $\{w_n\}_n$ is a sequence in $\Sigma_1$ converging to $q_1$, which in turn implies that $G(w_n) \to \hat{G}(q_1)$ as $n\to\infty$. Consequently, $\hat{G}(u_n) \to \hat{G}(q_1)$ since $\|G(w_n) - \hat{G}(u_n)\| \to 0$ as $n\to\infty$. This completes the proof.
	\end{proof}

    \begin{lemma}\label{lem: boundary-on-plane}
        Let $D$ be a closed round disk in $\mathbb{S}^2_+$, and let $U$ be a neighborhood of $D$ in $\mathbb{S}_+^2$. Let $\psi$ be a positive continuous function on $U\setminus\operatorname{int}D$. Suppose that
        $$
            S=\{\psi(w)w:w\in U\setminus D\}
        $$
        is locally convex with respect to $0$, and that
        $$
            \Lambda=\{\psi(w)w:w\in\partial D\}
        $$
        is the boundary of an open planar elliptical disk $\Omega$ contained in an affine plane $H$ with $0\notin H$. Then $S\cup\overline{\Omega}$ is locally convex with respect to $0$.
    \end{lemma}
    
    \begin{proof}
        Without loss of generality, we may assume that $D$ is centered at $e_3$. Identify the upper hemisphere with the tangent plane
        $$
            P=T_{e_3}\mathbb{S}^2=\{x_3=1\}
        $$
        by the radial map $w\mapsto w/w_3$. As in \cref{lem:Pi-not-segment-or-point}, we let $Z=\{kx:k\geq0,x\in S\}$ and let $q:Z\to\mathbb{R}$ be defined as
        \begin{equation*}
        		q(\alpha) = \inf\{k > 0 : \alpha/k \in S\}.
        \end{equation*}
        Let $F:=q|_{P\cap Z}$. Again, we see that $F$ is locally convex by \cref{lem: equivalence-locally-convex-to-0} and we have
        \begin{equation}\label{eq:expression-F-2}
            F\left(\frac{w}{w_3}\right)=\frac{1}{\psi(w)w_3}, \quad\forall w\in U \setminus D.
        \end{equation}
        Let $B, V \subset P$ be the subsets corresponding to $D$ and $U$ under the radial map, respectively. Note that $B$ is a Euclidean disk. Then $F$ can be continuously extended to $\partial B$ by the expression \eqref{eq:expression-F-2}.
        Let $L$ be the linear functional on $\mathbb{R}^3$ such that $H=\{x\in\mathbb{R}^3:Lx=1\}$.
        For each $y\in\partial B$, since the point $y/F(y)$ belongs to $\Lambda\subset H$, we see that $L(y)=F(y)>0$. Put $f=F-L|_{V\setminus \operatorname{int}B}$. Then $f$ is locally convex on $V\setminus B$ and vanishes on $\partial B$.
    
        We claim that there exists a neighborhood $V_0$ of $B$ such that $f\geq0$ on $V_0\setminus B$. Suppose, for contradiction, that there exists a sequence $\{q_n\}_{n=1}^{\infty}$ in $V\setminus B$ such that $f(q_n)<0$ for all $n\geq1$ and $d(q_n,B)\to 0$ as $n\to\infty$. After an affine change of coordinates in $P$, we may take $B$ to be the unit disk and choose a point $p_1=(r,0)$ with $f(p_1)<0$, where $r>1$ and $r-1$ is small enough that the following configuration lies in $V$. Denote by $p_2$ and $p_3$ the intersections of the tangent line at $p=(1,0)$ with the two tangent segments from $p_1$ to $B$. Since $f|_{\partial B}=0$ and $f(p_1)<0$, convexity of $f$ on the tangent segments implies that
        $$
            f(p_2)<0\quad\text{and}\quad f(p_3)<0.
        $$
        Translate the tangent line of $B$ at $p$ a small distance $\varepsilon>0$ outward, and let $p_2^\varepsilon,p_3^\varepsilon$ be its intersections with the two tangent segments emanating from $p_1$. It is clear that their midpoint $p^\varepsilon$ tends to $p$, while the segment $[p_2^\varepsilon,p_3^\varepsilon]$ lies outside $B$. By convexity of $f$, we have
        $$
            f(p^\varepsilon)
            \leq\frac{f(p_2^\varepsilon)+f(p_3^\varepsilon)}2.
        $$
        Letting $\varepsilon\to 0^{+}$ gives
        $$
            0=f(p)\leq\frac{f(p_2)+f(p_3)}2<0,
        $$
        which is a contradiction; see \cref{fig:lem4-6}.
    
        \begin{figure}[htbp]
            \centering
            \begin{tikzpicture}[
                x=2.15cm,y=2.15cm,
                point/.style={circle,fill=black,inner sep=1.05pt},
                every node/.style={font=\small}]
                \def\r{1.78}
                \def\eps{0.22}
                \pgfmathsetmacro{\ax}{1/\r}
                \pgfmathsetmacro{\ay}{sqrt(1-1/(\r*\r))}
                \pgfmathsetmacro{\py}{sqrt((\r-1)/(\r+1))}
                \pgfmathsetmacro{\pey}{(\r-1-\eps)/sqrt(\r*\r-1)}
    
                \fill[black!7] (0,0) circle (1);
                \draw[thick] (0,0) circle (1);
                \node at (-0.38,0.10) {$B$};
    
                \coordinate (Pone) at (\r,0);
                \coordinate (Aplus) at (\ax,\ay);
                \coordinate (Aminus) at (\ax,-\ay);
                \coordinate (P) at (1,0);
                \coordinate (Ptwo) at (1,\py);
                \coordinate (Pthree) at (1,-\py);
                \coordinate (Peps) at ({1+\eps},0);
                \coordinate (Ptwoeps) at ({1+\eps},\pey);
                \coordinate (Pthreeeps) at ({1+\eps},-\pey);
    
                \draw[thick] (Aplus)--(Pone)--(Aminus);
                \draw[dashed] (1,-0.82)--(1,0.82);
                \draw[very thick] (Pthreeeps)--(Ptwoeps);
    
                \node[point] at (Pone) {};
                \node[anchor=west] at ($(Pone)+(0.05,0.04)$) {$p_1$};
                \node[point] at (P) {};
                \node[anchor=east] at ($(P)+(-0.04,0)$) {$p$};
                \node[point] at (Ptwo) {};
                \node[anchor=south west] at ($(Ptwo)+(0.03,0.02)$) {$p_2$};
                \node[point] at (Pthree) {};
                \node[anchor=north west] at ($(Pthree)+(0.03,-0.02)$) {$p_3$};
                \node[point] at (Peps) {};
                \node[anchor=west] at ($(Peps)+(0.04,0)$) {$p^\varepsilon$};
                \node[point] at (Ptwoeps) {};
                \node[anchor=south west] at ($(Ptwoeps)+(0.03,0.02)$) {$p_2^\varepsilon$};
                \node[point] at (Pthreeeps) {};
                \node[anchor=north west] at ($(Pthreeeps)+(0.03,-0.02)$) {$p_3^\varepsilon$};
                \node[anchor=south] at (1,0.85) {$T_p\partial B$};
            \end{tikzpicture}
            \caption{The three-point argument in \cref{lem: boundary-on-plane}. The two tangent segments transfer the negativity of $f$ at $p_1$ to $p_2,p_3$, whereas convexity of $f$ contradicts $f(p)=0$.}
            \label{fig:lem4-6}
        \end{figure}
        Now, choose a convex neighborhood $V_0\subset V$ of $B$ such that $f\ge0$ on $V_0\setminus B$, and extend $f$ to $B$ by zero. It is easy to show that the extended function $\tilde{f}$ is convex on $V_0$. Since $L=F>0$ on $\partial B$, we have $L>0$ on $B$. Hence, the function $\tilde F:=\tilde{f}+ L|_{V_0}$ is convex and positive on $V_0$. Moreover, $\tilde{F}$ is equal to $L$ on $B$ and to $F$ on $V_0\setminus B$. Let $\tilde{q}$ be the positively homogeneous extension of $\tilde{F}$, that is,
        $$
            \tilde q(ty)=t\tilde F(y),
            \qquad \forall y\in V_0,\quad t>0.
        $$
        Then, $\tilde{q}$ is convex on the cone $C_0(V_0)$ over $V_0$. The level set $\{\tilde q=1\}$ is precisely $S\cup\overline{\Omega}$ near $\Lambda$. The conclusion then follows from \cref{lem: equivalence-locally-convex-to-0}.
    \end{proof}

	\begin{proposition}\label{prop: fill-isolated-disks}
	    Let $i\in\{1,2\}$. Suppose that $j\in J_i^{\xi}\setminus J_i^{\xi+1}$ and $(\Gamma_1,\Gamma_2)$ is a $\xi+1$-admissible pair of isometries of $\mathbb{H}^3_P$. Then there exists a compact planar elliptic disk $E_{i,j}^{\xi}(\Gamma_1,\Gamma_2)$ in $\overline{B_2(0)}\setminus\{0\}$ such that the inverse map $Q_{i; \, \Gamma_1, \Gamma_2}^{\xi}$ of the radial projection $R: {S}_i^{\xi}(\Gamma_1,\Gamma_2)\to\mathbb{S}^2$ extends to a homeomorphism 
        $$
        {Q}_{i}: \Gamma_i(\mathscr{C}_i^{\xi}\cup \partial D_{i,j})\to {S}_i^{\xi}(\Gamma_1,\Gamma_2)\cup \partial E_{i,j}^{\xi}(\Gamma_1,\Gamma_2).
        $$
        Furthermore, for each $j_1\in J_1^{\xi}\setminus J_1^{\xi+1}$, there exists a unique $j_2\in J_2^{\xi}\setminus J_2^{\xi+1}$ such that ${G}^{\xi}_{\Gamma_1,\Gamma_2}$ extends to an isometry between the surfaces
        $$
        {S}_1^{\xi}(\Gamma_1,\Gamma_2)\cup E_{1,j_1}^{\xi}(\Gamma_1,\Gamma_2)\quad\text{and}\quad {S}_2^{\xi}(\Gamma_1,\Gamma_2)\cup E_{2,j_2}^{\xi}(\Gamma_1,\Gamma_2),
        $$
        both of which are locally convex with respect to $0$, and the radial projection is a homeomorphism from ${S}_i^{\xi}(\Gamma_1,\Gamma_2)\cup E_{i,j_i}^{\xi}(\Gamma_1,\Gamma_2)$ onto $\Gamma_i(\mathscr{C}_i^{\xi}\cup D_{i,j_i})$.
	\end{proposition}
	\begin{proof}
		We only consider the case $i=1$. The case $i=2$ follows from a completely parallel argument. Let $j_1\in J_1^{\xi}\setminus J_1^{\xi+1}$. Then for all isometries $\Gamma$ of $\mathbb{H}^3_P$, $\Gamma(D_{1,j_1})$ is an isolated disk in $\Gamma(\mathscr{D}_1^{\xi})$, i.e., there exists an open set $U$ in $\mathbb{S}^2$ such that $U\cap\Gamma(\mathscr{D}_1^{\xi})=\Gamma(D_{1,j_1})$. Choose $\Gamma_1'$ such that $\Gamma_1'(D_{1,j_1})=\overline{\mathbb{S}^2_{-}}$. It is clear that $(\Gamma_1',\Gamma_2)$ is $\xi$-admissible. Since $\Gamma_1'(D_{1,j_1})=\overline{\mathbb{S}^2_{-}}$ is isolated in $\Gamma_1'(\mathscr{D}_1^{\xi})$, there exists $\delta>0$ such that the set $\Lambda_{\delta}=\{x_3\leq\delta\}\cap\mathbb{S}^2_+$ is contained in the circle domain $\Gamma_1'(\mathscr{C}_1^{\xi})$. By \cref{lem: bottom-is-convex}, there exists a compact convex set $P_1(\Gamma_1',\Gamma_2)\subset\mathbb{R}^2$ such that ${S}_1^{\xi}(\Gamma_1',\Gamma_2)\cup P_1(\Gamma_1',\Gamma_2)$ is a locally convex surface in $\mathbb{R}^3$. Since ${S}_1^{\xi}(\Gamma_1,\Gamma_2)$ is locally convex with respect to $0$ for all $\xi$-admissible pairs $(\Gamma_1,\Gamma_2)$, by an argument similar to that in the proof of \cref{prop: Pi-contain-0-interior}, we see that $P_1(\Gamma_1',\Gamma_2)$ contains $0$ as an interior point in $\mathbb{R}^2$. It follows that $Q_{1;\, \Gamma_1',\Gamma_2}^{\xi}$ can be extended to a continuous injection from $\Gamma_1'(\mathscr{C}_1^{\xi}\cup \partial D_{1,j_1})$ to $\mathbb{R}^3\setminus\{0\}$. By \cref{rmk:boundary-limits-under-filling} and \cref{lem: Expression-of-L1-General-gamma}, $Q_{1;\, \Gamma_1,\Gamma_2}^{\xi}$ can also be extended to a continuous injection ${Q_1}$ from $\Gamma_1(\mathscr{C}_1^{\xi}\cup \partial D_{1,j_1})$ to $\mathbb{R}^3\setminus\{0\}$. Since $\Gamma_1(D_{1,j_1})$ is isolated in $\Gamma_1(\mathscr{D}_1^{\xi})$, it follows that ${Q_1}$ is a homeomorphism from $\Gamma_1(\mathscr{C}_1^{\xi}\cup \partial D_{1,j_1})$ to its image, which is the union of ${S}_1^{\xi}(\Gamma_1,\Gamma_2)$ and a simple closed curve in $\mathbb{R}^3$.
		
		Now we turn to the extension of the isometry. By \cref{lem: bottom-is-convex}, there exists a neighborhood $U$ of $P_1(\Gamma_1',\Gamma_2)$ in $S'={S}_1^{\xi}(\Gamma_1',\Gamma_2)\cup P_1(\Gamma_1',\Gamma_2)$ and a constant $C_U$ such that 
        $$
        d_{{S}_1^{\xi}(\Gamma_1',\Gamma_2)}(x,y)\leq C_U\|x-y\|,\qquad\forall x,y\in {S}_1^{\xi}(\Gamma_1',\Gamma_2)\cap U.
        $$
        Write $C_{1,j_1}^{\xi}(\Gamma_1',\Gamma_2)=\partial P_1(\Gamma_1',\Gamma_2)$. By \cref{lem: continuous-extend-of-G}, ${G}^{\xi}_{\Gamma_1',\Gamma_2}$ can be extended to a continuous map 
        $$
        {G}^{\xi}_{j_1;\,\Gamma_1',\Gamma_2}:{S}_1^{\xi}(\Gamma_1',\Gamma_2)\cup C_{1,j_1}^{\xi}(\Gamma_1',\Gamma_2)\to \overline{{S}_2^{\xi}(\Gamma_1',\Gamma_2)}
        $$ such that $C_{1,j_1}^{\xi}(\Gamma_1',\Gamma_2)$ is mapped into $\overline{{S}_2^{\xi}(\Gamma_1',\Gamma_2)}\setminus {S}_2^{\xi}(\Gamma_1',\Gamma_2)$. Choose a sequence $\{u_n\}_n\subset{S}_1^{\xi}(\Gamma_1',\Gamma_2)$ such that $u_n$ converges to a point $q_1 \in \partial P_1(\Gamma_1',\Gamma_2)$ and $v_n ={G}^{\xi}_{\Gamma_1',\Gamma_2}(u_n)$ converges to a point $q_2 \in \overline{{S}_2^{\xi}(\Gamma_1',\Gamma_2)}\setminus {S}_2^{\xi}(\Gamma_1',\Gamma_2)$. Passing to a subsequence if necessary, we may assume that $\{R(v_n)\}_n$ converges to a point $w_0\in\mathbb{S}^2$. By \cref{lem:component-compactification}, there exists a unique $j_2\in J_2^{\xi}\setminus J_2^{\xi+1}$ such that $w_0\in\partial\Gamma_2(D_{2,j_2})$. By a completely parallel argument, $Q_{2;\, \Gamma_1',\Gamma_2}^{\xi}$ can also be extended to a continuous injection ${Q_2}$ from $\Gamma_2(\mathscr{C}_2^{\xi}\cup \partial D_{2,j_2})$ to $\mathbb{R}^3\setminus\{0\}$. Let $C_{2,j_2}^{\xi}(\Gamma_1',\Gamma_2)={Q_2}(\partial\Gamma_2(D_{2,j_2}))$. It follows that $C_{2,j_2}^{\xi}(\Gamma_1',\Gamma_2)$ is a simple closed curve in $\mathbb{R}^3\setminus\{0\}$ and is a connected component of $\overline{{S}_2^{\xi}(\Gamma_1',\Gamma_2)}\setminus {S}_2^{\xi}(\Gamma_1',\Gamma_2)$. Therefore, the extended map ${G}_{j_1;\, \Gamma_1',\Gamma_2}^{\xi}$ maps $C_{1,j_1}^{\xi}(\Gamma_1',\Gamma_2)=\partial P_1(\Gamma_1',\Gamma_2)$ to $C_{2,j_2}^{\xi}(\Gamma_1',\Gamma_2)$. 
		
		Let $\Gamma_2'$ be an isometry of $\mathbb{H}^3_P$ such that $\Gamma_2'(D_{2,j_2})=\overline{\mathbb{S}^2_{-}}$. Clearly, $(\Gamma_1',\Gamma_2')$ is a $\xi$-admissible pair of isometries of $\mathbb{H}^3_P$. Then $C_{1,j_1}^{\xi}(\Gamma_1',\Gamma_2')$ and $C_{2,j_2}^{\xi}(\Gamma_1',\Gamma_2')$ are two simple closed convex curves in $\mathbb R^2$ enclosing the origin. By applying the argument from \cref{sec: single-disk} and composing with a Möbius transformation of $\mathbb{D}$ if necessary, we may assume that each $C_{i,j_i}^{\xi}(\Gamma_1',\Gamma_2')$, $i\in\{1,2\}$ coincides with the unit circle $\mathbb{S}^1$ and ${G}^{\xi}_{\Gamma_1',\Gamma_2'}$ can be extended to an isometry from ${S}_1^{\xi}(\Gamma_1',\Gamma_2')\cup \mathbb{S}^1$ to ${S}_2^{\xi}(\Gamma_1',\Gamma_2')\cup \mathbb{S}^1$ such that the restriction of ${G}^{\xi}_{\Gamma_1',\Gamma_2'}$ on $\mathbb{S}^1$ is the identity map $ \mathrm{id}_{\mathbb{S}^1}$. Then, by \cref{lem: unit-circle-to-ellipse}, we see that $C_{i,j_i}^{\xi}(\Gamma_1,\Gamma_2)$, $i\in\{1,2\}$ are the boundaries of a pair of flat elliptic disks $E_{i,j_i}^{\xi}(\Gamma_1,\Gamma_2)$ that differ only by an isometry of $\mathbb{R}^3$. Since $(\Gamma_1,\Gamma_2)$ is $(\xi+1)$-admissible, $\Gamma_i(D_{i,j_i})$ is strictly contained in an open hemisphere and the elliptic disk $E_{i,j_i}^{\xi}(\Gamma_1,\Gamma_2)$ projects homeomorphically onto $\Gamma_i(D_{i,j_i})$. It follows that for each $i\in\{1,2\}$, the radial projection is a homeomorphism from ${S}_i^{\xi}(\Gamma_1,\Gamma_2)\cup E_{i,j_i}^{\xi}(\Gamma_1,\Gamma_2)$ to $\Gamma_i(\mathscr{C}_i^{\xi}\cup D_{i,j_i})$. Moreover, ${G}^{\xi}_{\Gamma_1,\Gamma_2}$ can be extended to an isometry between surfaces ${S}_i^{\xi}(\Gamma_1,\Gamma_2)\cup E_{i,j_i}^{\xi}(\Gamma_1,\Gamma_2)$, $i\in\{1,2\}$. By \cref{lem: boundary-on-plane}, we see that each surface ${S}_i^{\xi}(\Gamma_1,\Gamma_2)\cup E_{i,j_i}^{\xi}(\Gamma_1,\Gamma_2)$ is locally convex with respect to $0$. This completes the proof.
	\end{proof}
    
	Recall that in view of \cref{prop: fill-isolated-points}, we have assumed that $\tilde{S}_1^{\xi}(\Gamma_1,\Gamma_2) = {S}_1^{\xi}(\Gamma_1,\Gamma_2)$. Then, for each $j_1\in J_1^{\xi}\setminus J_1^{\xi+1}$, $D_{1,j_1}$ is an isolated disk component of $\mathscr{D}_1^{\xi}$. By the above proposition, there exists a unique $j_2\in J_2^{\xi}\setminus J_2^{\xi+1}$ 
    and two locally convex surfaces ${S}_i^{\xi}(\Gamma_1,\Gamma_2)\cup E_{i,j_i}^{\xi}(\Gamma_1,\Gamma_2)$ with respect to $0$ such that each ${S}_i^{\xi}(\Gamma_1,\Gamma_2)\cup E_{i,j_i}^{\xi}(\Gamma_1,\Gamma_2)$ is homeomorphic to $\Gamma_i(\mathscr{C}_i^{\xi}\cup D_{i,j_i})$ under the radial projection. Furthermore, ${G}^{\xi}_{\Gamma_1,\Gamma_2}$ extends to an isometry from ${S}_1^{\xi}(\Gamma_1,\Gamma_2)\cup E_{1,j_1}^{\xi}(\Gamma_1,\Gamma_2)$ to ${S}_2^{\xi}(\Gamma_1,\Gamma_2)\cup E_{2,j_2}^{\xi}(\Gamma_1,\Gamma_2)$. By symmetry, each $j_2\in J_2^{\xi}\setminus J_2^{\xi+1}$ also uniquely determines a $j_1\in J_1^{\xi}\setminus J_1^{\xi+1}$, ensuring that the map
    $$
    J_1^{\xi}\setminus J_1^{\xi+1}\to J_2^{\xi}\setminus J_2^{\xi+1}, \quad j_1\mapsto j_2
    $$
    is a bijection. Now, let 
	$$
	{S}_i^{\xi+1}(\Gamma_1,\Gamma_2)={S}_i^{\xi}(\Gamma_1,\Gamma_2)\cup \bigcup _{j\in J_i^{\xi}\setminus J_i^{\xi+1}}E_{i,j}^{\xi}(\Gamma_1,\Gamma_2),\quad i\in\{1,2\}.
	$$
	Since each $j\in J_i^{\xi}\setminus J_i^{\xi+1}$ corresponds to an isolated disk in $\mathscr{D}_i^{\xi}$, it follows that ${S}_i^{\xi+1}(\Gamma_1,\Gamma_2)$ is a locally convex surface with respect to $0$ and the radial projection $R:{S}_i^{\xi+1}(\Gamma_1,\Gamma_2)\to\mathbb{S}^2$ is a homeomorphism from ${S}_i^{\xi+1}(\Gamma_1,\Gamma_2)$ onto $\Gamma_i(\mathscr{C}_i^{\xi+1})$. Furthermore, it is clear that ${G}^{\xi}_{\Gamma_1,\Gamma_2}$ can be extended to an isometry from ${S}_1^{\xi+1}(\Gamma_1,\Gamma_2)$ to ${S}_2^{\xi+1}(\Gamma_1,\Gamma_2)$. Therefore, we have constructed two locally convex surfaces ${S}_i^{\xi+1}(\Gamma_1,\Gamma_2)$ with respect to $0$ and an isometry ${G}^{\xi+1}_{\Gamma_1,\Gamma_2}$ satisfying the properties {\refcond{P1}{$\mathrm{P}_1$}}--\refcond{P3}{$\mathrm{P}_3$}.

	\begin{proof}[Proof of \cref{thm: MainThm1}]
        In view of \cref{subsec:special-cases}, we only need to consider the case where both $X_1$ and $X_2$ are three-dimensional and neither is a hyperbolic half-space.
		Fix one $\lambda_0$-admissible pair and suppress it from the notation. For each $0\leq\lambda\leq\lambda_0$, we have constructed two locally convex surfaces ${S}_i^{\lambda}$, $i\in\{1,2\}$ and an isometry ${G}^{\lambda}:{S}_1^{\lambda}\to{S}_2^{\lambda}$ with respect to their intrinsic path metrics satisfying the properties {\refcond{P1}{$\mathrm{P}_1$}}--\refcond{P3}{$\mathrm{P}_3$}. By the choice of $\lambda_0$, we have $J_i^{\lambda_0}=\varnothing$ for $i=1,2$. It follows that $\mathscr{C}_i^{\lambda_0}=\mathbb{S}^2$ and ${S}_i^{\lambda_0}$ is a locally convex surface homeomorphic to $\mathbb{S}^2$ for each $i\in\{1,2\}$. By \cref{cor: locallyconvex-convex}, each ${S}_i^{\lambda_0}$ is the boundary of a compact convex body in $\mathbb{R}^3$. It follows from \cref{thm: Pogorelov} that ${G}^{\lambda_0}$ and thereby $G$ can be extended to be an isometry of $\mathbb{R}^3$. Finally, \cref{prop: RnIsometry-to-HnIsometry} implies that $F$ can be extended to an isometry of $\mathbb{H}^3_P$, which completes the proof.
	\end{proof}

     \begin{remark}\label{rmk: cantor-rank-infinite}
        We note that the use of transfinite recursion is necessary. 
        Indeed, there exists a circle-type closed set with countably many components such that any finite number of operations of removing isolated components fails to eliminate all components.
        First, for each $n\in\mathbb{N}$, we construct a circle-type closed set $A_n$ whose component space has Cantor--Bendixson rank~$n+1$, by induction. Let $A_1$ be the union of a disk $D^0$ and a sequence of isolated disks converging to the boundary of $D^0$ on $\mathbb{S}^2$. 
        Suppose $A_n$ has been defined. Let $D^n$ be a closed disk, and let $\{D^n_{m}\}_m$ be a sequence of isolated disks converging to the boundary of $D^n$ on $\mathbb{S}^2$. 
        For each $m$, we shrink $A_n$ sufficiently so that it fits inside $D^n_m$. 
        Then $A_{n+1}$ is defined as the union of $D^n$ and the isolated copies of $A_n$ placed in the disks $D^n_{m}$. By construction, the Cantor--Bendixson rank of $A_n$ is $n+1$ for every $n \in \mathbb{N}$.
        Now, let $D^{\infty}$ be a closed disk, and let $\{D^{\infty}_{m}\}_m$ be a sequence of isolated disks converging to the boundary of $D^{\infty}$ on $\mathbb{S}^2$. For each $m$, we shrink $A_m$ sufficiently so that it fits inside $D^{\infty}_m$. Let $K$ be the union of $D^{\infty}$ and the isolated copies of $A_m$ contained in the disks $D^{\infty}_{m}$. It follows that $K$ is a circle-type closed set and its Cantor--Bendixson rank is $\omega+1$, where $\omega$ is the first infinite ordinal.
    \end{remark}
    
    \begin{proof}[Proof of \cref{thm: MainThm2}]
        Let $F$ be an intrinsic isometry from $\partial X_1$ to $\partial X_2$. Recall the definition of $\tilde{S}_i$ in \cref{sec: pogorelov}. Let $\hat{S}_i$ be the closure of $\tilde{S}_i$ in $\mathbb{R}^3$. By an argument similar to that in \cref{prop: Pi-contain-0-interior}, it is clear that each $\hat{S}_i$ is a compact surface in $\mathbb{R}^3\setminus\{0\}$ such that the boundary $\partial \hat{S}_i$ is the disjoint union of finitely many simple closed curves $\{C_{i,j}\}_{j=1}^{n_i}$.
        
        We now show that $G$ extends to an isometry $\tilde{G}:\tilde{S}_1\to\tilde{S}_2$. It suffices to prove that for each sequence $\{x_n\}_n$ in ${S}_1$ such that $\{x_n\}_n$ converges to some point $x_0\in \tilde{S}_1$ and $\{G(x_n)\}_n$ converges to some point $y_0\in \hat{S}_2$, we have $y_0\in \tilde{S}_2$. Indeed, by \cref{lem: Isometry-Extension}, this implies that ${G}$ can be extended to an isometric embedding from $\tilde{S}_1$ to $\tilde{S}_2$. A similar argument shows that ${G}^{-1}$ can also be extended to an isometric embedding from $\tilde{S}_2$ to $\tilde{S}_1$, which establishes the claim. 
        Suppose, for contradiction, that there exists a sequence $\{x_n\}_n$ in ${S}_1$ such that $\{x_n\}_n$ converges to some point $x_0\in \tilde{S}_1$ but $\{G(x_n)\}_n$ converges to some point $y_0\in\partial \hat{S}_2$ in $\mathbb{R}^3$. Let $C_{2,j}$ be the component of $\partial\hat{S}_2$ containing $y_0$. After composing $F$ with an isometry of $\mathbb{H}^3_P$, we may assume that $C_{2,j}$ lies in the plane $\mathbb{R}^2$. Let $\hat{S}_{2,j}=\tilde{S}_2\cup C_{2,j}$. Since $\tilde{S}_2$ is a locally convex surface with respect to $0$, by \cref{lem: bottom-is-convex}, $C_{2,j}$ is a closed convex curve in $\mathbb{R}^2$, $\hat{S}_{2,j}$ is a path metric space and
        \begin{equation}\label{eq: Mainthm2}
             d_{\hat{S}_{2,j}}(x,y)=d_{\tilde{S}_2}(x,y),\quad\forall x,y\in \tilde{S}_2,
        \end{equation}
        where $d_{\hat{S}_{2,j}}$ and $d_{\tilde{S}_2}$ are intrinsic path metrics on $\hat{S}_{2,j}$ and $\tilde{S}_2$, respectively. For each $a>0$, let $D_a$ be a closed disk neighborhood of $x_0$ in $\tilde{S}_1$ with diameter $\operatorname{diam} (D_a)<a$. Since the map $G|_{D_a\cap S_1}:D_a\cap S_1\to \mathbb{R}^3$ is Lipschitz continuous and $D_a\cap S_1$ is dense in $D_a$, it can be extended to a continuous map $G_0: D_a\to \mathbb{R}^3$. Clearly, $G_0(x_0)=y_0$, $G_0(D_a)\subset \hat{S}_2$, and $G_0(D_a)\subset B_a(y_0)$, where $B_a(y_0)=\{x\in\mathbb{R}^3:\|x-y_0\|\leq a\}$. It follows that we can choose $a$ small enough such that $G_0(D_a)\subset\hat{S}_{2,j}$.
        We now claim that $G_0: D_a\to\hat{S}_{2,j}$ is an injection. Suppose, for contradiction, that there exist two distinct points $x_1,x_2\in D_a$ such that $G_0(x_1)=G_0(x_2)$. Let $\{x_n^1\}_n$ and $\{x_n^2\}_n$ be two sequences in $D_a\cap S_1$ converging to $x_1$ and $x_2$, respectively. Then by \eqref{eq: Mainthm2} and \cref{thm: zero-H1-dont-change-metric}, we have
        \begin{equation*}
            \begin{aligned}
            0=d_{\hat{S}_{2,j}}(G_0(x_1),G_0(x_2))&=\lim_{n\to\infty} d_{\hat{S}_{2,j}}(G(x^1_n),G(x^2_n))\\
            &=\lim_{n\to\infty} 
            d_{\tilde{S}_2}(G(x^1_n),G(x^2_n))\\
            &=\lim_{n\to\infty}d_{{S}_2}(G(x^1_n),G(x^2_n))\\
            &=\lim_{n\to\infty}d_{{S}_1}(x^1_n,x^2_n)=\lim_{n\to\infty}d_{\tilde{S}_1}(x^1_n,x^2_n)=d_{\tilde{S}_1}(x_1,x_2)>0,
            \end{aligned}
        \end{equation*}
        which is a contradiction. By shrinking $D_a$ if necessary, we may assume that $G_0(\operatorname{int}D_a)$ lies in a half-plane chart of $\hat{S}_2$ at $y_0$. Composing $G_0$ with this chart gives a continuous injection from the open disk $\operatorname{int}D_a$ into the closed half-plane $\{(x,y):y\geq0\}$, whose image contains a point in the boundary-line $\{(x,0):x\in\mathbb{R}\}$. This contradicts Brouwer's invariance of domain theorem. Thus, $G$ extends to an isometry $\tilde{G}:\tilde{S}_1\to\tilde{S}_2$.

        Now consider the extended isometry $\tilde{G}:\tilde{S}_1\to\tilde{S}_2$. By an argument similar to that in \cref{prop: fill-isolated-disks}, each $\tilde{S}_i$ can be extended to the boundary of a compact convex body $K_i$ in $\mathbb{R}^3$, and $\tilde{G}$ extends to an intrinsic isometry between $\partial K_1$ and $\partial K_2$. Finally, by \cref{thm: Pogorelov} and \cref{prop: RnIsometry-to-HnIsometry}, $F$ extends to an isometry of $\mathbb{H}^3_P$, which completes the proof.
    \end{proof}

    \section*{Statements and Declarations}

    \noindent\textbf{Competing Interests.}
    The author has no relevant financial or non-financial interests to disclose.
    
    \medskip
    
    \noindent\textbf{Funding.}
    The author received no specific funding for this work.
    
    \medskip
    \noindent\textbf{Acknowledgements.} I am grateful to Professor Feng Luo for his helpful suggestions and enlightening discussions.
    I further thank Professor Feng Luo, Yanwen Luo, and Zhenghao Rao for their careful reading of the manuscript. Their valuable comments have significantly improved the presentation of this paper; in particular, they pointed out a simpler proof for \cref{lem:Pi-not-segment-or-point}, which led to a substantial simplification of this work. Finally, I would like to express my gratitude to my advisor, Professor Bobo Hua, for his continued support and invaluable guidance.
    
	\bibliographystyle{amsplain}
    \bibliography{main}
\end{document}